\documentclass[11pt]{article}
\everymath{\displaystyle}
\usepackage{amsfonts}
\usepackage{amssymb}
\usepackage{amsmath}
\usepackage{amsthm}

\usepackage[top=2.5cm, bottom=2.25cm, left=1cm, right=1cm, headsep=1cm]{geometry}

\usepackage{hyperref}
\usepackage{setspace}

 \numberwithin{equation}{section}
 
\newtheorem{thm}{thm}[section]

\newtheorem{theorem}[thm]{Theorem}
\newtheorem{lemma}[thm]{Lemma}
\newtheorem{remark}[thm]{Remark}
\newtheorem{proposition}[thm]{Proposition}
\newtheorem{definition}[thm]{Definition}
\newtheorem{corollary}[thm]{Corollary}

\newtheorem{example}[thm]{Example}
\newtheorem{assumption}[thm]{Assumption}

\newcommand{\ip}[2]{\left\langle#1,#2\right\rangle} %%%

\newcommand{\KK}{\mathbb{K}}
\newcommand{\NN}{\mathbb{N}}
\newcommand{\EE}{\mathbb{E}}

\newcommand{\PP}{\mathbb{P}}
\newcommand{\DD}{\mathbb{D}}
\newcommand{\p}{\mathcal{P}}

\newcommand{\F}{\mathcal{F}}
\newcommand{\RR}{\mathbb{R}}

\newcommand{\B}{\mathcal{B}}

\newcommand{\K}{\mathcal{K}}

\newcommand{\C}{\mathcal{C}}

\newcommand{\cS}{\mathcal{S}}
\def\tX{\tilde X}
\def\t{\tilde}

\newcommand{\I}{\mathcal{I}}
\newcommand{\Xt}{X^{ \theta}}
\newcommand{\Xd}{X^{x, [\theta]}}

\newcommand{\D}{\mathcal{\mathbf{D}}}

\newcommand{\Vs}{V_0, \ldots, V_d}

\newcommand{\bv}{\boldsymbol v}
\newcommand{\bb}{\boldsymbol \beta}

\newtheorem*{remark*}{Remark}

\makeatletter
\newtheorem*{rep@theorem}{\rep@title}
\newcommand{\newreptheorem}[2]{%
\newenvironment{rep#1}[1]{%
 \def\rep@title{#2 \ref{##1}}%
 \begin{rep@theorem}}%
 {\end{rep@theorem}}}
\makeatother

\newreptheorem{theorem}{Theorem}
\newreptheorem{lemma}{Lemma}

\title{Smoothing properties of McKean-Vlasov SDEs\thanks{This work was partially supported by the Engineering and Physical Sciences Research Council [grant no EP/M506345/1].}}
\author{Dan Crisan\thanks{Department of Mathematics, Imperial College London, 180 Queen's Gate, London
SW7 2AZ, UK}, Eamon McMurray\thanks{Department of Mathematics, Imperial College London, 180 Queen's Gate, London
SW7 2AZ, UK. }}
\begin{document}

\maketitle

\abstract{In this article, we develop integration by parts formulae on Wiener space for solutions of SDEs with general McKean-Vlasov interaction and uniformly elliptic coefficients. These integration by parts formulae hold both for derivatives with respect to a real variable and derivatives with respect to a measure understood in the sense of Lions. They allows us to prove the existence of a classical solution to a related PDE with irregular terminal condition. We also develop bounds for the derivatives of the density of the solutions of McKean-Vlasov SDEs.\\[2mm]
{\bf Keywords}: Integration by Parts Formulae, Malliavin Calculus, McKean-Vlasov SDEs, Kusuoka-Stroock Functions. \\[2mm]
}

%%%%%%%%%%%%%%%%%%%%%%%%%%%%%%%%%%%%%%%%%%%%%%%%%%%%%%%%%%%%
\section{Introduction}
\label{sec:intro}
%%%%%%%%%%%%%%%%%%%%%%%%%%%%%%%%%%%%%%%%%%%%%%%%%%%%%%%%%%%%

The main object of study in this paper is the McKean-Vlasov stochastic differential equation (MVSDE)
\begin{equation}
\label{eq:trueMKV}
\Xt_t =  \theta + \int_0^t V_0 \left(\Xt_s, \left[\Xt_s \right] \right) \, ds +  \sum_{i=1}^d \int_0^t V_i \left(\Xt_s, \left[\Xt_s \right] \right) \, dB^i_s ,
\end{equation}
driven by a Brownian motion $B= \left(B^1, \ldots, B^d \right)$, with coefficients $V_0, \ldots, V_d: \RR^N \times \p_2(\RR^N) \to \RR^N$ and initial condition $\theta $, a square-integrable random variable independent of $B$. Here and throughout, we denote by $[\xi]$ the law of a random variable $ \xi$ and by $\p_2(\RR^N)$ the set of probability measures on $\RR^N$ with finite second moment.

MVSDEs are equations whose coefficients depend on the law of the solution. They are also referred to as mean-field SDEs and their solutions are often called nonlinear diffusions. 
These MVSDEs provide a probabilistic representation to the solutions of a class of  nonlinear PDEs. A particular example of such nonlinear PDEs was first studied by McKean \cite{mckean}. These equations describe the limiting behaviour of an individual particle evolving within a large system of particles undergoing diffusive motion and interacting in a `mean-field' sense, as the population size grows to infinity. A particular characteristic of the limiting behaviour of the system, is that any finite subset of particles become asymptotically independent of each other.      This \emph{propagation of chaos} phenomenon was studied by McKean \cite{mckean1967propagation} and Sznitman \cite{sznitman1991topics} among many other authors.  Existence and uniqueness results, the theory of propagation of chaos and numerical methods have been studied in a variety of settings (see, for example, \cite{bossy2005some,bossy97astochastic,Jourdain_nonlinearsdes,meleard1996asymptotic}).

 As MVSDEs can be interpreted as limiting equations for  large systems, they are widely  used as models in statistical physics \cite{bossy97astochastic,meleard1996asymptotic} as well as in the study of large-scale social interactions within the theory of mean-field games \cite{lasry2006jeux,lasry2006jeux2,lasry2007mean,HCM1,HCM2, carmonaprobmfg,carmonacontrolmkv}. Recently, these equations have also appeared in the mathematical finance literature in the specification and calibration of multi-factor stochastic volatility and hybrid models \cite{bergomi2008smile,guyon2011smile}.

In this paper,  we develop several new integration by parts formulae for solutions of MVSDE. In turn, these formulae enable us to   use MVSDE to define the solution of  a class of partial differential equations that has the form
\begin{align}
        \label{eq:MasterPDE}
        \begin{split}
                \left( \partial_t - \mathcal{L} \right) U(t,x,[\theta]) = 0 & \quad \quad \text{ for } (t,x,[\theta]) \in (0,T] \times \RR^N \times \p_2(\RR^N)\\
                U(0,x,[\theta]) = g(x,[\theta]) &  \quad \quad \text{ for } (x,[\theta]) \in  \RR^N \times \p_2(\RR^N),
        \end{split}
\end{align}
where $g: \RR^N \times \p_2(\RR^N) \to \RR$ and the operator $\mathcal{L}$ acts on sufficiently enough functions $F:\RR^N \times \p_2(\RR^N) \to \RR^N$  and is defined 
\begin{align*}
        \mathcal{L} F(x,[\theta]) = & \sum_{i=1}^N V_0^i(x,[\theta]) \, \partial_{x_i} F(x,[\theta])
        + \frac{1}{2} \sum_{i,j=1}^N [\sigma \sigma^{\top} (x,[\theta])]_{i,j} \, \partial_{x_i} \partial_{x_j} F(x,[\theta]) \\
        & + \EE \left[ \sum_{i=1}^N V_0^i(\theta,[\theta]) \, \partial_{\mu} F(x,[\theta], \theta)_i +  \frac{1}{2} \sum_{i,j=1}^N [\sigma \sigma^{\top} (\theta,[\theta])]_{i,j} \, \partial_{v_j} \partial_{\mu} F(x,[\theta],\theta)_i \right],
\end{align*}
where $\sigma(z, \mu)$ is the $N \times d$ matrix with columns $V_1(z,\mu), \ldots, V_d(z,\mu)$.
 The last two terms in the description of $\mathcal{L} F(x,[\theta])$ involve the derivative with respect to the measure variable as introduced by Lions in his seminal lectures at the \textit{Coll\`ege de France} (see \cite{Cardaliaguet} for details), which we describe in Section \ref{sec:diff}.
Papers \cite{BFY2, BFY, K} present further details of the relevance of the class of nonlinear partial differential equations \eqref{eq:MasterPDE}
% Following from the work of Lasry and Lions, equation \ref{eq:MasterPDE} appears in the description of the equilibria of mean field games (see \cite{lasry2007mean}). 

For linear parabolic PDEs on $[0,T] \times \RR^N$ it is well known from classical works such as \cite{friedmanparabolic, Hormander1967} that under uniform ellipticity or H\"ormander condition, there exist classical solutions even when the initial condition is not differentiable. In this paper, we explore to what extent the same is true for the PDE \eqref{eq:MasterPDE} under a uniform ellipticity assumption. That is, we consider the question of whether the PDE \eqref{eq:MasterPDE} has classical solutions when the initial condition $g$ is \emph{not} differentiable.
For this we exploit a probabilistic representation for the classical solution
\footnote{Because of the new requirement that the solution is differentiable in the measure direction, the notion of a classical solution \\
of \eqref{eq:MasterPDE} needs to be clarified. We do this in Definition \ref{def:classical} below. 
}
of the PDE \eqref{eq:MasterPDE} given in terms of a functional of $\Xt_t$ and of the solution of the following de-coupled equation:
\begin{equation}
\label{eq:MKVdecoupled}
\Xd_t = x + \int_0^t V_0 \left(\Xd_s, \left[\Xt_s \right] \right) \, ds + \sum_{i=1}^d \int_0^t V_i \left(\Xd_s, \left[\Xt_s \right] \right) \, dB^i_s .
\end{equation}
We say that this equation is de-coupled as the law appearing in the coefficients is $\left[\Xt_s \right]$ (the solution of equation \eqref{eq:trueMKV}), rather than the law of $\Xd_t $, the solution to equation \eqref{eq:MKVdecoupled} itself\footnote{Equation \eqref{eq:MKVdecoupled} is therefore not an MVSDE.}.
In the following, we show that, for a certain class of functions  $g: \RR^N \times \p_2(\RR^N) \to \RR$ (not necessarily smooth), the function
\begin{equation}
\label{eq:Udef}
U(t,x,[\theta]) := \EE  \, g \left(\Xd_t,[ \Xt_t ]\right) \quad \mathrm{for} \: (t,x,[\theta]) \in 
[0,T] \times \RR^N \times \p_2(\RR^N)
\end{equation}
solves the PDE \eqref{eq:MasterPDE}. % when the coefficients $\Vs$ are sufficiently smooth. 
A similar  result has been proved  in   \cite{buckdahnmean,crischassdel} under different conditions than ours and for an initial condition $g$ that is  sufficiently smooth.  

For the stochastic flow  $(X_t^x)_{t \geq 0}$ solving a classical SDE with initial condition $x \in \RR^N$, the standard strategy to show that the function $u(t,x):=\EE \, g(X_t^x)$ is a classical solution 
of a linear PDE is to show, using the flow property of $X_t^x$, that for $h>0$, $u(t+h,x)=\EE \, [u(t,X_h^{x})]$ and then show that $u$ is regular enough to apply It\^{o}'s formula to $ u(t,X_h^{x})$. Expanding this process using It\^{o}'s formula and sending $h \to 0$
% along with continuity arguments
 shows that $u$ does indeed solve the related PDE. For MVSDEs, one can develop a similar approach. In this setting, to expand a function depending not only on the process $(\Xd_t )_{t \geq 0}$ (where we can use the usual It\^{o} formula) but also on the flow of measures $\left([\Xt_t] \right)_{t \geq 0}$, we require an extension of the classical chain rule and we use here the chain rule proved in \cite{crischassdel}. Our main focus is therefore to provide conditions under which $U$, defined in \eqref{eq:Udef}, is regular enough to apply the It\^{o} formula and the extended chain rule.

For a general Lipschitz continuous function $g:\RR^N \times \p_2(\RR^N) \to \RR$, we cannot expect for the mapping   $(x,[\theta]) \mapsto \EE [ \, g (\Xd_t,[ \Xt_t ]) ]$ to be differentiable (for a fixed $t>0$) even when the coefficients in the equation for $\Xd_t$ are smooth and uniformly elliptic. This is shown in Example \ref{ex:counter}. We are, however, able to identify a class of non-smooth initial conditions (including interesting examples, see Example \ref{ex:ICs}) for which we can develop integration by parts formulas and establish sufficient smoothness of the associated function $U$. For $g$ in this class, we use Malliavin calculus to show that    
$(x,[\theta]) \mapsto \EE [ \, g (\Xd_t,[ \Xt_t ]) ]$ is differentiable. The differentiability in the measure direction is somewhat surprising since there is no noise added in the measure direction, and this smoothing property seems to be new.
We give further details of our results in the next section.

\subsection{Outline \& Main Results}

In Section 2, we introduce the notation and the basic results related to MVSDEs.  In particular, when describing the smoothness of the coefficients in equations \eqref{eq:trueMKV} and \eqref{eq:MKVdecoupled} in our assumptions, we  introduce the notation $\C^{k,k}_{b,\text{Lip}}(\RR^N\times \p_2(\RR^N);\RR^N)$
for functions $k$-times differentiable with bounded, Lipschitz derivatives, which we introduce precisely in Section \ref{sec:diff}. Similarly, we use the notation $\KK^q_r(E,M)$ to denote processes taking values in a Hilbert space $E$ which are smooth in both Euclidean and measure variables as well as in the Malliavin sense and $M$ denotes how many times the process can be differentiated. This class, which we call the class of Kusuoka-Stroock processes, is introduced in Section \ref{sec:introKSPprob}.
  The class represents a generalization of the class of processes introduced in \cite{KusStrIII} and  analysed in \cite{crisan2010cubature}.

In Section \ref{sec:regsoln}, we prove some results on the differentiability of $\Xd_t$, the solution to equation \eqref{eq:MKVdecoupled}, with respect to the parameters $(x,[\theta])$.
The main result of Section \ref{sec:regsoln} is Theorem \ref{th:XisKSP}, which says that if $\Vs \in \C^{k,k}_{b, Lip}(\RR^N \times \p_2(\RR^N);\RR^N)$, then $(t, x, [\theta]) \mapsto X_t^{x, [\theta]} \in \KK^1_{0}(\RR^N,k)$. This is proved in the Appendix \ref{sec:KSpfs}. We then introduce the uniform ellipticity assumption (UE) in Assumption \ref{ass:UE}, used throughout the rest of the paper. The rest of the section details several corollaries, where we analyse the processes that will  play the r\^ole of Malliavin weights in the integration by parts formulas and identify the class $\KK^q_r(E,M)$ of Kusuoka-Stroock processes to which they belong.

With the main technical results complete, in Section \ref{sec:IBPdecoupled} we develop integration by parts formulas for derivatives of $(x,[\theta]) \mapsto \EE f(\Xd_t)$ under (UE) and the assumption that $\Vs \in \C^{k,k}_{b,\text{Lip}}(\RR^N\times \p_2(\RR^N);\RR^N)$. We do this for derivatives with respect to $x$ and with respect to $\mu$. In particular we show that (see Propositions \ref{th:IBPx1} and \ref{th:IBPmu1}),         for $f \in \C^{\infty}_b(\RR^N;\RR)$, $\Psi \in \KK^q_r(\RR,n)$ and for $|\alpha| + |\beta| \leq [n \wedge (k-2)]$,
 we have
\begin{eqnarray*}
        \partial^{\alpha}_x \,  \EE[(\partial^{\beta} f)(\Xd_t)\, \Psi(t,x, [\theta])]& =& t^{-(|\alpha|+ |\beta|)/2} \, \EE [f(\Xd_t)\, I^3_{\alpha}\left(I^2_{\beta}(\Psi)\right)(t,x, [\theta])] ,\\
 \partial^{\beta}_{\mu} \,       \EE[(\partial^{\alpha}f)(\Xd_t)\, \Psi(t,x, [\theta])](\bv) &= & t^{-(|\alpha|+|\beta|)/2} \EE [f(\Xd_t)\, \I^3_{\beta}\left(I^2_{\alpha}(\Psi)\right)(t,x, [\theta], \bv)] ,
      \end{eqnarray*}        
        where $I^3_{\alpha}\left(I^2_{\beta}(\Psi)\right)$ and $\I^3_{\alpha}\left(I^2_{\beta}(\Psi)\right)$  are defined is defined in Section \ref{sec:IBPFinx} and   $I^3_{\alpha}\left(I^2_{\beta}(\Psi)\right) \in \KK_r^{q+2|\alpha|+3|\beta|}(\RR, m)$ and 
$\I^3_{\alpha}\left(I^2_{\beta}(\Psi)\right) \in \KK_r^{q+4|\alpha|+3|\beta|}(\RR, m),$ where $m=[n \wedge (k-2)]-|\alpha|-|\beta|$. 
\noindent We also consider integration by parts formulas for derivatives of the function
$
x \mapsto \EE f(X_t^{x, \delta_x})$ (see Theorem \ref{th:MKVIBP}).

In Section \ref{sec:masterPDE}, we return our attention to the PDE \eqref{eq:MasterPDE}. In Definition \ref{def:IC}, we introduce the class $\textbf{(IC)}$  of non-differentiable initial conditions $g$ for which we are able to prove 
$(x,[\theta]) \mapsto \EE [ \, g (\Xd_t,[ \Xt_t ]) ]$
 is differentiable. We do this by extending the integration by parts formulas of Section \ref{sec:IBPdecoupled} to cover this class. Then, for $g$ in this class and assuming uniform ellipticity, and the coefficients $\Vs \in \C^{3,3}_{b,Lip}(\RR^N \times \p_2(\RR^N); \RR^N)$ (and possibly bounded depending on the exact form of $g$) we are able to prove the  existence and uniqueness of solutions to the PDE  \eqref{eq:MasterPDE}. In particular, we show (see Theorem \ref{th:masterPDE}) that 
function $U$, defined in \eqref{eq:Udef},
is a classical solution of the PDE \eqref{eq:MasterPDE}. Moreover, $U$
        is unique among all of the classical solutions satisfying the polynomial growth condition $\left|U(t,x,[\theta])\right| \leq C (1+|x|+\|\theta\|_2)^q$ for some $q>0$ and all $(t,x,[\theta]) \in 
[0,T] \times \RR^N \times \p_2(\RR^N)$.

Finally, in Section \ref{sec:densities}, we apply the integration by parts formulae to the study of the density function of $X_t^{x,\delta_x}$. We study the smoothness of the density function and obtain estimates on its derivatives.
The main result (See Theorem \ref{th:densityestimate}) states that, under  suitable  conditions,  $X_t^{x,\delta_x}$ has a density $p(t,x, z)$ such that $(x,z) \mapsto p(t,x, z)$ is differentiable a number of times dependent on the regularity of the coefficients. Indeed, when these derivatives exist, there exist a constant $C$ such that 
        \begin{eqnarray*}             
              | \partial_x^{\alpha} \, \partial_z^{\beta}  p(t,x,z) | \leq C \, (1+ |x|)^{\mu} \,  t^{- \nu}  ,
        \end{eqnarray*}
        where $ \mu = 4|\alpha|+ 3 |\beta| + 3 N$ and  $ \nu = \textstyle \frac{1}{2} (N + | \alpha| + | \beta | )$. Moreover, if $\Vs$ are bounded then the following Gaussian type estimate holds
        \begin{eqnarray*}
                | \partial_x^{\alpha} \, \partial_z^{\beta}  p(t,x,z) | \leq C \,  t^{- \nu} \, \exp \left(- C \, \frac{|z-x|^2}{t} \right) .
        \end{eqnarray*}

\subsection{Comparison with other works}
\label{sec:comparison}

As mentioned previously, the PDE \eqref{eq:MasterPDE} is also studied in \cite{buckdahnmean} and \cite{crischassdel}. Let us explain the relationship between the results in those works and the results in this paper. 

In \cite{buckdahnmean}, the authors prove that derivatives of $(x,[\theta]) \mapsto \Xd_t$ exist up to second order. We also prove this as part of Theorem \ref{th:XisKSP}, although we extend this to derivatives of any order (assuming sufficient smoothness of the coefficients). In \cite{buckdahnmean}, the hypotheses on the continuity and differentiability of the coefficients are the same as ours
The authors then consider initial conditions $g:\RR^N \times \p_2(\RR^N) \to \RR^N$ for which
the derivatives up to second order exist and are bounded, which they use to prove regularity of $U$. Since $g$ is sufficiently smooth, they do not need to impose any non-degeneracy condition on the coefficients. In our work we remove the constraint on the smoothness of $g$ at the expense of assuming non-degeneracy condition on the coefficients of the MVSDEs. In this sense, their results are complementary to ours.

The paper \cite{crischassdel} has a completely different scope. The authors are interested in a nonlinear PDE on $[0,T]\times \RR^N \times \p_2(\RR^N)$, called the master equation in reference to the theory of mean-field games. The PDE we consider is a special case of this, although again they assume that the function $g$ is twice differentiable. Their strategy for proving regularity of $U$ is also different. In their setting, the authors prove that derivatives of the \emph{lifted} flow $\RR^N \times L^2(\Omega) \ni (x,\theta) \mapsto \Xd_t$ exist up to second order (with derivatives in the variable $\theta$ being Fr\'echet derivatives on the Hilbert space $L^2(\Omega)$) where $\Xd_t$ is the forward component in a coupled forward-backward system. They use this result, along with sufficient smoothness of $g$, to prove that the lifted function $\widetilde{U}$, defined on on $[0,T]\times \RR^N \times L^2(\Omega)$ is sufficiently regular in the Fr\'echet sense. They then prove a result which allows them to recover regularity of the second order derivatives of $U$ from properties of the second order Fr\'echet derivatives of $\widetilde{U}$.
Using their strategy, the authors of \cite{crischassdel} are able to impose hypotheses which only involve conditions on derivatives of the coefficients $\partial_{\mu}V_i(x,[\theta],v)$ evaluated at $v=\theta \in L^2(\Omega)$. 

This is in contrast to our assumptions which impose conditions on $\partial_{\mu}V_i(x,[\theta],v)$ for all $(x,[\theta],v) \in \RR^N \times \p_2(\RR^N) \times \RR^N$.

More recently, two other works \cite{Banos, CH} give some partial results related to the smoothness of the solutions of McKean-Vlasov SDEs. In \cite{Banos}, the Malliavin differentiability of McKean-Vlasov SDEs is studied using a stochastic perturbation approach of Bismut type. In \cite{CH}, the strong well-posedness of  a McKean-Vlasov SDEs is proven when the diffusion matrix
is Lipschitz with respect to both the space and measure arguments and uniformly elliptic and the
drift is bounded in space and H\"{o}lder continuous in the measure direction. 
Both works restrict themselves to the particular case when the coefficient dependence on the law of the  solution is of scalar type. We obtain some related results in \cite{CM}, under the same scalar dependence restriction, but under the more general H\"{o}rmander condition.

We base our results on the use of Malliavin calculus techniques. The new integration by parts formulae and, more importantly, the identification of the processes  appearing in these formulae as Kusuoka-Stroock processes is key to our analysis. The use of Kusuoka-Stroock processes is a very versatile tool. Not only that it enables us to identify the solution of the PDE \eqref{eq:MasterPDE}, but the also allows us to study the density of $X_t^{x,\delta_x}$ and obtain both polynomial and Gaussian local bounds  for their derivatives. We are not aware of similar bounds obtained elsewhere in the literature for densities of solutions of MVSDEs.

%%%%%%%%%%%%%%%%%%%%%%%%%%%%%%%%%%%%%%%%%%%%%%%%%%%%%%%%%%%%%%%%%%%%%%
\section{Preliminaries}
\label{sec:prelim}
%%%%%%%%%%%%%%%%%%%%%%%%%%%%%%%%%%%%%%%%%%%%%%%%%%%%%%%%%%%%%%%%%%%%%%

\subsection{Notation \& Basic Setup }
\label{sec:introprob}
We work on a filtered probability space $(\Omega, \F, \mathbb{F}= \{\F_t\}_{t \in [0,T]} , \PP )$ which supports an $\mathbb{F}$-adapted $d$-dimensional Brownian Motion, $B=(B^1, \ldots, B^d)$. We also often denote $B^0(s)=s$ for $s \in [0,T]$. We assume that there is a sufficiently rich sub-$\sigma$-algebra $\mathcal{G} \subset \F$ independent of $B$ such that all measures $\mu \in \p_2(\RR^N)$ correspond to the law of a random variable in $L^2((\Omega,\mathcal{G}, \PP) ;\RR^N)$. Then, we define $\mathbb{F}$ to be the filtration generated by $B$, completed and augmented by $\mathcal{G}$. This is to ensure that in the sequel when we consider processes starting from arbitrary initial conditions $ \theta \in L^2(\Omega;\RR^N)$ these processes will be $\mathbb{F}$-adapted. We denote the $L^p$ norm on $(\Omega, \F,\PP)$ by $\| \cdot \|_p $ and we also introduce the space $\cS^p_T$ of continuous $\mathbb{F}$-adapted processes $\varphi$ on $[0,T]$,  satisfying
\begin{align*}
&\left\| \varphi \right\|_{\cS^p_T} = \left(\EE  \sup_{s \in [0,T]} | \varphi_s|^p \right)^{1/p}<\infty.
\end{align*}
In addition to the probability space $(\Omega, \F , \PP )$, we will also make use of other probability spaces $( \t\Omega, \t \F, \t \PP)$ and $(\widehat{\Omega}, \widehat{\F}, \widehat{\PP})$ when performing the lifting operation associated with the Lions derivative. We assume that these satisfy the same conditions as $(\Omega, \F, \PP)$. We denote the $L^p$ norm on each of these spaces by $\| \cdot \|_p $ unless we want to emphasise which space we are working on, in which case we use $\| \cdot \|_{L^p(\widetilde{\Omega})} $ etc.
%We use $L^p(\Omega)$ to denote the space of variables on $\left(\Omega, \F, \PP \right)$ \emph{taking values in $\RR^N$} with finite $L^p$ norm.
We use $| \cdot |$ to denote the Euclidean norm. Throughout we denote by $\alpha $ and $\beta$ multi-indices on $\{1, \ldots, N\}$ including the empty multi-index. We denote by $Id_N$ the $N \times N$ identity matrix. We also use some terminology from Malliavin calculus: we denote by $\D$ the Malliavin derivative and by $\delta$ its adjoint, the Skorohod integral.
We outline very briefly the basic operators of Malliavin calculus in Appendix \ref{sec:introMalCal}.

\subsection{Basic results on McKean-Vlasov SDEs}
\label{sec:introMKV}

We study McKean-Vlasov SDEs with general Lipschitz interaction. The coefficients are functions from $\RR^N \times \p_2(\RR^N)$ to $\RR^N$, where $\p_2(\RR^N)$ denotes the space of probability measures on $\RR^N$ with finite second moment. We equip this space with the $2$-Wasserstein metric, $W_2$. For a general metric space $(M,d)$, we define the $2$-Wasserstein metric on $\p_2(M)$ by
\begin{align*}
        W_2(\mu, \nu) = \inf_{\Pi \in \p_{\mu,\nu}} \left(\int_{M \times M} d(x,y)^2 \, \Pi(dx,dy) \right)^{1/2},
\end{align*}
where $\p_{\mu,\nu}$ denotes the set of measures on $M \times M$ with marginals $\mu$ and $\nu$. When we refer to the Lipschitz property of the coefficients, it is with respect to product distance on $\RR^N \times \p_2(\RR^N)$.
\begin{proposition}[Existence, Uniqueness and $L^p$ estimates]
        \label{prop:MKVSDE}
        Suppose that $\theta \in L^2( \Omega)$ and $\Vs$ are uniformly Lipschitz continuous, then there exists a unique, strong solution to the equation
        \begin{equation}
        \label{eq:MKVSDE}
        \Xt_t =  \theta +  \sum_{i=0}^d \int_0^t V_i \left(\Xt_s, \left[\Xt_s \right] \right) \, dB^i_s ,
        \end{equation}
        and there exists a constant $C=C(T)$, such that
        \begin{equation}
        \label{eq:MKVmoment}
        \| X^{\theta} \|_{ \cS^2_T} \leq C  \, \left( 1+  \| \theta \|_2 \right).
        \end{equation}
        Similarly, there exists a unique, strong solution to the equation
        \begin{equation}
        \label{eq:decoupled}
        \Xd_t = x  + \sum_{i=0}^d \int_0^t V_i \left(\Xd_s, \left[\Xt_s \right] \right) \, dB^i_s ,
        \end{equation}
        and there exists a constant $C=C(p,T)$, such that for all $p \geq 1$,
        \begin{equation}
        \label{eq:decoupledmoment}
        \| \Xd \|_{\cS^p_T} \leq C  \, \left( 1+ |x|+ \| \theta\|_2  \right).
        \end{equation}
        Moreover, for all $(x, \theta, t), (x', \theta', t') \in \RR^N \times L^2(\Omega) \times [0,T]$ and $p \geq 1$,
        \begin{equation}
        \label{eq:MKVLpreg}
        \left\| X^{x, [\theta]}  - X^{x', [\theta']} \right\|_{\cS^p_T} \leq C \,  \left( |x-x'| + \|\theta - \theta'\|_2  \right) ,
        \end{equation}
        and 
        \begin{equation}
        \label{eq:MKVLpregtime}
        \left\| X^{x, [\theta]}_t  - X^{x, [\theta]}_{t'} \right\|_p \leq C \, (1+|x|+\|\theta\|_2) \,  |t-t'|^{\frac{1}{2}} .
        \end{equation}
        Finally, we have the following flow property for any $t \in [0,T) $, $s \in (t,T]$,  $x \in \RR^N$ and $\theta \in L^2(\Omega)$,
        \[
        \left( X_{t+s}^{x,[\theta]}, X_{t+s}^{\theta} \right) = \left( X_s^{X_{t}^{x,[\theta]},[X_{t}^{[\theta]}]}, X_s^{X_t^{\theta}} \right) \quad \PP-a.s.
        \]
\end{proposition}

\begin{proof}
        The proof is standard and we leave it to the reader. We note that the proof of existence and uniqueness of a solution to equation \eqref{eq:MKVSDE} was proved in \cite{sznitman1991topics} for first-order McKean-Vlasov interaction. The case of a generic Lipschitz McKean-Vlasov interaction is covered in \cite{Jourdain_nonlinearsdes}.  
\end{proof}

%\begin{remark}
%\label{rem:MKV}
%The estimate \eqref{eq:MKVLpreg} implies that for two random variables $\theta, \theta' \in L^2(\Omega)$ with the same law, the processes $X^{x,[\theta]}$ and $X^{x,[\theta']}$ are indistinguishable. This justifies using the notation $[\theta]$ to label the process, since the law of the process depends only on the law of $\theta$ and not the random variable itself.
%\end{remark}

\subsection{Differentiation in $\p_2(\RR^N)$}
\label{sec:diff}

In Section \ref{sec:masterPDE}, we study an SDE with a general McKean-Vlasov dependence. We will be interested in differentiability of the stochastic flow associated to this SDE and an associated PDE on $[0,T] \times \RR^N \times \p_2(\RR^N)$. We thus need a notion of derivative for a function on a space of probability measures.
The notion of differentiability we use was introduced by P.-L. Lions in his lectures at the \emph{Coll\`ege de France}, recorded in a set of notes by Cardaliaguet \cite{Cardaliaguet}. The underlying idea is very well exposed in \cite{carmonacontrolmkv}, which we draw on here.

Lions' notion of differentiability is based on the \emph{lifting} of functions $U: \p_2(\RR^N)\to \RR$ 
into functions $\t U$  defined on the Hilbert space $L^2(\t \Omega;\RR^N)$ over some probability space $(\t\Omega,\t\F,\t\PP)$, $\tilde{\Omega}$ being a Polish space and $\t \PP$ an atomless measure, by setting $\t U(\tX)=U([\tX])$ for $\tX \in L^2(\t\Omega;\RR^N)$. Then, a function $U$
is said to be differentiable  at $\mu_0\in \p_2(\RR^N)$ if there exists a random variable $\t X_0$ with law $\mu_0$ such that the lifted function $\t U$ is Fr\'echet differentiable at $\t X_0$. 
Whenever this is the case, the Fr\'echet derivative of $\t U$ at $\t X_0$ can be viewed as an element of $L^2(\t\Omega;\RR^N)$ by identifying $L^2(\t\Omega;\RR^N)$ and its dual. The derivative in a direction $\t\gamma \in L^2(\t\Omega;\RR^N)$ is given by
\[
D \tilde{U}(\t X_0) (\t \gamma) = \langle D \tilde{U} (\t X_0), \t \gamma \rangle_{L^2(\t\Omega;\RR^N)} = \widetilde{\EE} \left[ D \tilde{U} (\t X_0) \cdot \t \gamma \right].
\]
It then turns out (see Section 6 in \cite{Cardaliaguet} for details.) that the distribution of $D \tilde{U} (\t X_0) \in L^2(\t\Omega;\RR^N)$ depends only upon the law $\mu_0$ and not upon the particular random variable $ \t X_0$ having distribution $\mu_0$.  
It is shown in \cite{Cardaliaguet} that, as a random variable, $D \tilde{U} (\t X_0)$ is of the form $ g_{\mu_0}( \t X_0)$, where $ g_{\mu_0} : \RR^N \rightarrow \RR^N$ is a deterministic measurable function which is uniquely defined $\mu_0$-almost everywhere on $\RR^N$, and is square-integrable with respect to the measure $\mu_0$. We call $\partial_{\mu}U(\mu_0):=g_{\mu_0}$ the derivative of $U$ at $\mu_0$. We use the notation
$\partial_{\mu} U(\mu_{0}, \cdot ) : \RR^N \ni v \mapsto \partial_{\mu} U(\mu_{0},v)
\in \RR^N$, which satisfies, by definition,
\[
D \t U(\t X_0) = g_{\mu_0}(\t X_0) =: \partial_{\mu}U(\mu_0, \t X_0).
\]
This holds for any random variable $\t X_{0}$ with distribution $\mu_0$, irrespective of the probability space on which it is defined.
%\begin{figure}[h!]
%       \centering {\LARGE  \fbox{$L^2 \left(\t \Omega; \RR^N \right)$} }
%       \begin{center}
%               \begin{tikzpicture}[node distance=5.7cm, auto]
%               \node (grad)  { $D \tilde{U}(\t X_0)$ };
%               \node (lift) [left of=grad] { $\tilde{U}( \t X_0)$ };
%               \node (g) [right of=grad] { $D \tilde{U}(\t X_0) = g_{\mu_0}( \t X_0)$ };
%               \node (U) [below of=lift] { $U(\mu_0)$ };
%               \node (dmu) [below of=g] { $\partial_{\mu}U(\mu_0,v):=g_{\mu_0}(v)$ };
%               \draw[-triangle 90,blue] (U) to node {\textcolor{blue}{ Lifting}} (lift);
%               \draw[-triangle 90,blue] (lift) to node {\textcolor{blue}{Fr\'echet derivative}} (grad);
%               \draw[-triangle 90,blue] (grad) to node {\textcolor{blue}{Structure of gradient}} (g);
%               \draw[-triangle 90,blue] (g) to node {\textcolor{blue}{Definition}} (dmu);
%               \end{tikzpicture}
%       \end{center}
%       \centering {\LARGE \fbox{ $\p_2 \left( \RR^N \right)$}}
%       \caption{Schematic representation of the Lions derivative.}
%\end{figure}
%In particular, this identity is invariant by modification of the space $\tilde{\Omega}$ and of the variable $\t X_{0}$
%used for the representation  the lifting $\t U$,
%in the sense that $D \t U( \t X_0)$ always reads as $\partial_{\mu}U(\mu_{0})(\t X_{0})$, whatever the choices of 
%$\t \Omega$ and $\t X_{0}$ are.

In the sequel, we will consider functions which are differentiable globally on $ \p_2(\RR^N)$. Moreover, we will consider functions where for each $\mu \in \p_2(\RR^N)$, there exists a version of the derivative $\partial_{\mu}U(\mu)$ which is assumed to be a priori continuous as a function
\[
\p_2(\RR^N ) \times \RR^N \ni(\mu,v) \mapsto \partial_{\mu}U(\mu,v) \in \RR^N.
\]
In this case such a version is unique since, for each  $\theta \in L^2(\Omega; \RR^N)$, $\partial_{\mu}U([\theta],v)$ is defined $[\theta](dv)$-a.e., so taking a Gaussian random variable $G$ independent of $\theta$, and $\epsilon >0$, $\partial_{\mu}U([\theta+\epsilon G],v)$ is defined $(dv)$-a.e. and taking $\epsilon \to 0$ and using the continuity of $\partial_{\mu}U$, identifies $\partial_{\mu}U([\theta],v)$ uniquely. We show how this definition works in practice in Examples \ref{ex:scalar} and \ref{ex:first}.
%To see how this definition works in practice, consider the scalar interaction case:
%\begin{equation}
%\label{fo:Uofmu}
%U(\mu)=\int_{\RR^N} h(x)\mu(dx)
%\end{equation}
%for some scalar differentiable function $h$ defined on $\RR^N$. Indeed, in this case, $\t U( \t X_0)=\t\EE [h( \t X_0)]$ and $D \t U( \t X_0)( \t \gamma)=\t\EE[\partial h( \t X_0)\cdot \t \gamma]$ so that we can think of $\partial_\mu U(\mu_0,v)$ as $\partial h (v)$. 

For a function $f: \p_2(\RR^N) \to \RR^N$, we can straightforwardly apply the above discussion to each component of $f=(f^1, \ldots, f^N)$.  To extend to higher derivatives we note that $\partial_{\mu} f^i $ takes values in $\RR^N£$, so we denote its components by $ ( \partial_{\mu} f^i)_j : \p_2(\RR^N) \times \RR^N \to \RR$ for $j=1, \ldots, N$ and, for a fixed $v \in \RR^N$, we can discuss again the differentiability of $ \p_2(\RR^N) \ni \mu \mapsto (\partial_{\mu} f^i)_j(\mu,v) \in \RR$. If the derivative of this function exists and there is continuous version of
\[
\p_2(\RR^N) \times \RR^N \times \RR^N \ni       (\mu,v_1,v_2) \mapsto \partial_{\mu}( \partial_{\mu} f^i)_j(\mu,v_1,v_2) \in \RR^N,
\]
then it is unique. It makes sense to use the multi-index notation $\partial^{(j,k)}_{\mu} f^i: = ( \partial_{\mu}( \partial_{\mu} f^i)_j)_k$.
Similarly, for higher derivatives, if for each $(i_0, \ldots, i_n) \in \{1, \ldots, N\}^{n+1}$, 
\[
\underbrace{\partial_{\mu}(\partial_{\mu} \ldots (\partial_{\mu} }_{n \text{ times}} f^{i_0})_{i_1} \ldots)_{i_n}
\]
exists, we denote this $\partial^{\alpha}_{\mu}f^{i_0}$ with $\alpha = (i_1, \ldots, i_n)$. 
Now, each derivative in $\mu$ is a function of an `extra' variable, so $\partial^{\alpha}_{\mu}f^{i_0}: \p_2(\RR^N) \times (\RR^N)^n \to \RR$. We always denote these variables, by $v_1, \ldots, v_n$, so 
\[
\p_2(\RR^N) \times (\RR^N)^n \ni (\mu, v_1, \ldots, v_n)  \mapsto \partial^{\alpha}_{\mu}f^{i_0}(\mu, v_1, \ldots, v_n)  \in \RR .
\]
When there is no possibility of confusion, we will abbreviate $(v_1, \ldots, v_n)$ to $\bv$, so that
\[
\partial^{\alpha}_{\mu}f^{i_0}(\mu, \bv)  = \partial^{\alpha}_{\mu}f^{i_0}(\mu, v_1, \ldots, v_n).
\]
For $\bv=(v_1, \ldots, v_n) \in (\RR^N)^n$, we will denote
\[
| \bv | := |v_1| + \ldots + |v_n|,
\]
with $|\cdot|$ the Euclidean norm on $\RR^N$.
It then makes sense to discuss derivatives of the function $\partial^{\alpha}_{\mu}f^{i_0}$ with respect to the variables $v_1, \ldots, v_n$.
If, for some $j \in \{1, \ldots, N\}$ and all $(\mu, v_1, \ldots,v_{j-1}, v_{j+1}, \ldots,  v_n) \in \p_2(\RR^N) \times (\RR^N)^{n-1}$,
\[
\RR^N \ni v_j \mapsto \partial^{\alpha}_{\mu}f^{i_0}(\mu, v_1, \ldots, v_n)
\]
is $l$-times continuously differentiable, we denote the derivatives $\partial_{v_j}^{\beta_j}\partial^{\alpha}_{\mu}f^{i_0}$, for $\beta_j$ a multi-index on $\{1, \ldots, N\}$ with $|\beta_j| \leq l$. Similar to the above, we will denote by $\bb$ the $n$-tuple of multi-indices $(\beta_1,\ldots, \beta_n)$. We also associate a length to $\bb$ by
\[ 
|\bb|:=|\beta_1|+ \ldots +|\beta_n|,
\]
and denote $\# \bb:=n$. Then, we denote by $\B_n$ the collection of all such $\bb$ with $\# \bb:=n$, and $\B:= \cup_{n \geq 1} \B_n$. Again, to lighten notation, we will use
\[
\partial^{\bb}_{\bv} \partial^{\alpha}_{\mu}f^i (\mu, \bv):=    \partial_{v_{n}}^{\beta_{n}} \ldots \partial^{\beta_1}_{v_1} \partial^{\alpha}_{\mu}f^i(\mu, v_1, \ldots, v_n).
\]
The coefficients in equations \eqref{eq:MKVSDE} and \eqref{eq:decoupled} are of the type $\Vs : \RR^N \times \p_2(\RR^N) \to \RR^N$, so depend on a Euclidean variable as well as a measure variable. Considering functions on $\RR^N \times \p_2(\RR^N)$ raises a question about whether the order in which we take derivatives matters. A result from \cite{buckdahnmean} says that derivatives commute when the mixed derivatives are Lipschitz continuous.

\begin{lemma}[Lemma 4.1 in \cite{buckdahnmean} ]
        Let $g: \RR \times \p_2(\RR) \to \RR$ and suppose that the derivative functions
        \[
        (x, \mu, v )  \in \RR \times \p_2(\RR) \times \RR \to \left( \partial_{x} \partial_{\mu}  g( x, \mu, v), \partial_{\mu} \partial_{x}  g( x, \mu, v) \right) \in \RR \times \RR
        \] 
        both exist and are Lipschitz continuous: i.e. there exists a constant $C>0$ such that
        \[
        \left| \left( \partial_{x} \partial_{\mu}  g, \partial_{\mu} \partial_{x}  g \right)( x, \mu, v) - \left( \partial_{x} \partial_{\mu}  g, \partial_{\mu} \partial_{x}  g \right)( x', \mu', v') \right| \leq C \, \left( |x-x'| + W_2(\mu,\mu') + |v-v'|
        \right).
        \]
        Then, the functions $\partial_{x} \partial_{\mu}  g$ and $\partial_{\mu} \partial_{x}  $ are identical. 
\end{lemma}
\label{lem:swapderivative}

\noindent With this in mind, we can introduce the following definition.

\begin{definition}[$\C^{n,n}_{b,Lip}(\RR^N \times \p_2(\RR^N) ; \RR^N)$]
\leavevmode
        \label{def:Ckk}
        \begin{itemize}
                \item[(a)]
                Let $V: \RR^N \times \p_2(\RR^N) \to \RR^N$ with components $V^1, \ldots, V^N: \RR^N \times \p_2(\RR^N) \to \RR$. We say that $V \in \C^{1,1}_{b, Lip}(\RR^N \times \p_2(\RR^N) ;\RR^N)$ if the following hold true: for each $i=1, \ldots, N$, $\partial_{\mu} V^i$ exists and $\partial_xV$ exists. Moreover, assume that for all $(x, \mu, v) \in \RR^N \times \p_2(\RR^N) \times \RR^N$
                \begin{equation*}
                        %\label{eq:l2bd}
                        \left|  \partial_x V^i(x,\mu) \right| + \left| \partial_\mu V^i \left(x, \mu, v \right) \right| \leq C.
                \end{equation*}
                In addition, suppose that $\partial_{\mu }V^i$ and $\partial_xV$ are Lipschitz in the sense that
                for all $(x, \mu, v), ( x', \mu', v' ) \in \RR^N \times \p_2(\RR^N) \times \RR^N$, 
                \begin{align*}
                        \left| \partial_{\mu} V^i(x,\mu,v) - \partial_{\mu} V^i(x', \mu',v') \right| &\leq C \left( |x-x'| + W_2(\mu, \mu') + |v-v'| \right), \\
                        \left| \partial_{x} V(x,\mu) - \partial_{x} V(x', \mu') \right| & \leq C \left( |x-x'| + W_2(\mu, \mu') \right).
                \end{align*}
                
                \item[(b)]      We say that $V \in \C^{n,n}_{b, Lip}(\RR^N \times \p_2(\RR^N) ; \RR^N) $ if the following hold true:  for each $i=1, \ldots, N$, and all multi-indices $\alpha$ and $\gamma$ on $\{1, \ldots, N\}$ and all $\bb \in \B$ satisfying $|\alpha| + |\bb| + |\gamma| \leq n$, the derivatives
                \[
                \partial^{\gamma}_x     \partial^{\bb}_{\bv} \partial^{\alpha}_{\mu}V^i(x,\mu,\bv),
      \partial^{\bb}_{\bv} \partial^{\alpha}_{\mu}\partial^{\gamma}_xV^i(x,\mu,\bv),       \partial^{\bb}_{\bv} \partial^{\gamma}_x\partial^{\alpha}_{\mu}V^i(x,\mu,\bv)  
                \]
                exist. Moreover, suppose that each of these derivatives is bounded         and Lipschitz.
                \item[(c)] We say that $h \in \C^n_{b,Lip}(\p_2(\RR^N) ; \RR^N)$ if $h:\p_2(\RR^N) \to \RR^N$  does not depend on a Euclidean variable but otherwise satisfy the conditions in part (b). 
        \end{itemize}
\end{definition}

%\begin{definition}
%\label{def:Ckk}
%Let $V : \RR^N \times \p_2(\RR^N) \to \RR^N$. Suppose that for each $x \in \RR^N$, $\mu \mapsto V(x, \mu) \in \C^n_{b,Lip}(\p_2(\RR^N);\RR^N)$ and for each $\mu \in \p_2(\RR^N)$, $x \mapsto V(x, \mu)$ is $n$ times continuously differentiable with bounded and Lipschitz derivatives. Then, we say $V \in \C^{n,n}_{b,Lip}(\RR^N \times \p_2(\RR^N); \RR^N)$.
%\end{definition}

\begin{remark}
        \label{rem:derivs}
        \begin{enumerate}
                \item For functions $V:\RR^N \times \p_2(\RR^N) \to \RR^N$, we will also consider the lifting $\t V : \RR^N \times L^2( \Omega) \to \RR^N$. Then, for $\xi \in L^2(\Omega)$, $\t V(\xi, \xi)$ should be interpreted as $\t V(\xi( \omega), \xi) = V(\xi( \omega), [\xi])$ with the first argument being considered pointwise by $ \omega$ and the second depending on the  random variable $\xi$ through its law.
                
                \item From the bounds in Definition \ref{def:Ckk}(a), we have the following simple consequences for the Fr\'echet derivative of the lifting $\tilde{V}$ of $V$: for all $x,x' \in \RR^N$ and $\theta,\theta', \gamma,\gamma' \in L^2(\Omega)$,
                \begin{align*}
                        & \left|D \tilde{V}(x,\theta)(\gamma) \right| \leq C \, \|\gamma\|_2 \\
                        & \left|D \tilde{V}(x,\theta)(\gamma) - D \tilde{V}(x',\theta')(\gamma') \right| \leq C \left[\|\gamma\|_2 \left(|x-x'|+\|\theta -  \theta'\|_2 \right) + \|\gamma-\gamma'\|_2 \right].
                \end{align*}
                
                \item Note that we cannot interchange the order of $\partial_\mu$ and $\partial_v$ in $  \partial_v \partial_{\mu}V(x,\mu,v)$ since $V(x,\mu)$ does not depend on $v$. However, if $V \in \C^{n,n}_{b, Lip}(\RR^N \times \p_2(\RR^N) ; \RR^N) $ then for all $\alpha, \bb, \gamma$ with $|\alpha| + |\bb| + |\gamma| \leq n$, we have that
                \begin{align*}
                        \partial^{\gamma}_x \partial^{\bb}_{\bv} \partial^{\alpha}_{\mu} V(x,\mu,\bv) =\partial^{\bb}_{\bv}  \partial^{\gamma}_x  \partial^{\alpha}_{\mu} V(x,\mu,\bv) = \partial^{\bb}_{\bv}    \partial^{\alpha}_{\mu} \partial^{\gamma}_x V(x,\mu,\bv)
                \end{align*}
                due to Lemma 2.2.
        \end{enumerate}
\end{remark}

\noindent We now introduce some concrete examples of functions $V: \RR^N \times \p_2(\RR^N) \to \RR^N$.

\begin{example}[Scalar interaction]

        \label{ex:scalar}
        Take $U \in \C^{k+1}_b(\RR^N \times \RR;\RR^N)$, $\phi \in \C^{k+1}_b(\RR^N;\RR)$ and $\textstyle V(x,\mu):=U(x, \int \phi d \mu )$.
\end{example}

\begin{example}[First-order interaction]
        \label{ex:first}
        
        Take $W \in \C^{k+1}_b(\RR^N \times \RR^N;\RR^N)$ and $ \textstyle V(x,\mu):= \int W(x, \cdot )d \mu $. 
\end{example}

\begin{lemma}
        \label{lem:examples}
        In both examples, $V \in \C^{k,k}_{b,Lip}(\RR^N \times \p_2(\RR^N); \RR^N)$.
\end{lemma}
The proof is straightforward.
%\begin{proof}
%       For the scalar interaction example: set $ \textstyle F(\mu) = \int \phi d \mu$, so that $V(x,\mu)=U(x,F(\mu))$. The lifting of $F$ is given by $\tilde{F}(\theta)= \EE [\phi(\theta)]$ with Fr\'echet derivative $D \tilde{F}(\theta)(\gamma) = \EE \left[\partial \phi(\theta)\, \gamma \right]$, so that $\partial_{\mu} F(\mu,v) = \partial \phi(v)$. So, denoting by $\partial_y$ differentiation in the second argument of $U$,
%       \[
%       \partial_{\mu} V^i(x,\mu,v) =  \partial_y U^i(x, \textstyle \int \phi d \mu) \, \partial \phi(v) \quad \text{ for } i=1, \ldots, N.
%       \]
%       Similar expressions can be found for higher order derivatives using the regular chain and product rules. All the required properties follow from the boundedness and Lipschitz properties of derivatives of $U$ and $\phi$, and the fact that a product of bounded Lipschitz functions is itself Lipschitz.
%       %
%       For the first order interaction example, we have the lifting for $\theta, \gamma \in L^2(\Omega)$,
%       \[
%       \tilde{V}^i(x, \theta) = \EE \left[W^i(x, \theta)\right],
%       \]
%       so that 
%       \[
%       D\tilde{V}^i(x, \theta)(\gamma) = \EE \left[ \partial_y W^i(x, \theta) \, \gamma\right],
%       \]
%       and
%       \[
%       \partial_{\mu } V^i(x,[\theta], v) = \partial_y W^i(x,v).
%       \]
%       This no longer depends on the measure $[\theta]$ and all further properties follow easily.
%\end{proof}

\subsection{Kusuoka-Stroock processes on $\RR^N \times \p_2(\RR^N)$}
\label{sec:introKSPprob}

In Section \ref{sec:IBPdecoupled}, we develop integration by parts formulas modelled on those developed in works of Kusuoka \cite{Kusuoka3} along with Stroock \cite{KusStrIII} for solutions of classical SDEs. These integration by parts formulas take the form
\begin{align*}
\partial^{\alpha}_x \,  \EE \left[f(\Xd_t)\, \Psi(t,x, [\theta]) \right] &=   \EE  \left[f(\Xd_t)\, \Psi_{\alpha}(t,x, [\theta]) \right] ,\\
\partial^{\beta}_{\mu} \,       \EE \left[ f(\Xd_t)\, \Psi(t,x, [\theta]) \right](\bv) & =  \EE \left[f(\Xd_t)\, \Psi_{\beta}(t,x, [\theta], \bv) \right] 
\end{align*}
for processes $\Psi, \Psi_{\alpha},\Psi_{\beta}$ belonging to a specific class. We work with a class of processes similar to one introduced in \cite{KusStrIII}, which we call the class of Kusuoka-Stroock processes.
%and we extend the definition to processes depending on a probability measure.

\begin{definition}[ Kusuoka-Stroock processes on $\RR^N \times \p_2(\RR^N)$]\label{KSPDefn}
        Let $E$ be a separable Hilbert space and let $r \in \RR$, $q,M \in \mathbb{N}$. We
        denote by $\KK^q_{r}(E,M)$ the set of processes $\Psi : [0,T] \times
        \mathbb{R}^N \times \p_2(\RR^N) \to \mathbb{D}^{M,\infty}(E)$ satisfying the
        following:
        \begin{enumerate} 
                \item For any multi-indices $\alpha, \bb$, $\gamma$ satisfying $\vert \alpha \vert + |\bb| +|\gamma| \leq  M$, the function 
                \[
                [0,T] \times \RR^N \times \p_2(\RR^N) \ni (t,x , [\theta]) \mapsto \partial^{\gamma}_x \partial^{\bb}_{\bv}  \partial^{\alpha}_{\mu} \Psi(t,x,[\theta], \bv) \in L^p(\Omega)
                \]
                exists and is continuous for all $p \geq 1$.
                \item For any $p \geq 1$ and $m \in \NN$ with $|\alpha| + |\bb| + |\gamma| +m \leq M$, we have 
                \begin{align}\label{ksineq}
                        \begin{split}
                                \sup_{ \bv \in (\RR^N)^{\# \bb}} \sup_{t \in (0,T]} t^{-r/2} & \left\| \partial^{\gamma}_x \partial^{\bb}_{\bv}  \partial^{\alpha}_{\mu} \Psi(t,x,[\theta], \bv) \right\|_{ \DD^{m,p}(E) }  \leq C \, \left(1 + |x| + \|\theta\|_2 \right)^q.
                        \end{split}
                \end{align}
        \end{enumerate}
\end{definition}

\begin{remark}
                This definition is different to that in \cite{KusStrIII} in the following ways:
                \begin{enumerate}
                \item The processes depend on a parameter $\mu \in \p_2(\RR^N)$.
                \item We keep track of polynomial growth in $x$ of the $\DD^{m,p}$-norm through a parameter $q>0$ instead of requiring it to be uniformly bounded.
                \item We require continuity in $L^p(\Omega)$ rather than almost surely.
                \end{enumerate}
\end{remark}

\begin{remark}
        \begin{enumerate}
                \item The number $M$ denotes how many times the Kusuoka-Stroock process can be differentiated; $q$ measures the polynomial growth of the $\DD^{m,p}$-norm of the process in $(x,[\theta])$, and $r$ measures the growth in $t$.
                \item In the definition, we are able to stipulate that the $\DD^{m,p}$-norm of all the derivatives will be uniformly bounded w.r.t. $\bv$ because in the sequel the only dependence on $\bv$ in any Kusuoka-Stroock processes will come from $\partial_{\mu} \Xd_t(v)$. In Lemma \ref{lem:linearestimate} $\partial_{\mu} \Xd_t(v)$ is bounded w.r.t $v$ and this carries over to the $\DD^{m,p}$-norm.
        \end{enumerate}
        
\end{remark}

To analyse the density of solutions of the MVSDE \eqref{eq:MKVSDE} started from a fixed initial point in $\RR^N$, it is useful to have notation for Kusuoka-Stroock processes which do not depend on a measure $\mu \in \p_2(\RR^N)$. We denote this class by $\K_r^q(\RR,M)$.
The following lemma says that if we take a Kusuoka-Stroock process on $\RR^N \times \p_2(\RR^N)$ and evaluate its measure argument at a Dirac mass, then this forms a Kusuoka-Stroock process on $\RR^N $. Its proof is straightforward.

\begin{lemma}
        \label{lem:conversion}
        If $\Psi \in \KK_r^q(\RR, M)$ and we define $\Phi(t,x):=\Psi(t,x,\delta_x)$, then $\Phi \in \K_r^q(\RR,M)$.
\end{lemma}
\section{Regularity of Solutions of McKean-Vlasov SDEs}
\label{sec:regsoln}
%%%%%%%%%%%%%%%%%%%%%%%%%%%%%%%%%%%%%%%%%%%%%%%%%%%%%%%%%%%%%%%%%%%%%%%%%%%%%%%%%%%%%%
This section contains some basic results about solutions of the equations involved, their integrability and their differentiability with respect to parameters. Existence and uniqueness of solutions to \eqref{eq:MKVdecoupled} is covered in Section \ref{sec:introMKV}.

\begin{proposition}[First-order derivatives]
        \label{prop:firstderivs}
        Suppose that \newline $\Vs \in \C^{1,1}_{b,\text{Lip}}(\RR^N \times \p_2(\RR^N);\RR^N)$. Then the following hold:
        \begin{itemize}
                \item[(a)]
                There exists a modification of $X^{x, [\theta]}$ such that, for all $t \in [0,T]$, the map $x \mapsto X_t^{x, \theta}$ is $\PP$-a.s. differentiable. We denote the derivative $\partial_x X^{x, [\theta]}$ and note that it solves the following SDE
                \begin{equation}
                \label{eq:jac}
                \partial_x \Xd_t = \text{Id}_N  + \sum_{i=0}^d \int_0^t   \partial V_i \left(\Xd_s, [\Xt_s ] \right) \, \partial_x \Xd_s \, dB^i_s.
                \end{equation}
                
                \item[(b)] For all $t \in [0,T]$, the maps $\theta \mapsto \Xt_t$ and $\theta \mapsto X_t^{x, [\theta]}$ are Fr\'echet differentiable in $L^2(\Omega)$, i.e. there exists a linear continuous map $D \Xt_t : L^2(\Omega) \to L^2(\Omega)$ such that for all $ \gamma \in L^2(\Omega)$, 
                \[
                \| X_t^{\theta +  \gamma}- \Xt_t - D \Xt_t(\gamma)\|_2 =o(\|\gamma\|_2) \:\: \text{  as  } \:\: \|\gamma\|_2 \to 0 ,
                \]
                and similarly for $\Xd_t$. These processes satisfy the following  stochastic differential equations
                \begin{align}
                        \label{eq:Frech} D \Xd_t (\gamma) &=  \sum_{i=0}^d \int_0^t \left[ \partial V_i(\Xd_s, [X_s^{\theta}]) \, D \Xd_s(\gamma) +   D \tilde{V}_i(\Xd_s,  \Xt_s )(D  \Xt_s (\gamma) ) \right] \, dB^i_s ,\\
                        \label{eq:FrechMKV}     D \Xt_t (\gamma) &= \gamma + \sum_{i=0}^d \int_0^t \left[ \partial V_i(\Xt_s, [X_s^{\theta}]) \, D\Xt_s (\gamma) +    D \tilde{V}_i(\Xt_s,  \Xt_s )(D  \Xt_s (\gamma) ) \right] \, dB^i_s ,
                \end{align}
                where we denote by $\tilde{V}_i$ the lifting of $V_i$ to a function on $\RR^N \times L^2(\Omega)$.
                Moreover, for each $x \in \RR^N$, $t \in [0,T]$, the map $\p_2(\RR^N) \ni [\theta] \mapsto \Xd_t \in L^p(\Omega)$ is differentiable for all $p \geq 1$. So, $\partial_{\mu} X^{x, [\theta]}_t(v) $ exists and
                % exists as a function
                %  \[
                %        \p_2(\RR^N) \times \RR^N \ni ([\theta], v) \mapsto \partial_{\mu} X_t^{x,[\theta]}(v) \in L^2(\Omega;\RR^{N \times N}).
                %  \]
                it satisfies the following equation
                \begin{align}
                        \label{eq:partialmuX}
                        \begin{split}
                                \partial_{\mu} \Xd_t(v) = \sum_{i=0}^d \int_0^t & \bigg\{ \partial V_i(\Xd_s, [X_s^{\theta}]) \,\partial_{\mu} \Xd_s(v) 
                                  + \widetilde{\EE} \left[ \partial_{\mu}V_i (\Xd_s, [X_s^{\theta}], \widetilde{ X}^{v,[\theta]}_s ) \, \partial_x \widetilde{ X}_s^{v, [\theta]}    \right]  \\
                                & \quad + \widetilde{\EE} \left[ \partial_{\mu}V_i (\Xd_s, [X_s^{\theta}], \widetilde{ X}^{\t \theta}_s ) \,  \partial_{\mu} \widetilde{ X}_s^{\t\theta,[\theta]}(v)   \right] \bigg\} dB^i_s ,
                        \end{split}
                \end{align}
                where $\widetilde{ X}^{\t \theta}_s$ is copy of $\Xt_s$ on the probability space $(\t \Omega, \t \F, \t \PP)$ driven by the Brownian motion $\t B$ and with initial condition $\t \theta$. Similarly, $ \partial_x  \widetilde{X}_s^{v, [\theta]} $ is a copy of $ \partial_x  X_s^{v, [\theta]} $ driven by the Brownian motion $\t B$ and $\partial_{\mu} \widetilde{ X}_s^{\t\theta,[\theta]}(v)= \left. \partial_{\mu} \widetilde{ X}_s^{x,[\theta]}(v) \right|_{x = \t \theta}$.
                Finally, the following representation holds for all $\gamma \in L^2(\Omega)$:
                \begin{equation}
                \label{eq:Frechrep}
                D \Xd_t (\gamma) = \widetilde{\EE} \left[\partial_{\mu} \Xd_t(\t \theta) \, \t \gamma \right].
                \end{equation} 
                
                \item[(c)] For all $t \in [0,T]$, $\Xd_t, \Xt_t \in \DD^{1, \infty}$. Moreover, $\D_r \Xd = \left(\D^j_r (\Xd)^i \right)_{\substack{1 \leq i \leq N \\ 1 \leq j \leq d}}$ satisfies, for $0 \leq r \leq t$
                \begin{equation}
                \label{eq:mall}
                \D_r \Xd_t =  \sigma(\Xd_r, [\Xt_r])  + \sum_{i=0}^d \int_r^t \partial V_i(\Xd_s, [\Xt_s]) \, \D_r \Xd_s \, dB^i_s ,
                \end{equation}
                where $\sigma(z, \mu)$ is the $N \times d$ matrix with columns $V_1(z,\mu), \ldots, V_d(z,\mu)$.
        \end{itemize}
\end{proposition}

\begin{proof}
        \begin{itemize}
                \item[(a)] Recalling again that $X^{x, [\theta]}$ satisfies a classical SDE with time-dependent coefficients, it follows from \cite{Kun84}  Theorem 4.6.5 there exists a modification of $X^{x, [\theta]}_t$ which is continuously differentiable in $x$, and the first derivative satisfies equation \eqref{eq:jac}.
                %The bound \eqref{eq:decoupledmoment}, follows from Lemma \ref{lem:apriori1} with $a_0=Id_N$, $a_1(s)=\partial V_i(\Xd_s,[X^{\theta}_s])$, $a_2=a_3=0$. Note the bound is independent of $x, \theta$ - it depends only on the bound on $\partial V_i$.
                
                \item[(b)] It is shown in \cite[Lemma 4.17]{crischassdel} that the map $ \theta \mapsto (\Xt_t,\Xd_t)$ is Fr\'echet differentiable. It is then easy to see the Fr\'echet derivative processes satisfy equations \eqref{eq:Frech} and \eqref{eq:FrechMKV}. Now, we follow the idea in \cite{buckdahnmean} to show that $\partial_{\mu} \Xd_t(v)$ solves equation \eqref{eq:partialmuX}. We first re-write the equation for $D \Xt_t(\gamma)$ in terms of $\partial_{\mu}V_i$ instead of the Fr\'echet derivative of the lifting $\t V_i$, as follows
                \begin{align}
                        \label{eq:FrechMKValt}\begin{split}
                                D \Xt_t (\gamma) = \gamma + \sum_{i=0}^d  \int_0^t & \bigg\{  \partial V_i(\Xt_s, [X_s^{\theta}]) \, D\Xt_s (\gamma)  + \widetilde{\EE} \left[ \partial_{\mu} V_i \left(\Xt_s, [\Xt_s], \widetilde{X}^{\t \theta}_s\right) D\widetilde{X}_s^{\t \theta}(\widetilde{\gamma})\right]  \bigg\} \, dB^i_s .
                        \end{split}
                \end{align}
                
                Consider the equation satisfied by $\partial_{\mu}  X_s^{\t\theta,[\theta]}(v)$,
                 evaluated at $v= \widehat{\theta}$ and multiplied by $\widehat{\gamma}$ with both random variables defined on a probability space $(\widehat{\Omega}, \widehat{\F}, \widehat{\PP})$. Taking expectation with respect to $\widehat{\PP}$, we get
                \begin{align}
                        \begin{split}
                                \label{eq:productU}
                                \widehat{\EE}\left[\partial_{\mu}  X_t^{\t\theta,[\theta]}(\widehat{\theta})\, \widehat{\gamma} \right] = \sum_{i=0}^d  \int_0^t \bigg\{  \partial V_i(\Xt_s, [X_s^{\theta}]) \,& \widehat{\EE}\left[\partial_{\mu}  X_s^{\t\theta,[\theta]}(\widehat{\theta})\, \widehat{\gamma} \right] 
                                 + \widehat{\EE} \widetilde{\EE} \left[ \partial_{\mu}V_i (\Xt_s, [X_s^{\theta}], \widetilde{ X}^{\hat{\theta}, [ \theta]}_s )  \partial_x  \widetilde{X}_s^{\widehat{\theta}, [\theta]} \, \widehat{\gamma} \right] \\
                                &  \quad + \widetilde{\EE} \left[ \partial_{\mu}V_i (\Xt_s, [X_s^{\theta}], \widetilde{ X}^{\t \theta}_s ) \widehat{\EE} \left[  \partial_{\mu}  \widetilde{X}_s^{\t\theta,[\theta]}(\widehat{\theta}) \, \widehat{\gamma} \right] \right] \bigg\}dB^i_s.
                        \end{split}
                \end{align}
                In the above equation, we are able to take $\widehat{\gamma}$ inside the It\^{o} integral with no problem since it is defined on a separate probability space to the Brownian motion, $B$. We are also able to interchange the order of the It\^{o} integral and  expectation with respect to $\widehat{\PP}$ using a stochastic Fubini theorem (see for example \cite[Theorem 65]{protter}).
                Again, since $(\widehat{\theta}, \widehat{\gamma})$ are defined on a separate probability space,
                \begin{align*}
                        \widehat{\EE} \widetilde{\EE} & \left[ \partial_{\mu}V_i (\Xt_s, [X_s^{\theta}], \widetilde{ X}^{\hat{\theta}, [ \theta]}_s )  \partial_x  \widetilde{X}_s^{\widehat{\theta}, [\theta]} \, \widehat{\gamma} \right]  
                        = \widetilde{\EE}  \left[ \partial_{\mu}V_i (\Xt_s, [X_s^{\theta}], \widetilde{ X}^{\t \theta}_s )   \partial_x  \widetilde{X}_s^{\tilde{\theta}, [\theta]}  \, \tilde{\gamma}  \right],
                \end{align*}
                which we can replace in equation \eqref{eq:productU} to get:
                \begin{align}
                        \begin{split}
                                \label{eq:productU2}
                                \widehat{\EE}\left[ \partial_{\mu}  X_t^{\t\theta,[\theta]}(\widehat{\theta}) \, \widehat{\gamma} \right] = \sum_{i=0}^d  \int_0^t \bigg\{  \partial & V_i(\Xt_s, [X_s^{\theta}]) \, \widehat{\EE}\left[ \partial_{\mu}  X_s^{\t\theta,[\theta]}(\widehat{\theta}) \, \widehat{\gamma} \right]  \\
                                & 
                                 + \widetilde{\EE} \left[ \partial_{\mu}V_i (\Xt_s, [X_s^{\theta}], \widetilde{ X}^{\t \theta}_s ) \left(\partial_x  \widetilde{X}_s^{\tilde{\theta}, [\theta]}  \, \tilde{\gamma}  + \widehat{\EE} \left[  \partial_{\mu}  \widetilde{X}_s^{\t\theta,[\theta]}(\widehat{\theta}) \, \widehat{\gamma} \right] \right) \right] \bigg\}dB^i_s.
                        \end{split}
                \end{align}
                Now, taking equation \eqref{eq:jac}, satisfied by $\partial_x \Xd_t$ and evaluating at $x= \theta$, multiplying by $\gamma$ and adding to equation \eqref{eq:productU}, we see that $ \partial_x X^{\theta,[\theta]}_t \gamma + \widehat{\EE}\left[ \partial_{\mu}  X_t^{\t\theta,[\theta]}(\widehat{\theta})\, \widehat{\gamma} \right] $ is equal to
                \begin{align*}
                         \gamma +  \sum_{i=0}^d \int_0^t  \bigg\{ &  \partial V_i(\Xd_s, [X_s^{\theta}]) \, \left( \partial_x X^{\theta,[\theta]}_s \gamma + \widehat{\EE}\left[ \partial_{\mu}  X_s^{\t\theta,[\theta]}(\widehat{\theta})\, \widehat{\gamma} \right] \right) \\
                        & + \widetilde{\EE}  \left[ \partial_{\mu}V_i (\Xd_s, [X_s^{\theta}], \widetilde{ X}^{\t \theta}_s ) \left(  \partial_x  \widetilde{X}_s^{\tilde{\theta}, [\theta]} \, \widetilde{\gamma}  + \widehat{\EE}\left[ \partial_{\mu}  \widetilde{X}_s^{\t\theta,[\theta]}(\widehat{\theta}) \widehat{\gamma}  \right] \right)  \right] \bigg\} dB^i_s.
                \end{align*}
                
                One can therefore see that the equation satisfied by $ \partial_x X^{\theta,[\theta]}_t \gamma + \widehat{\EE}\left[ \partial_{\mu}  X_t^{\t\theta,[\theta]}(\widehat{\theta})\, \widehat{\gamma} \right] $ is the same as equation \eqref{eq:FrechMKValt} satisfied by $D \Xt_t(\gamma)$, so by uniqueness they are equal. This representation also makes clear the linearity and continuity of $\gamma \mapsto D \Xt_t(\gamma)$.
                
                Following essentially the same procedure shows that
                 $\widehat{\EE} \left[\partial_{\mu} \Xd_t(\widehat{\theta}) \, \widehat{\gamma} \right]$ satisfies the same equation as $D \Xd_t(\gamma)$, so that \eqref{eq:Frechrep} holds. Hence, by definition $\partial_{\mu} \Xd_t(v)$ exists and satisfies equation \eqref{eq:partialmuX}. This representation also makes clear the linearity and continuity of $\gamma \mapsto D \Xd_t(\gamma)$.
                
                \item[(c)]
                Let $X^{\theta, n}$ denote the Picard approximation of the solution to the McKean-Vlasov SDE \eqref{eq:MKVSDE}, given by
                \begin{align*}
                        X^{\theta, 0}_t & =\theta, \quad t \in [0,T] \\
                        X^{\theta, n}_t &=  \theta +  \sum_{i=0}^d \int_0^t V_i \left(X^{\theta, n}_s, \left[X^{\theta, n-1}_s \right] \right) \, dB^i_s ,
                        \end{align*}    
                For each $n \geq 1$, $X^{\theta, n}$ is the solution of a classical SDE with time-dependent coefficients, which are differentiable in space, with each derivative of the coefficients being Lipschitz continuous. Therefore, by Nualart \cite{Nualart} Theorem
                2.2.1 $X^{\theta, n}_t \in \DD^{1, \infty}$ for all $t \in [0,T]$. The form of the equation satisfied by $\D X^{\theta, n}_t$ is the same as \eqref{eq:mall}. It is then easy to show that $ \|X^{\theta, n}_t\|_{\DD^{1, \infty}} < C(1  + \| \theta \|_2)$ uniformly in $n$. Now, since for all $p \geq 2$, $ \|X^{\theta, n}_t - X^{\theta}_t\|_p \to 0$ as $n \to \infty$, by Nualart  \cite{Nualart} Lemma 1.5.3, $ X^{\theta}_t \in \DD^{1, \infty}$. Similarly, $\Xd_t \in \DD^{1, \infty}$ since it solves a classical SDE with time-dependent coefficients.
                The measure term in the coefficients of the equation for $\Xd_t$ is deterministic, so $\D_r(\Xd_t)$ satisfies the usual equation for the Malliavin derivative of an SDE which is precisely equation \eqref{eq:mall}.
        \end{itemize}
\end{proof}
For our aplications, we need to extend the above result to higher order derivatives of $X^{x, [\theta]}_t$. The main result is summarised in the following theorem, which classifies $X^{x, [\theta]}_t$ as a Kusuoka-Stroock process.

\begin{theorem}
        \label{th:XisKSP}
        Suppose $\Vs \in \C^{k,k}_{b, Lip}(\RR^N \times \p_2(\RR^N);\RR^N)$, then $(t, x, [\theta]) \mapsto X_t^{x, [\theta]} \in \KK^1_{0}(\RR^N,k)$. If, in addition, $\Vs$ are uniformly bounded then $(t, x, [\theta]) \mapsto X_t^{x, [\theta]} \in \KK^0_{0}(\RR^N,k)$.
\end{theorem}
Since each derivative process satisfies a linear equation (whose exact form is not important for our purposes) the proof is quite mechanical and reserved to the Appendix \ref{sec:KSpfs}.
Now we introduce some operators acting on Kusuoka-Stroock processes. These are the building blocks of the integration by parts formulae to come.
For the rest of this section, we will need the following uniform ellipticity assumption.
\begin{assumption}[UE]
        \label{ass:UE}
        Let $\sigma: \RR^N \times \p_2(\RR^N) \to \RR^{N \times d}$ be given by
        \[
        \sigma(z,\mu):= \left[ V_1(z,\mu) | \cdots | V_d(z, \mu) \right] .
        \]
        We make the assumption that there exists $\epsilon>0$ such that, for all $\xi \in \RR^N$, $z \in \RR^N$ and $\mu \in \p_2(\RR^N)$,
        \[
        \xi^{\top} \sigma(z, \mu) \sigma(z, \mu)^{\top} \xi \geq \epsilon | \xi |^2 .
        \]
\end{assumption}
Now, for a multi-index $\alpha$ on $\{1, \ldots, N\}$, we introduce the following operators acting on elements of $\KK^q_r(\RR,n)$, defined for $\alpha=(i)$, by
\begin{align*}
        & I^1_{(i)}(\Psi)(t,x, [\theta]) :=  \frac{1}{\sqrt{t}} \, \delta \left(r \mapsto \Psi(t,x, [\theta]) \,   \left( \sigma^{\top} (\sigma \sigma^{\top})^{-1} (X^{x,\mu}_r, [X_r^{\theta}]) \partial_x X^{x,\mu}_r \right)_{i} \right),  \\
        & I^2_{(i)}(\Psi)(t,x, [\theta]) := \sum_{j=1}^N I^1_{(j)} \left((\partial_x X^{x,\mu}_t)^{-1}_{j,i} \Psi(t,x, [\theta])\right), \\
        & I^3_{(i)}(\Psi)(t,x, [\theta]) := I^1_{(i)}(\Psi)(t,x,[\theta]) + \sqrt{t} \partial^i\Psi(t,x,[\theta]), \\
        & \I^1_{(i)} (\Psi)(t,x, [\theta],v_1) :=  \\ 
        & \phantom{XXXX} \frac{1}{\sqrt{t}} \, \delta\left( r \mapsto \left( \sigma^{\top} (\sigma \sigma^{\top})^{-1} (X^{x,\mu}_r, [X_r^{\theta}]) \partial_x X^{x,\mu}_r \, (\partial_x X^{x,\mu}_t)^{-1}  \partial_{\mu} \Xd_t(v_1) \right)_{i} \Psi(t,x, [\theta]) \right),  \\
        & \I^3_{(i)}(\Psi)(t,x, [\theta],v_1) := \I^1_{(i)}(\Psi)(t,x, [\theta],v_1) +  \sqrt{t} (\partial_{\mu} \Psi)_{i}(t,x, [\theta],v_1).
\end{align*}
For $\alpha = (\alpha_1, \ldots, \alpha_n)$ we inductively define
\begin{align*}
        & I^1_{\alpha}:= I^1_{\alpha_n} \circ I^1_{\alpha_{n-1}} \circ \cdots \circ I^1_{\alpha_1} ,
\end{align*}
and make analogous definitions for each of the other operators. The following result states that these operators are well-defined and describes how each operator transforms a given Kusuoka-Stroock process. The proof is contained in Appendix \ref{sec:KSpfs}.

\begin{proposition}
        \label{prop:defwieghts}
        If $\Vs \in \C^{k,k}_{b,Lip}(\RR^N \times \p_2(\RR^N);\RR^N)$, (UE) holds and $\Psi \in \KK^q_r(\RR,n)$, then $I^1_{\alpha}(\Psi)$ and $I^3_{\alpha}(\Psi)$, are all well-defined for $|\alpha|\leq (k \wedge n)$. $I^2_{\alpha}(\Psi)$, $\I^1_{\alpha}(\Psi)$ and $\I^3_{\alpha}(\Psi)$ are well defined for $|\alpha|\leq n \wedge (k-2)$. Moreover,
        \begin{align*}
                I^1_{\alpha}(\Psi),  I^3_{\alpha}(\Psi) & \in  \KK^{q+2|\alpha|}_r(\RR,(k \wedge n)-|\alpha|) , \\
                I^2_{\alpha}(\Psi) & \in \KK^{q+3|\alpha|}_r(\RR,[n \wedge (k-2)]-|\alpha|) , \\
                \I^1_{\alpha}(\Psi), \I^3_{\alpha}(\Psi) & \in \KK^{q+4|\alpha|}_r(\RR,[n \wedge (k-2) ]-|\alpha|).
        \end{align*}
        If $\Psi \in \KK^0_r(\RR,n)$ and $\Vs$ are uniformly bounded, then
        \begin{align*}
                I^1_{\alpha}(\Psi),  I^3_{\alpha}(\Psi) & \in  \KK^{0}_r(\RR,(k \wedge n)-|\alpha|) , \\
                I^2_{\alpha}(\Psi) & \in \KK^{0}_r(\RR,[n \wedge (k-2)]-|\alpha|) , \\
                \I^1_{\alpha}(\Psi), \I^3_{\alpha}(\Psi) & \in \KK^{0}_r(\RR,[n \wedge (k-2) ]-|\alpha|).
        \end{align*}
        
\end{proposition}

% % % % % % % % % % % % % % % % % % % % % % % % % % % % % % % %
% % % % % % % % % % % % % % % % % % % % % % % % % % %
\section{Integration by parts formulae for the de-coupled equation}
%%%%%%%%%%%% % % % % % % % % % % % % % % % % % % % % % % %
% % % % % % % % % % % % % % % % % % % % % % % % % % % % % % %
\label{sec:IBPdecoupled}

Having introduced some operators acting on Kusuoka-Stroock processes, we now show how to use these operators to construct Malliavin weights in integration by parts formulas. We first develop integration by parts formulas for derivatives of $x \mapsto \EE \, f(\Xd_t)$ and then separately $[\theta] \mapsto\EE \, f(\Xd_t)$. In the last part of this section, we will show how to combine these results to construct integration by parts formulas for derivatives of the function $x \mapsto \EE \,f(X^{x,\delta_x}_t)$.

\subsection{Integration by parts in the space variable}
\label{sec:IBPFinx}
%%%%%%%%%%%%

\begin{proposition}
        \label{th:IBPx1}
        Let $f \in \C^{\infty}_b(\RR^N;\RR)$ and $\Psi \in \KK^q_r(\RR,n)$
        \begin{enumerate} 
        \item
        If $|\alpha| \leq [n \wedge k]$, then
        \[
        \EE[\partial^{\alpha}_x(f(\Xd_t))\, \Psi(t,x, [\theta])] = t^{-|\alpha|/2} \, \EE [f(\Xd_t)\, I^1_{\alpha}(\Psi)(t,x, [\theta])] .
        \]
        \item If $|\alpha| \leq [n \wedge (k-2)]$, then
                \[
                \EE[(\partial^{\alpha}f)(\Xd_t)\, \Psi(t,x, [\theta])] = t^{-|\alpha|/2} \, \EE [f(\Xd_t)\, I^2_{\alpha}(\Psi)(t,x, [\theta])] .
                \]
                \item If $|\alpha| \leq [n \wedge k]$, then
                        \[
                        \partial^{\alpha}_x \,  \EE[f(\Xd_t)\, \Psi(t,x, [\theta])] = t^{-|\alpha|/2} \, \EE [f(\Xd_t)\, I^3_{\alpha}(\Psi)(t,x, [\theta])] .
                        \]
        \item
        If $|\alpha| + |\beta| \leq [n \wedge (k-2)]$, then
                \[
                \partial^{\alpha}_x \,  \EE[(\partial^{\beta} f)(\Xd_t)\, \Psi(t,x, [\theta])] = t^{-(|\alpha|+ |\beta|)/2} \, \EE [f(\Xd_t)\, I^3_{\alpha}\left(I^2_{\beta}(\Psi)\right)(t,x, [\theta])] .
                \]
        \end{enumerate}
\end{proposition}
\begin{proof}
\begin{enumerate}
\item
        First, we note that equation  \eqref{eq:jac} satisfied by $\partial_x X^{x,[\theta]}_t$ and equation \eqref{eq:mall} satisfied by  $\D_r X^{x,[\theta]}_t$ are the same except their initial conditions. It therefore follows that for $r \leq t$,
        \[
        \partial_x X^{x,[\theta]}_t = \D_r X^{x,[\theta]}_t \sigma^{\top} (\sigma \sigma^{\top})^{-1} (X^{x,[\theta]}_r, [X_r^{\theta}]) \partial_x X^{x,[\theta]}_r .
        \]
        This allows us to make the following computations for $f \in \C^{\infty}_b(\RR^N;\RR)$,
        \begin{align*}
                &\EE \left[ \partial_x [f(\Xd_t)] \Psi(t,x, [\theta]) \right]\\
                = & \EE \left[ \partial f (X_t^{x,[\theta]} ) \,\partial_x X^{x,[\theta]}_t\, \Psi(t,x, [\theta]) \right] \\
                = & \frac{1}{t} \EE \left[ \int_0^t \partial f (X_t^{x,[\theta]}) \, \partial_x X^{x,[\theta]}_t  \Psi(t,x, [\theta]) \, dr \right] \\
                = & \frac{1}{t} \EE \left[ \int_0^t  \partial f (X_t^{x,[\theta]}) \,  \D_r X^{x,[\theta]}_t  \sigma^{\top} (\sigma \sigma^{\top})^{-1} (X^{x,[\theta]}_r, [X_r^{\theta}]) \partial_x X^{x,[\theta]}_r  \Psi(t,x, [\theta]) \, dr \right] \\
                = & \frac{1}{t} \EE \left[ \int_0^t  \D_r f (X_t^{x,[\theta]})  \, \sigma^{\top} (\sigma \sigma^{\top})^{-1} (X^{x,[\theta]}_r, [X_r^{\theta}]) \partial_x X^{x,[\theta]}_r \, \Psi(t,x, [\theta]) \, dr \right]        \\
                = & \frac{1}{t} \EE \left[  f (X_t^{x,[\theta]}) \,  \delta\left( r \mapsto \left( \sigma^{\top} (\sigma \sigma^{\top})^{-1} (X^{x,[\theta]}_r, [X_r^{\theta}]) \partial_x X^{x,[\theta]}_r   \right)^{\top} \Psi(t,x, [\theta]) \right) \right],
        \end{align*}
        where we have used Malliavin integration by parts
        $
        \EE \langle \D \phi, u \rangle_{H_d} = \EE \left[ \phi \, \delta(u) \right]
        $ in the last line. This proves the result for $|\alpha|=1$. By Proposition \ref{prop:defwieghts}, $I^1_{\alpha}(\Psi) \in  \KK^{q+2}_r(\RR,(k \wedge n)-1) $ when $|\alpha|=1$. We can therefore iterate this argument another $|\alpha|-1$ times to obtain the result for all $\alpha$ satisfying $|\alpha| \leq [n \wedge k]$.
        \item
                By the chain rule, 
                \begin{align*}
                        \EE[(\partial^{i} f)(\Xd_t) \Psi(t,x, [\theta])] 
                        & = \sum_{j=1}^N \EE[\partial_{x_i} (f(\Xd_t)) \left( (\partial_x X^{x,[\theta]}_t)^{-1} \right)^{j,i} \Psi(t,x, [\theta])]  \\
                        & = t^{-1/2} \, \sum_{j=1}^N \EE \left[f(\Xd_t) I^1_{(j)} \left(\left((\partial_x X^{x,[\theta]}_t)^{-1} \right)^{j,i} \Psi(t,x, [\theta])\right) \right] \\
                        & = t^{-1/2} \,\EE [f(\Xd_t)\, I^2_{(i)}(\Psi)(t,x, [\theta])] .                                                            
                \end{align*}
                By Proposition \ref{prop:defwieghts}, $I^2_{(i)}(\Psi) \in \KK^{q+3}_r \left(\RR,[n \wedge (k-2)]-1 \right)$, so since  $|\alpha| \leq [n \wedge (k-2)]$, we can apply this argument another $|\alpha|-1$ times to get the result.
                \item
                        We compute, for any $i=1, \ldots, N$ 
                        \begin{align*}
                                \partial^{i}_x \, \EE[ f(\Xd_t) \Psi(t,x, [\theta])] & = \EE[\partial^i_x (f(\Xd_t)\Psi(t,x, [\theta]) + \partial^i_x\Psi(t,x, [\theta]) f(\Xd_t)  ] \\
                                & = t^{-1/2} \EE \left[ f(\Xd_t) \left\{ I^1_{(i)}(\Psi)(t,x, [\theta]) + \sqrt{t} \partial_x^i\Psi(t,x, [\theta])  \right\} \right],
                        \end{align*}
                        which proves the result for $|\alpha|=1$.
                        Again, using  Proposition \ref{prop:defwieghts}, $I^3_{\alpha}(\Psi) \in  \KK^{q+2}_r(\RR,(k \wedge n)-1) $ when $|\alpha|=1$. We can therefore iterate this argument another $|\alpha|-1$ times to obtain the result for all $\alpha$ satisfying $|\alpha| \leq [n \wedge k]$.
                        \item This follows from parts 2 and 3.
                \end{enumerate}
\end{proof}

% % % % % % % % % % % % % % % % % % %%%%%%%%%%%%%
\subsection{Integration by parts in the measure variable}
\label{sec:IBPFinmu}
% % % % % % % % % % % % % % % % %%%%%%%%%%%%%

We now consider derivatives of the function
$
[\theta] \mapsto \EE [f (\Xd_t)] .
$

\begin{proposition}
        \label{th:IBPmu1}
        Let $f \in \C^{\infty}_b(\RR^N;\RR)$ and $\Psi \in \KK^q_r(\RR,n)$.
        \begin{enumerate}
        \item If $|\beta| \leq [n \wedge (k-2)]$, then
        \[
        \EE[\partial^{\beta}_{\mu}(f(\Xd_t))(\bv)\, \Psi(t,x, [\theta])] = t^{-|\beta|/2} \,  \EE [f(\Xd_t)\, \I^1_{\beta}(\Psi)(t,x, [\theta], \bv)] .
        \]
        %Moreover, if $\Psi \in \KK_r(k)$ and $V_i$ are uniformly bounded, then
        %\[
        %       \sup_{x \in \RR^N} \left|       \EE \left[\partial^{\gamma}_x(f(\Xd_t))\, \Psi(t,x, [\theta]) \right] \right| \leq C \, t^{(r-|\gamma|)/2}
        %\]
        \item If $|\beta| \leq[n \wedge (k-2)]$, then
                \[
                \partial^{\beta}_{\mu} \,       \EE[f(\Xd_t)\, \Psi(t,x, [\theta])](\bv) = t^{-|\beta|/2} \,  \EE [f(\Xd_t)\, \I^3_{\beta}(\Psi)(t,x, [\theta], \bv)] .
                \]
                \item
                If $|\alpha| + |\beta| \leq [n \wedge (k-2)]$,  then
                        \[
                        \partial^{\beta}_{\mu} \,       \EE[(\partial^{\alpha}f)(\Xd_t)\, \Psi(t,x, [\theta])](\bv) = t^{-(|\alpha|+|\beta|)/2}  \EE [f(\Xd_t)\, \I^3_{\beta}\left(I^2_{\alpha}(\Psi)\right)(t,x, [\theta], \bv)] .
                        \]
                        %Moreover, if $\Psi \in \KK_r(k)$ and $V_i$ are uniformly bounded, then
                        %\[
                        %       \sup_{x \in \RR^N} \left|\partial^{\gamma}_x \,      \EE \left[f(\Xd_t)\, \Psi(t,x, [\theta]) \right] \right| \leq C \, t^{(r-|\gamma|)/2}
                        %\]
        \end{enumerate}
\end{proposition}
\begin{proof}
\begin{enumerate}
\item
        We use again that for $r \leq t$,
        \[
        \partial_x X^{x,[\theta]}_t = \D_r X^{x,[\theta]}_t \sigma^{\top} (\sigma \sigma^{\top})^{-1} (X^{x,[\theta]}_r, [X_r^{\theta}]) \partial_x X^{x,\mu}_r .
        \]
        This allows us to make the following computations for $f \in \C^{\infty}_b(\RR^N;\RR)$,
        \begin{align*}
                & \EE \left[ \partial_{\mu}(f(\Xd_t)) \Psi(t,x, [\theta]) \right] \\
                & = \EE \left[ \partial f (X_t^{x, [\theta]} ) \, \partial_{\mu} \Xd_t \, \Psi(t,x, [\theta]) \right] \\
                & = \frac{1}{t} \EE \left[ \int_0^t \partial f (X_t^{x, [\theta]}) \,  \,\partial_x X^{x, [\theta]}_t \, (\partial_x X^{x, [\theta]})^{-1}_t \partial_{\mu} \Xd_t(v) \Psi(t,x, [\theta]) \, dr \right] \\
                &= \frac{1}{t} \EE  \int_0^t \big\{ \partial f (X_t^{x, [\theta]}) \,  \D_r X^{x, [\theta]}_t  \sigma^{\top} (\sigma \sigma^{\top})^{-1} (X^{x, [\theta]}_r, [X_r^{\theta}]) \partial_x X^{x, [\theta]}_r  (\partial_x X^{x, [\theta]}_t)^{-1} 
         \partial_{\mu} \Xd_t(v) \Psi(t,x, [\theta]) \big\} dr  \\
                & = \frac{1}{t} \EE  \int_0^t \big\{  \D_r f (X_t^{x, [\theta]})  \, \sigma^{\top} (\sigma \sigma^{\top})^{-1} (X^{x, [\theta]}_r, [X_r^{\theta}]) \partial_x X^{x, [\theta]}_r \, (\partial_x X^{x, [\theta]}_t)^{-1}  \partial_{\mu} \Xd_t(v) \Psi(t,x, [\theta]) \big\} dr  \\
                & = \frac{1}{t} \EE \big[  f (X_t^{x, [\theta]}) \,  \delta\big( r \mapsto \left( \sigma^{\top} (\sigma \sigma^{\top})^{-1} (X^{x, [\theta]}_r, [X_r^{\theta}]) \partial_x X^{x, [\theta]}_r \, (\partial_x X^{x, [\theta]}_t)^{-1}  \partial_{\mu} \Xd_t(v) \right)^{\top} \ \Psi(t,x, [\theta]) \big) \big] .
        \end{align*}
        where we have used Malliavin integration by parts
        $
        \EE \langle \D \phi, u \rangle_{H_d} = \EE \left[ \phi \, \delta(u) \right]
        $ in the last line.
        This proves the claim for $|\beta|=1$. For general $\beta$, it follows by iterating this integration by parts $|\beta|$ times.
        \item   \[
                \partial_{\mu} \,       \EE[f(\Xd_t)\, \Psi(t,x, [\theta])](v) = t^{-|\beta|/2} \, \EE [\partial_{\mu} (f(\Xd_t))(v) \, \Psi(t,x, [\theta]) + f(\Xd_t) \, \partial_{\mu}\Psi(t,x, [\theta],v)] .
                \]
                This is enough to prove the proposition when $|\beta|=1$. For $|\beta|>1$, simply repeat this argument.
        \item This follows from parts 1 and 2.
        \end{enumerate}
\end{proof}

%%%%%%%%%%%%
\subsection{Integration by parts for McKean-Vlasov SDE with fixed initial condition }
\label{sec:MKVIBP}
%%%%%%%%%%%%
We now consider developing integration by parts formulae for derivatives of the function
\[
x \mapsto \EE f(X_t^{x, \delta_x}).
\]
We introduce the following operator acting on elements of $\K_r^q(\RR, M)$, the set of Kusuoka-Stroock processes on $\RR^N$. For $\alpha=(i)$
\[
J_{(i)}(\Phi)(t,x):= I^3_{(i)}(\Phi)(t,x,\delta_x) +  \I^3_{(i)}(\Phi)(t,x,\delta_x)
\]
and inductively, for $\alpha=(\alpha_1, \ldots, \alpha_n)$,
\[
J_{\alpha}:= J_{\alpha_n} \circ J_{\alpha_1} \cdots \circ J_{\alpha_1}.
\]

\begin{lemma}
        \label{lem:Jprops}
        If $\Vs \in \C^{k,k}_{b,Lip}(\RR^N \times \p_2(\RR^N);\RR^N)$ and $\Phi \in \K^q_r(\RR,n)$, then $J_{\alpha}(\Phi)$ is well-defined for $|\alpha|\leq [n \wedge (k-2)]$, 
        and
        \[
        J_{\alpha}(\Phi) \in \K^{q+4|\alpha|}_r(\RR,[n \wedge (k-2)]-|\alpha|).
        \]
        Moreover, if $\Phi \in \K^0_r(\RR,k)$ and $\Vs$ are uniformly bounded, then
        \[
        J_{\alpha}(\Phi) \in \K^{0}_r(\RR,[n \wedge (k-2)]-|\alpha|).
        \]
\end{lemma}

\begin{proof}
        This is a direct result of Proposition \ref{prop:defwieghts} and Lemma \ref{lem:conversion}.
\end{proof}

\begin{theorem}
        \label{th:MKVIBP}
        Let $f \in \C^{\infty}_b(\RR^N;\RR)$.
        For all multi-indices $\alpha$ on $\{1, \ldots, N\}$ with $|\alpha| \leq k-2$
        \[
        \partial^{\alpha}_x \, \EE \left[f(X_t^{x,\delta_x}) \right] = t^{-|\alpha|/2} \, \EE \left[ f(X_t^{x,\delta_x}) \, J_{\alpha}(1)(t,x) \right].
        \]
        In particular, we get the following bound
        \[
        \left|  \partial^{\alpha}_x \, \EE \left[f(X_t^{x,\delta_x}) \right] \right| \leq C \, \|f \|_{\infty} \,  t^{-|\alpha|/2}  \, (1+|x|)^{4 |\alpha|} .
        \]
\end{theorem}
\begin{proof}
        
        By the above discussion,
        \[
        \partial^{i}_x \, \EE \left[f(X_t^{x, \delta_x}) \right] = \partial_z^{i} \, \EE \left. \left[f(X_t^{z, \delta_x}) \right] \right|_{z=x} + \partial^{i}_{\mu}  \EE \left. \left[f(X_t^{x,[\theta]}) \right](v) \right|_{ [\theta]=\delta_x,v=x}
        \]
        Now, we apply the IBPFs developed earlier in Proposition \ref{th:IBPx1}  part 3 and Theorem \ref{th:IBPmu1} part 3.
        \begin{align*}
                & \partial_z^{i} \, \EE \left. \left[f(X_t^{z, \delta_x}) \right] \right|_{z=x} = t^{-1/2} \, \EE  \left[ f(X_t^{x,\delta_x})  I^3_{(i)}(1)(t,x) \right] \\
                &\partial^{i}_{\mu}  \EE \left. \left[f(X_t^{x,[\theta]}) \right](v) \right|_{ [\theta]=\delta_x,v=x} =  t^{-1/2} \, \EE  \left[ f(X_t^{x,\delta_x})  \I^3_{(i)}(1)(t,x, \delta_x, x) \right]
        \end{align*}
        and we can iterate this argument $|\alpha|$ times.
        
        %The moment bound on $\Gamma_{\alpha}$ comes from noting that the constant function $1 \in \cap_{n \geq 1} \K_0^0(\RR, n)$ and using the moment bounds for $I^3_{\alpha}(1)$ and $\I^3_{\alpha}(1)$ coming from their membership of their respective Kusuoka-Stroock class.
\end{proof}

\begin{corollary}
        \label{cor:MKVIBP}
        Let $ f \in \C^{\infty}_b(\RR^N;\RR)$ and $\alpha$ and $\beta$ multi-indices on $\{1, \ldots, N\}$ with $|\alpha| + |\beta| \leq k-2$. Then, 
        \[
        \partial^{\alpha}_x \, \EE \left[(\partial^{\beta} f)(X_t^{x,\delta_x}) \right] = t^{-\frac{|\alpha|+|\beta|}{2}} \, \EE \left[ f(X_t^{x,\delta_x}) \,  I^2_{\beta}(J_{\alpha}(1))(t,x) \right]
        \]
        and $I^2_{\beta}(J_{\alpha}(1))\in \K^{4|\alpha|+3|\beta|}_0(\RR,k-2-|\alpha|-|\beta|)$.
        
\end{corollary}
\begin{proof}
        Theorem \ref{th:MKVIBP} gives
        \[
        \partial^{\alpha}_x \, \EE \left[(\partial^{\beta}f)(X_t^{x,\delta_x}) \right] = t^{-|\alpha|/2} \, \EE \left[ (\partial^{\beta}f)(X_t^{x,\delta_x}) \, J_{\alpha}(1)(t,x) \right]
        \]
        with $J_{\alpha}(1) \in \K^{4|\alpha|}(\RR, k-2-|\alpha|)$.
        Then, using Proposition \ref{th:IBPx1} part 2, we get 
        \[
        \partial^{\alpha}_x \, \EE \left[(\partial^{\beta} f)(X_t^{x,\delta_x}) \right] = t^{-\frac{|\alpha|+|\beta|}{2}} \, \EE \left[ f(X_t^{x,\delta_x}) \, I^2_{\beta}(J_{\alpha}(1))(t,x)  \right] .
        \]
        
\end{proof}

% % % % % % % % % % % % % % % % % % % % % % % % % % % % % % % % % % % % % % % % % % % % %
% % % % % % % % % % % % % % % % % % % % % % % % % % % % % % % % % % % % % % % % % % % % %
\section{Connection with PDE}
% % % % % % % % % % % % % % % % % % % % % % % % % % % % % % % % % % % % % % % % % % % % %
% % % % % % % % % % % % % % % % % % % % % % % % % % % % % % % % % % % % % % % % % % % % % 
\label{sec:masterPDE}

We return our attention to the PDE \eqref{eq:MasterPDE}. The results of the last section suggest that for initial conditions $g(z,\mu)=g(z)$, which do not depend on the measure, we can still expect there to be a classical solution, even if $g$ is not differentiable. Indeed, we spell out the conditions under which this is true in Theorem \ref{th:masterPDE}. But first, let us consider whether the same can be true for initial conditions which do depend on the measure.

\begin{example}
        \label{ex:counter}
        Let $g(z,\mu) = g(\mu) := \textstyle \left| \int y \, \mu(dy) \right|$ and $V_0 \equiv 0$, $V_1\equiv1$ and $N=d=1$, then 
        \[
        X_t^{\theta} = \theta  + B_t,
        \]
        and
        \[  g([X_t^{\theta}]) =  \left|  \EE[\theta]\right|.
        \]
        We now show that $[ \theta ] \mapsto  g([X_t^{\theta}])$ is not differentiable. If we choose $\theta \in L^2(\Omega)$ with $\EE\theta=0$, then for any $t>0$, $h>0$ and any $\gamma \in L^2(\Omega)$,
        \begin{align*}
                \frac{1}{ h} \left| g( [X_t^{\theta+ h \gamma}]) - g( [X_t^{\theta}]) \right|         & = \frac{|h|}{h} \, \left| \EE \gamma \right|
        \end{align*}
        and this limit does not exist as $h \to 0$. Hence, the G\^{a}teaux derivative of the map $L^2(\Omega) \ni \theta \mapsto g( [X_t^{\theta}])$ does not exist.
\end{example}

The above example shows that for a function $g:\RR^N \times \p_2(\RR^N) \to \RR$ which is Lipschitz continuous, we cannot, in general, expect $[\theta] \mapsto \EE \left( \, g \left(\Xd_t,\left[ \Xd_t \right]\right) \right)$ to be differentiable (for a fixed $t>0$) even when the coefficients in the equation for $\Xd_t$ are smooth and uniformly elliptic. There are, however, interesting examples of initial conditions for which we can develop integration by parts formulas.
Before we introduce this class of initial conditions, we consider what form derivatives of 
$
U(t,x,[\theta]):= \EE \left( g \left( \Xd_t, [\Xt_t] \right) \right)
$
take when $g$ is smooth. The following result is Lemma 5.1 from \cite{buckdahnmean}.
\begin{lemma}
        \label{lem:dmu}
        We assume that the function $g :\RR^N \times \p_2(\RR^N) \to \RR^N$ admits continuous derivatives $\partial_x g$ and $\partial_{\mu}g$ satisfying for some $q>0$ and $0 \leq p <2$
        \begin{align*}
                & \left| \partial_x g(x,[\theta]) \right| \leq C \left(1+|x|+\|\theta\|_2 \right)^q\\
                &  \left| \partial_\mu g(x,[\theta],v) \right| \leq C \left(1+|x|^q+\|\theta\|^q_2 + |v|^p \right)
        \end{align*} 
        and we assume $\Vs \in \C^{1,1}_{b,Lip}(\RR^N \times \p_2(\RR^N);\RR^N)$. 
        Then, $\partial_{\mu}U$ exists and takes the following form:
        \begin{align}
                \label{eq:dmuU}
                \begin{split}
                        & \partial_{\mu}U(t,x,[\theta],v) = \EE \left[ \partial g \left( \Xd_t, [\Xt_t] \right) \partial_{\mu} \Xd_t (v)   \right] \\
                        & \quad + \:  \EE \widetilde{\EE} \left[ \partial_{\mu} g \left( \Xd_t, [\Xt_t], \widetilde{X}^{v,[\theta]}_t \right) \partial_{v} \widetilde{X}_t^{v,[\theta]} +  \partial_{\mu} g \left( \Xd_t, [\Xt_t], \widetilde{X}^{\t \theta}_t \right) \partial_{\mu} \widetilde{X}_t^{\t \theta,[\theta]}(v)   \right].
                \end{split}
        \end{align}
\end{lemma}

Now we introduce a class of initial conditions $g: \RR^N \times \p_2(\RR^N) \to \RR$ for which we will be able to develop integration by parts formulas.

\begin{definition}[\textbf{(IC)$_x$} and \textbf{(IC)$_v$} ]
        \label{def:IC}
        
        We say that $g: \RR^N \times \p_2(\RR^N) \to \RR$ is in the class \textbf{(IC)} if the following conditions hold:
        \begin{enumerate}
                \item $g$ is continuous with polynomial growth: i.e. there exists $q>0$ such that for all $(x,[\theta]) \in \RR^N \times \p_2(\RR^N)$:  $|g(x,[\theta])| \leq C (1+ |x|+ \|\theta\|_2)^q$.
                \item There exists a sequence of functions $(g_l)_{l \geq 1}$, $g_l: \RR^N \times \p_2(\RR^N) \to \RR$ with polynomial growth such that $g_l \to g$ uniformly on compacts and $\partial_x g_l$ exists and also has polynomial growth for each $l \geq 1$.
                \item For each $l \geq 1$ there exists a function
                $G_l: \RR^N \times \p_2(\RR^N) \times \RR^N \to \RR$ which is either differentiable in $x$ or $v$ and  $\partial_{\mu} g_l(x,\mu,v) = \partial_x G_l(x,\mu,v)$ or $\partial_{\mu} g_l(x,\mu,v) = \partial_v G_l(x,\mu,v)$. Moreover, each $G_l$ and its derivatives satisfies the growth condition: there exist $q>0$ and $0\leq r <1$ such that for all $(x,[\theta],v) \in\RR^N \times \p_2(\RR^N) \times \RR^N$:
                \[
                |h(x,[\theta],v)| \leq C \left(1 + |x|^q + \| \theta\|_2^q + |v|^{r} \right).
                \]
                where $h$ is $G_l$, $\partial_x G_l$ or $\partial_v G_l$.
                In addition, we assume that for all $(x,\mu,v)$ the pointwise limit $\lim_{l \to \infty} G_l(x,\mu,v)$ exists and the function $G$ defined by $G(x,\mu,v):= \lim_{l \to \infty} G_l(x,\mu,v)$ is continuous and satisfies the same growth condition.
        \end{enumerate}
        If  $\partial_{\mu} g_l= \partial_x G_l$ we say $g$ is in the class \textbf{(IC)$_x$}. If $\partial_{\mu} g_l = \partial_v G_l$, we say $g$ is in the class \textbf{(IC)$_v$}.
\end{definition}
We give some examples of functions $g$ in the class (IC).

\begin{example}
        \label{ex:ICs}$\left.\right.$
        \begin{enumerate}
                %\item Convolutions: suppose $g(x,\mu) := (\varphi \ast \mu)(x)$ then $\partial_{\mu}g(x,\mu,v)= \partial \varphi(x-v)$. So $G$ in this case would be $G(x,\mu,v) = \varphi(x-v)$.
                
                \item Functions with no dependence on the measure: \newline
                Suppose that $g(x,\mu) = \varphi(x)$ where $\varphi \in \C_p(\RR^N;\RR)$. Then, let $(\varphi_l)_{l \geq 1}$ be a sequence of mollifications of $\varphi$ and $(g_l)_{l \geq 1}$ the corresponding functions defined in the same way. Then,  $\partial_{\mu}g_l(x,\mu,v)= 0$. So, $g$ belongs to the class  \textbf{(IC)$_x$} and $G$ in this case would be $G \equiv 0$.
                
                \item Centred random variables: \newline
                Suppose that $g(x,\mu) = \varphi\left(x- \textstyle\int y \mu(dy)\right)$ where $\varphi \in \C_p(\RR^N;\RR)$. Then, let $(\varphi_l)_{l \geq 1}$ be a sequence of mollifications of $\varphi$ and $(g_l)_{l \geq 1}$ the corresponding functions defined in the same way. Then,  $\partial_{\mu}g_l(x,\mu,v)= - \partial \varphi_l(x- \textstyle\int y \mu(dy) )$. So, $g$ belongs to the class  \textbf{(IC)$_x$} and $G$ in this case would be $G(x,\mu,v) = - \varphi(x-\textstyle\int y \mu(dy))$.
                
                %\item Scalar interaction: suppose $g(x,\mu) := \textstyle \int \varphi d \mu$ then $\partial_{\mu}g(x,\mu,v)= \partial \varphi(v)$. So $G$ in this case would be $G(x,\mu,v) = \varphi(v)$.
                
                \item First order interaction: \newline
                Suppose $g(x,\mu) := \textstyle \int \varphi(x,y) \mu(dy)$ where $\varphi: \RR^N \times \RR^N \to \RR$ is continuous with $|\varphi(x,y)| \leq C(1+ |x|^q + |y|^r)$ for some $q>0$ and $0 \leq r <1$. Then, let $(\varphi_l)_{l \geq 1}$ be a sequence of mollifications of $\varphi$ and $(g_l)_{l \geq 1}$ the corresponding functions defined in the same way. Then, $\partial_{\mu}g_l(x,\mu,v)= \partial_v \varphi_l(x,v)$. So, $g$ belongs to the class  \textbf{(IC)$_v$} and $G$ in this case would be $G(x,\mu,v) = \varphi(x,v)$. Note, this example includes the case of convolutions where $\varphi(x,y)= \varphi(x-y)$.
                
                \item Second order interaction: \newline
                Suppose $g(x,\mu) := \textstyle \int \varphi(x,y,z) \mu(dy) \mu(dz)$  where $\varphi: \RR^{3N} \to \RR$ is continuous with $|\varphi(x,y,z)| \leq C(1+ |x|^q + |y|^r + |z|^r)$ for some $q>0$ and $0 \leq r <1$. Then, let $(\varphi_l)_{l \geq 1}$ be a sequence of mollifications of $\varphi$ and $(g_l)_{l \geq 1}$ the corresponding functions defined in the same way. Then, $\partial_{\mu}g_l(x,\mu,v)= \textstyle \int \left[ \partial_v \varphi_l(x,v,y) + \partial_v \varphi_l (x,y,v) \right] \mu(dy)$. So, $g$ belongs to the class  \textbf{(IC)$_v$} and $G$ in this case would be
                \[
                G(x,\mu,v) = \textstyle \int \left[ \varphi(x,v,y) + \varphi (x,y,v) \right] \mu(dy).
                \] 
                
                \item Polynomials on the Wasserstein space: \newline
                Suppose
                $g(x,\mu) = \textstyle \prod_{i=1}^n \int \varphi_i(x,y) \mu (dy) $, where $n \geq 1$ and each $\varphi_i: \RR^N \times \RR^N \to \RR$ is continuous with $|\varphi_i(x,y)| \leq C(1+ |x|^q )$ for some $q>0$. Then, let $(\varphi_{i,l})_{l \geq 1}$ be a sequence of mollifications of $\varphi_i$ and $(g_l)_{l \geq 1}$ the corresponding functions defined in the same way. Then, 
                \[
                \partial_{\mu}g_l(x,\mu,v)= \sum_{j=1}^n \prod_{i=1,i \neq j}^n \left( \int \varphi_{i,l}(x,y) \mu (dy) \right) \partial_v \varphi_{j,l}(x,v).
                \]
                Therefore $g$ belongs to the class  \textbf{(IC)$_v$} and $G$ in this case would be
                \[
                G(x,\mu,v)= \sum_{j=1}^n \prod_{i=1,i \neq j}^n \left( \int \varphi_i(x,y) \mu (dy) \right)  \varphi_j(x,v).
                \]
        \end{enumerate}
        
\end{example}

Now, we introduce the hypotheses under which we will be able to prove existence and uniqueness of a solution to the PDE \eqref{eq:MasterPDE}.
\begin{itemize}
        \item[\textbf{(H1):}] (UE) holds, and the coefficients $\Vs \in \C^{3,3}_{b,Lip}(\RR^N \times \p_2(\RR^N); \RR^N)$, and $g:\RR^N \times \p_2(\RR^N) \to \RR^N$ is in the class \textbf{(IC)$_x$}.
        \item[\textbf{(H2):}] (UE) holds, and the coefficients $\Vs \in \C^{3,3}_{b,Lip}(\RR^N \times \p_2(\RR^N); \RR^N)$ as well as being uniformly bounded, and that
        $g:\RR^N \times \p_2(\RR^N) \to \RR^N$ is in the class \textbf{(IC)$_v$}.
\end{itemize}

\begin{lemma}
        \label{lem:ctyU}
        Under either \textbf{(H1)} or \textbf{(H2)}, for the function $U(t,x,[\theta]):= \EE \left[g\left(\Xd_t \left[\Xt_t\right]\right)\right]$, the derivative functions
        \begin{align*}
                (0,T]\times\RR^N\times\p_2(\RR^N) \ni(t,x,[\theta]) & \mapsto \left( \partial_x U(t,x,[\theta]), \, \partial^2_{x,x}U(t,x,[\theta]) \right) \\ 
                (0,T]\times\RR^N\times\p_2(\RR^N)\times \RR^N \ni(t,x,[\theta],v) & \mapsto \left(  \partial_{\mu}U(t,x,[\theta],v), \, \partial_v \partial_{\mu}U(t,x,[\theta],v) \right)
        \end{align*}
        exist and are continuous. Moreover, for all compacts $K \subset \p_2(\RR^N)$
        \[
        \sup_{[\theta] \in K} \EE \left| \partial_{\mu}U(t,x,[\theta],\theta)\right|^2 + \left| \partial_v \partial_{\mu}U(t,x,[\theta],\theta) \right|^2 < \infty.
        \]
\end{lemma}
\begin{proof}
        Under both \textbf{(H1)} and \textbf{(H2)}, $g$ is in the class \textbf{(IC)}, so there is a sequence of functions $(g_l)_{l \geq 1}$ approximating $g$. Let $U_l(t,x,[\theta]= \EE \left[g_l(\Xd_t,[\Xt_t])\right]$ .
        From Proposition \ref{th:IBPx1} we know that for $i,j \in \{1, \ldots, N\}$ 
        %  there exist Kusuoka Stroock functions $\Phi_i, \Phi_{i,j} \in \KK^2_0(\RR,0)$ for some $q>0$ such that, for all $t \in (0,T]$,
        \begin{align*}
                \partial^i_x U_l(t,x,[\theta])&= t^{-1/2} \, \EE \left[ g_l(\Xd_t,[\Xt_t]) I^1_{(i)}(1)(t,x,[\theta]) \right], \\
                \partial^{(i,j)}_x U_l(t,x,[\theta]) &= t^{-1} \,  \EE \left[ g_l(\Xd_t,[\Xt_t])I^1_{(i,j)}(1)(t,x,[\theta]) \right].
        \end{align*}
        By the growth assumption on $g_l$, H\"older's inequality and the moment estimates already obtained for the processes $\Xd_t,\Xt_t$ and the Kusuoka-Stroock processes in \eqref{eq:MKVmoment}, \eqref{eq:decoupledmoment} and Proposition \ref{prop:KSprocesses}, we can show that the expectations above are bounded independently of $l \geq 1$. By dominated convergence, we can take the limit in each equation.
        Now, each of the Kusuoka-Stroock processes appearing in the above representations for the derivatives are, by definition, jointly continuous in $(t,x,[\theta])$ in $L^p(\Omega)$, $p \geq 1$. So is $(t,x,[\theta]) \mapsto g(\Xd_t,[\Xt_t])$ by Theorem \ref{th:XisKSP} (which guarantees that $(t,x,[\theta]) \mapsto \Xd_t$ is a Kusuoka-Stroock process) and the continuity of $g$.
        
        To lighten notation, we restrict to the case $N=1$ through the rest of this proof. First, we assume \textbf{(H1)} holds, so $g$ is in the class \textbf{(IC)$_x$}. Note that $g_l$ satisfies the hypotheses of Lemma \ref{lem:dmu}, which gives
        \begin{align}
                \label{eq:Ux}
                \begin{split}
                        & \partial_{\mu}U_l(t,x,[\theta],v) = \EE \left[ \partial g_l \left( \Xd_t, [\Xt_t] \right) \partial_{\mu} \Xd_t (v)   \right] \\
                        & \quad + \:  \EE \widetilde{\EE} \left[ \partial_{x} G_l \left( \Xd_t, [\Xt_t], \widetilde{X}^{v,[\theta]}_t \right) \partial_{v} \widetilde{X}_t^{v,[\theta]} +  \partial_{x} G_l \left( \Xd_t, [\Xt_t], \widetilde{X}^{\t \theta}_t \right) \partial_{\mu} \widetilde{X}_t^{\t \theta,[\theta]}(v)   \right].
                \end{split}
        \end{align}
        Now, we recall the following identity connecting $\D_r X^{x,[\theta]}_t$ and $ \partial_x X^{x,[\theta]}_r$:
        \[
        Id_N = \D_r X^{x,[\theta]}_t \sigma^{\top} (\sigma \sigma^{\top})^{-1} (X^{x,[\theta]}_r, [X_r^{\theta}]) \partial_x X^{x,[\theta]}_r \left(   \partial_x X^{x,[\theta]}_t\right)^{-1}.
        \]
        So,
        \begin{align*}
                &       \partial_{x} G_l \left( \Xd_t, [\Xt_t], \widetilde{X}^{v,[\theta]}_t \right) \partial_{v} \widetilde{X}_t^{v,[\theta]} \\
                = \: & \partial_{x} G_l \left( \Xd_t, [\Xt_t], \widetilde{X}^{v,[\theta]}_t \right) \D_r X^{x,[\theta]}_t \sigma^{\top} (\sigma \sigma^{\top})^{-1} (X^{x,[\theta]}_r, [X_r^{\theta}]) \partial_x X^{x,[\theta]}_r \left(      \partial_x X^{x,[\theta]}_t\right)^{-1}  \partial_{v} \widetilde{X}_t^{v,[\theta]} \\
                = \: & \D_r \left[G_l \left( \Xd_t, [\Xt_t], \widetilde{X}^{v,[\theta]}_t \right) \right] \sigma^{\top} (\sigma \sigma^{\top})^{-1} (X^{x,[\theta]}_r, [X_r^{\theta}]) \partial_x X^{x,[\theta]}_r \left(      \partial_x X^{x,[\theta]}_t\right)^{-1}  \partial_{v} \widetilde{X}_t^{v,[\theta]}
        \end{align*}
        and, applying Proposition \ref{th:IBPx1} part 2, we get
        \begin{align*}
                &       \EE \widetilde{\EE} \left[ \D_r \left[G_l \left( \Xd_t, [\Xt_t], \widetilde{X}^{v,[\theta]}_t \right) \right] \sigma^{\top} (\sigma \sigma^{\top})^{-1} (X^{x,[\theta]}_r, [X_r^{\theta}]) \partial_x X^{x,[\theta]}_r \left( \partial_x X^{x,[\theta]}_t\right)^{-1}  \partial_{v} \widetilde{X}_t^{v,[\theta]} \right] \\
                = \: &  t^{-1/2} \,     \EE \widetilde{\EE} \left[ G_l \left( \Xd_t, [\Xt_t], \widetilde{X}^{v,[\theta]}_t \right) I^2(1)(t,x,[\theta]) \,  \partial_{v} \widetilde{X}_t^{v,[\theta]} \right].
        \end{align*}
        Similarly,
        \begin{align*}
                &        \partial_{x} G_l \left( \Xd_t, [\Xt_t], \widetilde{X}^{\t \theta}_t \right) \partial_{\mu} \widetilde{X}_t^{\t \theta,[\theta]}(v)\\
                = \: & \partial_{x} G_l \left( \Xd_t, [\Xt_t], \widetilde{X}^{\t \theta}_t \right)\D_r X^{x,[\theta]}_t \sigma^{\top} (\sigma \sigma^{\top})^{-1} (X^{x,[\theta]}_r, [X_r^{\theta}]) \partial_x X^{x,[\theta]}_r \left(   \partial_x X^{x,[\theta]}_t\right)^{-1}   \partial_{\mu} \widetilde{X}_t^{\t \theta,[\theta]}(v) \\
                = \: & \D_r \left[G_l \left( \Xd_t, [\Xt_t],\widetilde{X}^{\t \theta}_t \right) \right] \sigma^{\top} (\sigma \sigma^{\top})^{-1} (X^{x,[\theta]}_r, [X_r^{\theta}]) \partial_x X^{x,[\theta]}_r \left(      \partial_x X^{x,[\theta]}_t\right)^{-1}   \partial_{\mu} \widetilde{X}_t^{\t \theta,[\theta]}(v)
        \end{align*}
        and applying Proposition \ref{th:IBPx1} part 2 again, we get
        \begin{align*}
                &       \EE \widetilde{\EE} \left[ \D_r \left[G_l \left( \Xd_t, [\Xt_t],\widetilde{X}^{\t \theta}_t \right) \right] \sigma^{\top} (\sigma \sigma^{\top})^{-1} (X^{x,[\theta]}_r, [X_r^{\theta}]) \partial_x X^{x,[\theta]}_r \left( \partial_x X^{x,[\theta]}_t\right)^{-1}   \partial_{\mu} \widetilde{X}_t^{\t \theta,[\theta]}(v) \right] \\
                = \: & t^{-1/2} \, \EE \widetilde{\EE} \left[ G_l \left( \Xd_t, [\Xt_t],\widetilde{X}^{\t \theta}_t \right) \, I^2(1)(t,x,[\theta]) \, \partial_{\mu} \widetilde{X}_t^{\t \theta,[\theta]}(v) \right].
        \end{align*}
        So, in this case, \eqref{eq:Ux} can be rewritten as
        \begin{align}
                \label{eq:dmuIBPx}
                \begin{split}
                        \partial_{\mu}U_l(t,x,[\theta],v)  = t^{-1/2} \, \EE \bigg\{ & g_l \left( \Xd_t, [\Xt_t] \right) \I^1(t,x,[\theta],v)  + \widetilde{\EE} \big[ G_l \left( \Xd_t, [\Xt_t], \widetilde{X}^{v,[\theta]}_t \right)I^2(1)(t,x,[\theta]) \,  \partial_{v} \widetilde{X}_t^{v,[\theta]} \\
                        & \quad \quad +  G_l \left( \Xd_t, [\Xt_t],\widetilde{X}^{\t \theta}_t \right) \, I^2(1)(t,x,[\theta]) \, \partial_{\mu} \widetilde{X}_t^{\t \theta,[\theta]}(v) \big] \bigg\}.
                \end{split}
        \end{align}
        To show that $\textstyle \sup_{[\theta] \in K} \EE \left| \partial_{\mu}U(t,x,[\theta],\theta)\right|^2 < \infty$, we note that all processes on the right hands side of \eqref{eq:dmuIBPx} have moments of all orders bounded polynomially in $\|\theta\|_2$ except $\widetilde{X}^{\t \theta}_t$ in the final term. For the final term, by the growth conditions on $G_l$,
        \begin{align*}
                & \left| \EE \widetilde{\EE} \left[G_l \left( \Xd_t, [\Xt_t],\widetilde{X}^{\t \theta}_t \right) \, I^2(1)(t,x,[\theta]) \, \partial_{\mu} \widetilde{X}_t^{\t \theta,[\theta]}(v) \right] \right|^2 \\
                \leq & \left\| G_l \left( \Xd_t, [\Xt_t],\widetilde{X}^{\t \theta}_t \right) \right\|^2_{L^{2/r}(\Omega \times \t \Omega)}  \left\|  I^2(1)(t,x,[\theta]) \right\|^2_{L^{4/(1-r)}(\Omega \times \t \Omega)} \, \left\| \partial_{\mu} \widetilde{X}_t^{\t \theta,[\theta]}(v) \right\|^2_{L^{4/(1-r)}(\Omega \times \t \Omega)} \\
                \leq & C \left( \EE \widetilde{\EE} \left[  \left( 1 + | \Xd_t|^q+ \|\Xt_t\|_2^q + |\widetilde{X}^{\t \theta}_t|^r \right)^{2/r} \right] \right)^r \, (1+|x|+\|\theta\|_2)^6 \\
                \leq & C  \EE \widetilde{\EE} \left[ \left( 1 + | \Xd_t|^{2q/r}+ \|\Xt_t\|_2^{2q/r} + |\widetilde{X}^{\t \theta}_t|^2 \right) \right] \, (1+|x|+\|\theta\|_2)^6 \\
                \leq & C  \left( 1 + |x|^{2q/r}+ \|\theta\|_2^{2q/r} + \|\theta\|_2^2 \right)  \, (1+|x|+\|\theta\|_2)^6.
        \end{align*}
        Clearly this is bounded in $[\theta]$ over compacts in $\p_2(\RR^N)$.
        
        Now, we consider the derivative $\partial_v \partial_{\mu} U_l$. We note that in the definition of $ \I^1(t,x,[\theta],v)$, the only term depending on $v$ is $\partial_{\mu} \Xd_t(v)$. Since $\Vs \in \C^{3,3}_{b,Lip}(\RR^N \times \p_2(\RR^N);\RR^N)$ by assumption, $ \partial_v \I^1(t,x,[\theta],v)$ exists and we obtain:
        \begin{align}
                \label{eq:dmudvU}
                \begin{split}
                        \partial_v\partial_{\mu}U_l(t,x,[\theta],v)  = t^{-1/2} \, \EE \bigg\{ & g_l \left( \Xd_t, [\Xt_t] \right)\, \partial_v \I^1(t,x,[\theta],v) \\
                        & + \widetilde{\EE} \big[ \partial_v G_l \left( \Xd_t, [\Xt_t], \widetilde{X}^{v,[\theta]}_t \right)I^2(1)(t,x,[\theta]) \,  \left(\partial_{v} \widetilde{X}_t^{v,[\theta]}\right)^2 \\
                        & \quad  +  G_l \left( \Xd_t, [\Xt_t], \widetilde{X}^{v,[\theta]}_t \right)I^2(1)(t,x,[\theta]) \,  \partial^2_{v} \widetilde{X}_t^{v,[\theta]} \\
                        & \quad \: +  G_l \left( \Xd_t, [\Xt_t],\widetilde{X}^{\t \theta}_t \right) \, I^2(1)(t,x,[\theta]) \, \partial_v \partial_{\mu} \widetilde{X}_t^{\t \theta,[\theta]}(v) \big] \bigg\}.
                \end{split}
        \end{align}
        We again use that
        \[
        Id_N = \D_r X^{x,[\theta]}_t \sigma^{\top} (\sigma \sigma^{\top})^{-1} (X^{x,[\theta]}_r, [X_r^{\theta}]) \partial_x X^{x,[\theta]}_r \left(   \partial_x X^{x,[\theta]}_t\right)^{-1}.
        \]
        Of course, this identity also holds for `tilde' processes defined on $\left(\t \Omega, \t \F, \t \PP \right)$ and we denote by $\widetilde{\D}$ the Malliavin derivative on this space.
        So,  using the above identity and the Malliavin chain rule, we obtain
        \begin{align*}
                &       \partial_{v} G_l \left( \Xd_t, [\Xt_t], \widetilde{X}^{v,[\theta]}_t \right) I^2(1)(t,x,[\theta]) \,  \left(\partial_{v} \widetilde{X}_t^{v,[\theta]}\right)^2 \\
                = \: & \partial_{v} G_l \left( \Xd_t, [\Xt_t], \widetilde{X}^{v,[\theta]}_t \right) \widetilde{\D}_r \widetilde{X}^{v,[\theta]}_t \sigma^{\top} (\sigma \sigma^{\top})^{-1} (\widetilde{X}^{v,[\theta]}_r, [\widetilde{X}_r^{\theta}]) \partial_x \widetilde{X}^{v,[\theta]}_r I^2(1)(t,x,[\theta]) \,  \partial_{v} \widetilde{X}_t^{v,[\theta]} \\
                = \: & \widetilde{\D}_r \left[G_l \left( \Xd_t, [\Xt_t], \widetilde{X}^{v,[\theta]}_t \right) \right] \sigma^{\top} (\sigma \sigma^{\top})^{-1} (\widetilde{X}^{v,[\theta]}_r, [\widetilde{X}_r^{\theta}]) \partial_x \widetilde{X}^{v,[\theta]}_r I^2(1)(t,x,[\theta]) \,  \partial_{v} \widetilde{X}_t^{v,[\theta]}
        \end{align*}
        and, applying the integration by parts formula in Proposition \ref{th:IBPx1} on the space $\left(\t \Omega, \t \F, \t \PP \right)$, we get
        \begin{align*}
                &       \EE \widetilde{\EE} \left[ \widetilde{\D}_r \left[G_l \left( \Xd_t, [\Xt_t], \widetilde{X}^{v,[\theta]}_t \right) \right] \sigma^{\top} (\sigma \sigma^{\top})^{-1} (\widetilde{X}^{v,[\theta]}_r, [\widetilde{X}_r^{\theta}]) \partial_x \widetilde{X}^{v,[\theta]}_r I^2(1)(t,x,[\theta]) \,  \partial_{v} \widetilde{X}_t^{v,[\theta]} \right] \\
                = \: &  t^{-1/2} \,     \EE \widetilde{\EE} \left[ G_l \left( \Xd_t, [\Xt_t], \widetilde{X}^{v,[\theta]}_t \right) \widetilde{I}^2 \left(  \partial_{x} \widetilde{X}_{\cdot}^{\cdot,\cdot}\right)(t,v,[\theta]) \, I^2(1)(t,x,[\theta]) \right].
        \end{align*}
        So, \eqref{eq:dmudvU} becomes
        \begin{align}
                \label{eq:dmudvUIBP}
                \begin{split}
                        & \partial_v\partial_{\mu}U_l(t,x,[\theta],v)  = t^{-1} \, \EE \bigg\{  \sqrt{t} \, g_l \left( \Xd_t, [\Xt_t] \right)\, \partial_v \I^1(t,x,[\theta],v) \\
                        & + \widetilde{\EE} \big[ G_l \left( \Xd_t, [\Xt_t],  \widetilde{X}^{v,[\theta]}_t  \right) \, I^2(1)(t,x,[\theta]) \left( \widetilde{I}^2 \left(  \partial_{x} \widetilde{X}_{\cdot}^{\cdot,\cdot}\right)(t,v,[\theta]) + \sqrt{t} \,  \,  \partial^2_{v} \widetilde{X}_t^{v,[\theta]} \right) \\
                        & \quad \: + \sqrt{t} \, G_l \left( \Xd_t, [\Xt_t],\widetilde{X}^{\t \theta}_t \right) \, I^2(1)(t,x,[\theta]) \, \partial_v \partial_{\mu} \widetilde{X}_t^{\t \theta,[\theta]}(v) \big] \bigg\}.
                \end{split}
        \end{align}
        
        We can check each expectation above is finite by using the growth conditions on the functions $g_l$, $G_l$ and their derivatives along with H\"older's inequality and the moment estimates on the processes involved, similar to before. In particular, note that we can obtain estimates on \eqref{eq:dmuIBPx} and \eqref{eq:dmudvUIBP} independently of $l$. This allows us to use dominated convergence to pass to the limit in these equations.
        
        Now, suppose that  \textbf{(H2)} holds instead of \textbf{(H1)}. Under \textbf{(H2)}, $g$ in the class \textbf{(IC)$_v$}.
        By Lemma \ref{lem:dmu}, we have an expression for  $\partial_{\mu}U_l$ and using the special form of $\partial_{\mu}g_l$ for initial conditions in the class \textbf{(IC)$_v$}, we get
        \begin{align}
                \label{eq:Uv}
                \begin{split}
                        & \partial_{\mu}U_l(t,x,[\theta],v) = \EE \left[ \partial g_l \left( \Xd_t, [\Xt_t] \right) \partial_{\mu} \Xd_t (v)   \right] \\
                        & \quad + \:  \EE \widetilde{\EE} \left[ \partial_{v} G_l \left( \Xd_t, [\Xt_t], \widetilde{X}^{v,[\theta]}_t \right) \partial_{v} \widetilde{X}_t^{v,[\theta]} +  \partial_{v} G_l \left( \Xd_t, [\Xt_t], \widetilde{X}^{\t \theta}_t \right) \partial_{\mu} \widetilde{X}_t^{\t \theta,[\theta]}(v)   \right].
                \end{split}
        \end{align}
        We again use that
        \[
        Id_N = \D_r X^{x,[\theta]}_t \sigma^{\top} (\sigma \sigma^{\top})^{-1} (X^{x,[\theta]}_r, [X_r^{\theta}]) \partial_x X^{x,[\theta]}_r \left(   \partial_x X^{x,[\theta]}_t\right)^{-1}.
        \]
        Of course, this identity also holds for `tilde' processes defined on $\left(\t \Omega, \t \F, \t \PP \right)$ and we denote by $\widetilde{\D}$ the Malliavin derivative on this space.
        So,  using the above identity and the Malliavin chain rule, we obtain
        \begin{align*}
                &       \partial_{v} G_l \left( \Xd_t, [\Xt_t], \widetilde{X}^{v,[\theta]}_t \right) \partial_{v} \widetilde{X}_t^{v,[\theta]} \\
                = \: & \partial_{v} G_l \left( \Xd_t, [\Xt_t], \widetilde{X}^{v,[\theta]}_t \right) \widetilde{\D}_r \widetilde{X}^{v,[\theta]}_t \sigma^{\top} (\sigma \sigma^{\top})^{-1} (\widetilde{X}^{v,[\theta]}_r, [\widetilde{X}_r^{\theta}]) \partial_x \widetilde{X}^{v,[\theta]}_r \\
                = \: & \widetilde{\D}_r \left[G_l \left( \Xd_t, [\Xt_t], \widetilde{X}^{v,[\theta]}_t \right) \right] \sigma^{\top} (\sigma \sigma^{\top})^{-1} (\widetilde{X}^{v,[\theta]}_r, [\widetilde{X}_r^{\theta}]) \partial_x \widetilde{X}^{v,[\theta]}_r
        \end{align*}
        and, applying the integration by parts formula in Proposition \ref{th:IBPx1} on the space $\left(\t \Omega, \t \F, \t \PP \right)$, we get
        \begin{align*}
                &       \EE \widetilde{\EE} \left[ \widetilde{\D}_r \left[G_l \left( \Xd_t, [\Xt_t], \widetilde{X}^{v,[\theta]}_t \right) \right] \sigma^{\top} (\sigma \sigma^{\top})^{-1} (\widetilde{X}^{v,[\theta]}_r, [\widetilde{X}_r^{\theta}]) \partial_x \widetilde{X}^{v,[\theta]}_r \right] \\
                = \: &  t^{-1/2} \,     \EE \widetilde{\EE} \left[ G_l \left( \Xd_t, [\Xt_t], \widetilde{X}^{v,[\theta]}_t \right) \widetilde{I}(1)(t,v,[\theta])\right].
        \end{align*}
        Similarly,
        \begin{align*}
                &        \partial_{v} G_l \left( \Xd_t, [\Xt_t], \widetilde{X}^{\t \theta}_t \right) \partial_{\mu} \widetilde{X}_t^{\t \theta,[\theta]}(v)\\
                = \: & \partial_{v} G_l \left( \Xd_t, [\Xt_t], \widetilde{X}^{\t \theta}_t \right)\widetilde{\D}_r \widetilde{X}^{\t \theta}_t \sigma^{\top} (\sigma \sigma^{\top})^{-1} (\widetilde{X}^{\t \theta}_r, [X_r^{\theta}]) \partial_x \widetilde{X}^{\t \theta,[\theta]}_r \left(  \partial_x \widetilde{X}^{\t \theta,[\theta]}_t\right)^{-1}  \partial_{\mu} \widetilde{X}_t^{\t \theta,[\theta]}(v) \\
                = \: & \widetilde{\D}_r \left[G_l \left( \Xd_t, [\Xt_t],\widetilde{X}^{\t \theta}_t \right) \right] \sigma^{\top} (\sigma \sigma^{\top})^{-1} (\widetilde{X}^{\t \theta}_r, [X_r^{\theta}]) \partial_x \widetilde{X}^{\t \theta,[\theta]}_r \left(  \partial_x \widetilde{X}^{\t \theta,[\theta]}_t\right)^{-1}  \partial_{\mu} \widetilde{X}_t^{\t \theta,[\theta]}(v)
        \end{align*}
        and applying the integration by parts formula in Proposition \ref{th:IBPmu1} on the space $\left(\t \Omega, \t \F, \t \PP \right)$, we get
        \begin{align*}
                &       \EE \widetilde{\EE} \left[ \widetilde{\D}_r \left[G_l \left( \Xd_t, [\Xt_t],\widetilde{X}^{\t \theta}_t \right) \right] \sigma^{\top} (\sigma \sigma^{\top})^{-1} (\widetilde{X}^{\t \theta}_r, [X_r^{\theta}]) \partial_x \widetilde{X}^{\t \theta,[\theta]}_r \left(  \partial_x \widetilde{X}^{\t \theta,[\theta]}_t\right)^{-1}  \partial_{\mu} \widetilde{X}_t^{\t \theta,[\theta]}(v) \right] \\
                = \: & t^{-1/2} \, \EE \widetilde{\EE} \left[ G_l \left( \Xd_t, [\Xt_t],\widetilde{X}^{\t \theta}_t \right) \, \widetilde{\I}^1(1)(t,\t \theta,[\theta],v) \right].
        \end{align*}
        Here we explain the reason for insisting that the coefficients $\Vs$ are bounded: the Kusuoka-Stroock process $\widetilde{\I}^1(1)(t,x,[\theta],v)$ is bounded in $L^p(\t \Omega)$ uniformly in $(x,[\theta],v)$. This allows us to evaluate at $x=\t \theta$ and take expectation with respect to $\widetilde{\EE}$. If the coefficients are not bounded, the bound we have on  $\|\widetilde{\I}^1(1)(t,x,[\theta],v)\|_p$ grows like $|x|^4$ according to Proposition \ref{prop:defwieghts} and we cannot guarantee that $\textstyle \EE \widetilde{\EE} \left[ \widetilde{\I}^1(1)(t,\t \theta,[\theta],v) \right]$ is finite.
        
        Putting the above integration by parts formulas together and using Proposition \ref{th:IBPmu1} on the space $\left( \Omega,  \F,  \PP \right)$ for the first term on the right hand side of \eqref{eq:Uv}, we see that it can be re-written as
        \begin{align}
                \label{eq:dmuIBP}
                \begin{split}
                        & \partial_{\mu}U_l(t,x,[\theta],v)  = t^{-1/2} \, \EE \bigg\{ g_l \left( \Xd_t, [\Xt_t] \right) \I^1(t,x,[\theta],v) \\
                        & + \widetilde{\EE} \left[ G_l \left( \Xd_t, [\Xt_t], \widetilde{X}^{v,[\theta]}_t \right) \widetilde{I}^1(1)(t,v,[\theta]) +
                        G_l \left( \Xd_t, [\Xt_t],\widetilde{X}^{\t \theta}_t \right) \, \widetilde{\I}^1(1)(t,\t \theta,[\theta],v) \right] \bigg\}
                \end{split}
        \end{align}
        and we note the RHS does not depend on derivatives of the functions $g$ and $G$.
        Also,
        \begin{align}
                \label{eq:dvdmuv}
                \begin{split}
                        & \partial_v \partial_{\mu}U_l(t,x,[\theta],v)  = t^{-1/2} \, \EE \bigg\{ g_l \left( \Xd_t, [\Xt_t] \right) \partial_v \I^1(t,x,[\theta],v) \\
                        & + \widetilde{\EE} \bigg[ G_l \left( \Xd_t, [\Xt_t], \widetilde{X}^{v,[\theta]}_t \right) \partial_v \widetilde{I}^1(1)(t,v,[\theta]) +
                        G_l \left( \Xd_t, [\Xt_t],\widetilde{X}^{\t \theta}_t \right) \, \partial_v \widetilde{\I}^1(1)(t,\t \theta,[\theta],v)  \bigg\} \\
                        & +  \partial_v G_l \left( \Xd_t, [\Xt_t], \widetilde{X}^{v,[\theta]}_t \right) \, \partial_v \widetilde{X}^{v,[\theta]}_t \, \widetilde{I}^1(1)(t,v,[\theta])  \bigg] \bigg\}
                \end{split}
        \end{align}
        so, applying Proposition \ref{th:IBPx1}, we get
        \begin{align}
                \label{eq:IBPdvdmuv}
                \begin{split}
                        & \partial_v \partial_{\mu}U_l(t,x,[\theta],v)  = t^{-1/2} \, \EE \bigg\{ g_l \left( \Xd_t, [\Xt_t] \right) \partial_v \I^1(t,x,[\theta],v) \\
                        & + \widetilde{\EE} \bigg[ G_l \left( \Xd_t, [\Xt_t], \widetilde{X}^{v,[\theta]}_t \right) \partial_v \widetilde{I}^1(1)(t,v,[\theta]) +
                        G_l \left( \Xd_t, [\Xt_t],\widetilde{X}^{\t \theta}_t \right) \, \partial_v \widetilde{\I}^1(1)(t,\t \theta,[\theta],v)  \bigg\} \\
                        & +  t^{-1/2} G_l \left( \Xd_t, [\Xt_t], \widetilde{X}^{v,[\theta]}_t \right) \, \widetilde{I}^1 \left(\widetilde{I}^1(1)\right)(t,v,[\theta])  \bigg] \bigg\}.
                \end{split}
        \end{align}
        
\end{proof}

\begin{remark}
Immediately from the proof of Lemma \ref{lem:ctyU} one can deduce the following gradient bounds for the function $U(t,x,[\theta]):= \EE \left[g\left(\Xd_t \left[\Xt_t\right]\right)\right]$ under the same conditions \textbf{(H1)} or \textbf{(H2)}:
There exists positive constants $C$ and $q$ such that for any $(t,x,[\theta]) \in (0,T] \times \RR^N \times \p_2(\RR^N), v \in \RR^N $  
\begin{align*}
                |\partial^i_x U(t,x,[\theta])|&,|\partial_{\mu} U(t,x,[\theta])|\le\ Ct^{-1/2} (1+|x|+\|\theta\|_2)^q,\\
                |\partial^{(i,j)}_x U(t,x,[\theta])| &, |\partial_{v}\partial_{\mu} U(t,x,[\theta])|\le Ct^{-1} (1+|x|+\|\theta\|_2)^q.
\end{align*}

\end{remark}

\noindent We now define what we mean by a classical solution to the PDE \eqref{eq:MasterPDE}.

\begin{definition}
        \label{def:classical}
        Suppose that $U: [0,T] \times \RR^N \times \p_2(\RR^N) \to \RR$ satisfies \eqref{eq:MasterPDE} and
        \begin{align*}
                (0,T]\times\RR^N\times\p_2(\RR^N) \ni(t,x,[\theta]) & \mapsto \left( \partial_x U(t,x,[\theta]), \, \partial^2_{x,x}U(t,x,[\theta]) \right) \\ 
                (0,T]\times\RR^N\times\p_2(\RR^N)\times \RR^N \ni(t,x,[\theta],v) & \mapsto \left(  \partial_{\mu}U(t,x,[\theta],v), \, \partial_v \partial_{\mu}U(t,x,[\theta],v) \right)
        \end{align*}
        exist and are continuous. Moreover, suppose that for all $(x,\theta) \in \RR^N \times L^2(\Omega)$
        \begin{equation}
        \label{eq:bdycty}
        \lim_{(t,y,[\gamma]) \to (0,x,[\theta]) } U(t,y,[\gamma]) = g(x,[\theta]).
        \end{equation}
        Then we say that $U$ is a classical solution to the PDE \eqref{eq:MasterPDE}.
\end{definition}
%Now, we will prove a theorem on existence and uniqueness result for classical solutions to the PDE \eqref{eq:MasterPDE}. Let us describe the class of initial conditions we consider. We say $g:\RR^N \times \p_2(\RR^N) \to \RR$ satisfies (IC) if either
%\begin{enumerate}
%\item $g(z, \mu)=g(z)$ is does not depend on the measure variable, and $ g:\RR^N \to \RR$ has polynomial growth.
%\item $g$ is either a function of a centred random variable or a convolution as in Examples \ref{ex:centred} and \ref{ex:convolutions}, where $h:\RR^N \to \RR$ has polynomial growth.
%\end{enumerate}

\begin{theorem}
        \label{th:masterPDE}
        Suppose that either \textbf{(H1)} or \textbf{(H2)} holds. Then 
        $$U(t,x,[\theta]):= \EE \left( g\left(X_t^{x, [\theta]},\left[\Xt_t \right]\right) \right)$$ is a classical solution of the PDE \eqref{eq:MasterPDE}. Moreover, $U$
        is unique among all of the classical solutions satisfying the polynomial growth condition $\left|U(t,x,[\theta])\right| \leq C (1+|x|+\|\theta\|_2)^q$ for some $q>0$ and all $(t,x,[\theta]) \in [0,T] \times \RR^N \times \p_2(\RR^N)$.
\end{theorem}
\begin{proof}
        \emph{Existence:}
        To prove continuity at the boundary, we use continuity of $g$ and the fact that 
        \[
        \left\| \Xt_t - \theta \right\|_2 +     \left\| \Xd_t - x \right\|_2 \to 0 \quad \text{ as } \quad t \to 0,
        \]
        which follows from \eqref{eq:MKVLpreg}.
        
        Now, we note that by the flow property we have, for $h >0$, 
        \[
        \left( X_{t+h}^{x,[\theta]}, X_{t+h}^{\theta} \right) = \left( X_t^{X_{h}^{x,[\theta]},[X_{h}^{[\theta]}]}, X_t^{X_h^{\theta}} \right)
        \]
        so that,
        \begin{align*}
                U(t+h,x,[\theta]) &= \EE \left[ g \left( X_{t+h}^{x,[\theta]},\left[ X^{\theta}_{t+h}\right] \right) \right]
                = \EE \left[ \EE  \left. \left\{ g \left( X_t^{X_h^{x,[\theta]},[X_h^{\theta}]}, \left[ X_t^{X_h^{\theta}}\right] \right)  \right\} \right| \mathcal{F}_h \right]\\
                &= \EE \, U(t,X_h^{x,[\theta]},[X_h^{\theta}]).
        \end{align*}
        Hence,
        \begin{align}
                \nonumber
                & U(t+h,x,[\theta]) - U(t,x,[\theta]) \\
                \nonumber
                =& \, \EE \, U(t,X_h^{x,[\theta]},[X_h^{\theta}]) - U(t,x,[\theta]) \\
                \label{eq:Uincrement}
                = &  \left\{  U(t,x,[X_h^{\theta}]) - U(t,x,[\theta]) \right\}  + \EE \left\{ U(t,X_h^{x,[\theta]},[X_h^{\theta}]) - U(t,x,[X_h^{\theta}]) \right\}.
        \end{align}
        The idea is to expand the first term using the chain rule introduced in \cite{crischassdel} and the second term using It\^o's formula. Then, dividing by $h$ and sending it to 0, along with continuity of the terms appearing in the expansion, will prove that $U$ indeed solves the PDE \eqref{eq:MasterPDE}. 
        
        Lemma \ref{lem:ctyU} guarantees that we can apply the chain rule proved in \cite{crischassdel}. 
        We apply it to the function $U(t,x,\cdot)$ to get
        \begin{align*}
                & U(t,x,[X_h^{\theta}]) - U(t,x,[\theta]) = \int_0^h \EE \left[ \sum_{i=1}^N V_0^i(X_r^{ \theta },[X_r^{ \theta}]) \, \partial_{\mu} U(t,x,[X_r^{ \theta}],X_r^{ \theta} )_i \right] \, dr \\
                & \quad + \frac{1}{2} \int_0^h \EE \left[  \sum_{i,j=1}^N [\sigma \sigma^{\top} (X_r^{ \theta} ,[X_r^{ \theta}])]_{i,j} \, \partial_{v_j} \partial_{\mu} U(t,x,[X_r^{ \theta}],X_r^{ \theta} )_i \right] \, dr
 .       \end{align*}
        It\^o's formula applied to $U(t,\cdot,[X_h^{\theta}])$ gives
        \begin{align*}
                & U(t,X_h^{x,[\theta]},[X_h^{\theta}]) - U(t,x,[X_h^{\theta}]) = \int_0^h \sum_{i=1}^N V_0^i(X_r^{x,[\theta]},[X_r^{\theta}]) \, \partial_{x_i} U(t,X_r^{x,[\theta]},[X_h^{\theta}]) \, dr \\
                & + \frac{1}{2} \int_0^h \sum_{i,j=1}^N [\sigma \sigma^{\top} (X_r^{x,[\theta]},[X_r^{\theta}])]_{i,j} \, \partial_{x_i} \partial_{x_j}  U(t,X_r^{x,[\theta]},[X_h^{\theta}]) \, dr \\
                & + \int_0^h \sum_{j=1}^d \sum_{i=1}^N V_j^i(X_r^{x,[\theta]},[X_r^{\theta}]) \, \partial_{x_i} U(t,X_r^{x,[\theta]},[X_h^{\theta}]) \, dB^j_r
.        \end{align*}
        We want the final term to be square integrable, so that it is a true martingale with zero expectation. We have that for some $q>0$,
        \begin{align*}
                \left| \partial_{x_i} U(t,x,[\theta]) \right| &\leq t^{-1/2} \left\| g(\Xd_t,[\Xt_t]) \right\|_2 \left\|I^1_{(i)}(1)(t,x,[\theta])\right\|_2 \\
                & \leq C\,  t^{-1/2} \left\| \left(1+ \left| \Xd_t \right| + \|\Xt_t\|_2 \right)^q \right\|_2  \left(1+ |x|+\|\theta\|_2 \right)^3 \\
                & \leq C \,  t^{-1/2}  \left(1+ |x|+\|\theta\|_2\right)^{q+3},
        \end{align*}
        so that for all $p \geq 1$,
        \begin{align*}
                \EE \left|\partial_{x_i} U \left(t,X_r^{x,[\theta]},[X_h^{\theta}]\right)\right|^p & \leq C\,t^{-1/2} \EE \left(1+ \left| X_r^{x,[\theta]} \right|+ \left\|X_h^{\theta}\right\|_2\right)^{p(q+3)} \\
                &  \leq C\,t^{-1/2} \EE \left(1+ \left| x \right|+ \left\|\theta\right\|_2\right)^{p(q+3)},
        \end{align*}
        and by the linear growth of $V_j^i$, we have
        \[
        \EE \left|      V_j^i(X_r^{x,[\theta]},[X_r^{\theta}]) \right|^p \leq C (1+|x|+\|\theta\|_2)^p.
        \]
        Hence, the final term is indeed square integrable, and has zero expectation.
        
        Putting the expansions back into \eqref{eq:Uincrement}, we get
        \begin{align*}
                & U(t+h,x,[\theta]) - U(t,x,[\theta]) \\
                = & \int_0^h \EE \left[ \sum_{i=1}^N V_0^i(X_r^{ \theta} ,[X_r^{ \theta}]) \, \partial_{\mu} U(t,x,[X_r^{ \theta}],X_r^{ \theta} )_i \right] \, dr \\
                & \quad + \frac{1}{2} \int_0^h \EE \left[  \sum_{i,j=1}^N [\sigma \sigma^{\top} (X_r^{ \theta} ,[X_r^{ \theta}])]_{i,j} \, \partial_{v_j} \partial_{\mu}  U(t,x,[X_r^{ \theta}],X_r^{\theta} )_i \right] \, dr \\
                & \quad + \EE \int_0^h \sum_{i=1}^N V_0^i(X_r^{x,[\theta]},[X_r^{\theta}]) \, \partial_{x_i} U(t,X_h^{x,[\theta]},[X_h^{\theta}]) \, dr \\
                & \quad +  \frac{1}{2} \EE \int_0^h \sum_{i,j=1}^N [\sigma \sigma^{\top} (X_r^{x,[\theta]},[X_r^{\theta}])]_{i,j} \, \partial_{x_i} \partial_{x_j}  U(t,X_h^{x,[\theta]},[X_h^{\theta}]) \, dr
      .  \end{align*}
        By the earlier results on continuity of $U$ and its derivatives and the a priori continuity of the coefficients $V_0, \ldots, V_d$ we see that the integrand on the right hand side is a continuous function of $h$. Dividing by $h$ and sending it to zero, we see that $U$ solves the PDE \eqref{eq:MasterPDE}.
        
        \emph{Uniqueness:} Fix any $t \in (0,T]$ and any classical solution $W$ with polynomial growth.  Set $\delta>0$, so 
        \begin{align*}
                & W(t,x,[\theta]) - W(0,\Xd_t,[\Xt_t]) \\
                & = W(t,x,[\theta]) - W(\delta,\Xd_{t-\delta},[\Xt_{t-\delta}]) + W(\delta,\Xd_{t-\delta},[\Xt_{t-\delta}]) - W(0,\Xd_t,[\Xt_t]).
        \end{align*}
        By the polynomial growth of $W$, this is square integrable.
        Now we expand the process $(W(t-s,\Xd_s,[\Xt_s]))_{s \in [\delta,t]}$ and use that $W$ is a solution of the PDE \eqref{eq:MasterPDE}, so that the drift is zero, to get
        \begin{align*}
                & W(t,x,[\theta]) - W(0,\Xd_t,[\Xt_t]) \\
                & = \sum_{j=1}^d \sum_{i=1}^N \int_{\delta}^t V_j^i(X_r^{x,[\theta]},[X_r^{\theta}]) \, \partial_{x_i} W(t-r,X_r^{x,[\theta]},[X_r^{\theta}]) \, dB^j_r \\
                &  \quad \quad + W(\delta,\Xd_{t-\delta},[\Xt_{t-\delta}]) - W(0,\Xd_t,[\Xt_t])
     .   \end{align*}
        As we have already noted, this is square-integrable, so the stochastic integral is a true martingale with zero expectation. So taking expectation in the above expansion, we get:
        \[
        \, W(t,x,[\theta]) -    \EE  W(0,\Xd_t,[\Xt_t])  = \EE \left[W(\delta,\Xd_{t-\delta},[\Xt_{t-\delta}]) - W(0,\Xd_t,[\Xt_t])\right] 
      .  \] 
        Now, sending $\delta \searrow 0$ and using continuity of $W$ at the boundary (condition \eqref{eq:bdycty} in the definition of classical solution), the right hand side disappears, and we get that
        \[
        W(t,x,[\theta]) =       \EE \,  W(0,\Xd_t,[\Xt_t]) = \EE \left[ g \left(\Xd_t,[\Xt_t]\right) \right],
        \]
        which  completes the proof.
\end{proof}

% % % % % % % % % % % % % % % % % % % % % % % % % % %
\section{Application to the density function}
% % % % % % % % % % % % % % % % % % % % % % % % % % %
\label{sec:densities}

In this section, we apply the integration by parts formulae to the study of the density function $p(t,x,z)$ of the McKean-Vlasov SDE started from a fixed point, $X_t^{x,\delta_x}$, at a fixed time $t \in [0,T]$. 
Throughout this section, we assume that (UE) holds and  $\Vs \in  \C^{k,k}_{b,Lip}(\RR^N \times \p_2(\RR^N);\RR^N)$. We can consider $\Xd_t$ as the solution of a classical SDE with time-dependent coefficients. Hence, under (UE), the smoothness of its density (call it $q(t,x,[\theta],\cdot)$) has been studied in the classical work of Friedman \cite{friedmanparabolic}. Since $p(t,x,z)=q(t,x,\delta_x,z)$, Friedman's results also establish the smoothness of $p(t,x,z)$ in the forward variable, $z$. However, they do not cover the smoothness of the function $p(t,x,z)$ in the backward variable, $x$. The density $p(t,x,z)$ has also been studied by Antonelli \& Kohatsu-Higa in \cite{antonellikohatsu} under a H\"ormander condition on the coefficients. In this case, they establish smoothness of the density in the forward variable, $z$, but do not establish estimates on the derivatives of this function. The theorem which follows esatblishes the smoothness of $p(t,x,z)$ in the variables $(x,z)$ and we also obtain estimates on its derivatives.

\begin{theorem}
        \label{th:densityestimate}
        Let $\alpha, \beta$ be  multi-indices on $\{1, \ldots, N\}$ and let $k \geq |\alpha|+|\beta|+N+2$. 
        Then, for all $t \in (0,T]$ and $\theta \in L^2(\Omega)$, $X_t^{x,\delta_x}$ has a density $p(t,x, \cdot)$ such that $(x,z) \mapsto  \partial_x^{\alpha} \, \partial_z^{\beta}p(t,x, z) $ exists and is continuous. Moreover, there exists a constant $C$ which depends on $T$, $N$ and bounds on the coefficients, such that for all $t \in (0,T]$
        \begin{align}
                \label{eq:densityest2}  | \partial_x^{\alpha} \, \partial_z^{\beta}  p(t,x,z) | &\leq C \, (1+ |x|)^{\mu} \,  t^{- \nu}  ,
        \end{align}
        where $ \mu = 4|\alpha|+ 3 |\beta| + 3 N$ and  $ \nu = \textstyle \frac{1}{2} (N + | \alpha| + | \beta | )$. If $\Vs$ are bounded then the following estimate holds
        \begin{align}
                \label{eq:densityest4}  | \partial_x^{\alpha} \, \partial_z^{\beta}  p(t,x;z) | &\leq C \,  t^{- \nu} \, \exp \left(- C \, \frac{|z-x|^2}{t} \right) .
        \end{align}
\end{theorem}

\begin{proof}
        Let $\eta = (1,2, \ldots, N)$ and introduce the multi-dimensional indicator function
        $
        \mathbf{1}_{ \{ z_0>z \} } := \textstyle \prod_{i=1}^N \mathbf{1}_{ \{ z_0^i>z^i \} }  .
        $ For any $g \in \C^{\infty}_0(\RR^N;\RR)$ the function $f $ defined by 
        \begin{equation}
        \label{eq:FTC}
        f(z_0) := \int_{\RR^N} g (z) \mathbf{1}_{\{z_0>z \} } \, dz 
        \end{equation}
        is in $ \C^{\infty}_p(\RR^N;\RR)$ and satisfies $\partial^{\eta} f = g$. 
        Now, we first focus on $p(t,x,\cdot)$, the density of $X_t^{x,\delta_x}$.
        \begin{align}
                \nonumber       & \partial_x^{\alpha} \, \EE[ (\partial^{\beta} g)(X_t^{x,\delta_x})] \\
                \nonumber       &= \partial_x^{\alpha} \,  \EE[(\partial^{\beta * \eta} f)(X_t^{x,\delta_x})] \\
                \nonumber       &= t^{-(|\eta| + |\beta|+|\alpha| )/2} \,  \EE[ f(X_t^{x,\delta_x}) I^2_{\beta*\eta}(J_{\alpha}(1))(t,x) ] \\
                \nonumber       & = t^{\frac{-(N + |\beta|+ | \alpha|)}{2}}  \, \EE \left[ \left( \int_{\RR^N} g(z) \mathbf{1}_{ \{X_t^{x,\delta_x} > z \} } \, dz \right) \,  I^2_{\beta*\eta}(J_{\alpha}(1))(t,x) \right ] \\
                \label{eq:IBPfordensity}        & = t^{\frac{-(N + |\beta|+ | \alpha|  )}{2}} \, \int_{\RR^N} g (z) \, \EE \left[  \mathbf{1}_{ \{X_t^{x,\delta_x} > z \} } I^2_{\beta*\eta}(J_{\alpha}(1))(t,x) \right] \, dz,  
        \end{align} where we have used at each step respectively: $\partial^{\eta} f = g$; Corollary \ref{cor:MKVIBP} ; equation \eqref{eq:FTC}, and Fubini's theorem.
        It then follows that, for any $R>0$ and $t \in(0,T]$, there exists $C=C(R,t)>0$ such that
        \begin{align*}
                \sup_{|x|\leq R} \left( \left| \partial_x^{\alpha} \, \EE[ (\partial^{\beta} g)(X^x_t)] \right| +  \left| \partial_x^{\alpha} \, \EE[ (\partial^{\beta} g)(\Xd_t)] \right| \right) \leq C \, \|g\|_{\infty}.
        \end{align*}
        Then, it is a result from Taniguchi \cite[Lemma 3.1]{taniguchi1985applications} that  $X^{x,\delta_x}_t$ has a density function, $p(t,x,\cdot)$ and that  $\partial_x^{\alpha} \, \partial_z^{\beta}p(t,x, z)$ exists. Once we know that a smooth density exists, it follows from \eqref{eq:IBPfordensity} that
        we can identify $       \partial_x^{\alpha} \, \partial_z^{\beta}  p(t,x,z)$ as
        \begin{align*}
                \partial_x^{\alpha} \, \partial_z^{\beta}   p(t,x,z) &= t^{\frac{-(N + |\beta|+ | \alpha|  )}{2}} \, (-1)^{|\beta|} \, \EE \left[ \mathbf{1}_{ \{X_t^{x,\delta_x} > z \} } I^2_{\beta*\eta}(J_{\alpha}(1))(t,x) \right] .
        \end{align*}
        Now, the following estimates come from each term's membership of the Kusuoka-Stroock class, as guaranteed by Proposition \ref{prop:defwieghts} and Corollary \ref{cor:MKVIBP}:
        \begin{align*}
                \| I^2_{\beta*\eta}(J_{\alpha}(1))(t,x) \|_{p} &\leq C\, (1+|x|)^{\mu}.
        \end{align*} 
        This proves the estimate \eqref{eq:densityest2}.
        In addition, if $\Vs$ are bounded, we can estimate
        \begin{align*}
                \left\|   \mathbf{1}_{ \{X_t^{x,\delta_x} > z \} } \right\|_{p} &= \PP \left( \cap_{i=1}^N \{(X_t^{x,\delta_x})^i > z^i \}   \right) \\
                & \leq \min_{i=1, \ldots, N} \PP \left( (X_t^{x,\delta_x})^i > z^i  \right) \\
                & = \min_{i=1, \ldots, N} \PP \left( \sum_{j=1}^d \int_0^t V_j^i(X_s^{x,\delta_x},[X_s^{x,\delta_x}]) dB^j_s > z^i - x^i - \int_0^t V_0^i(X_s^{x,\delta_x},[X_s^{x,\delta_x}]) ds \right) 
  .      \end{align*}
        Now, we have that $\textstyle \int_0^t V_0^i(X_s^{x,\delta_x},[X_s^{x,\delta_x}]) ds \leq \|V_0\|_{\infty} t$ and the term
        \[
        M^i_t= \sum_{j=1}^d \int_0^t V_j^i(X_s^{x,\delta_x},[X_s^{x,\delta_x}]) dB^j_s ,
        \]
        is a martingale with quadratic variation $ \langle M^i \rangle_t \leq \textstyle \sum_{j=1}^d \|V_j\|^2 t$. We can therefore apply  the exponential martingale inequality to obtain
        \begin{align*}
                & \left\|   \mathbf{1}_{ \{X_t^{x,\delta_x} > z \} } \right\|_{p} \leq \min_{i=1, \ldots, N} \exp\left( - c' \frac{ |z^i - x^i - t \, \| V_0 \|_{\infty} \, |^2 }{t} \right).
        \end{align*}
        Then, we use $(a+b)^2 \geq \textstyle \frac{a^2}{2} - b^2$, which is re-arrangement of Young's inequality, to get
        \begin{align*}
                \frac{ |z^i - x^i - t \, \| V_0 \|_{\infty} \, |^2 }{t} \geq \frac{ |z^i - x^i  |^2 }{2t} -  \| V_0 \|^2_{\infty}.
        \end{align*}
        So, 
        \begin{align*}
                \min_{i=1, \ldots, N} \exp\left( - c' \frac{ |z^i - x^i - t \, \| V_0 \|_{\infty} \, |^2 }{t} \right) & \leq \min_{i=1, \ldots, N} \exp\left( - C \frac{ |z^i - x^i |^2 }{t} \right) \exp(c' \| V_0 \|^2_{\infty}) \\
                & \leq C \, \exp\left( - C \frac{ |z - x |^2 }{t} \right).
        \end{align*}
        This establishes \eqref{eq:densityest4}.
\end{proof}

\appendix
\section{Appendix}

\subsection{Elements of Malliavin Calculus}
\label{sec:introMalCal}

As indicated in the introduction, we will use some tools from Malliavin Calculus to develop integration by parts formulas. Here we introduce the basic terminology. We follow the exposition in \cite{crisan2010cubature}, with all proofs contained in the book by Nualart \cite{Nualart}.
%For an absolutely continuous path $h\in\C([0,T];\RR^d)$, we denote by $h'$ its derivative.
We denote $H_d:=L^2([0,T];\RR^d)$.
%and by $H^1_d$ be the space  
%\[
%H^1_d = \{h \in \Omega,\ h \ 
%\mathrm{absolutely\ \  continuous}, \ \ h' \in 
%H_d. \}
%\] 
%$H^1_d$ is endowed with a Hilbert structure under the inner product
%\[
%\langle h,g \rangle_{H^1_d} := \langle h',g'\rangle_{H_d} :=
%\int_0^{\infty}h'(u)\cdot g'(u) du
%\] 
%and is called the \textit{Cameron-Martin} space.
and use this space to define the Malliavin derivative.  \begin{definition}[Malliavin Derivative]
        Let $f \in \C_p^{\infty}(\RR^{n};\RR)$, for some $n \in \mathbb{N}$,
        $h_1,\ldots,h_n \in H_d$ and $F: \Omega
        \to\mathbb{R}$ be the functional given by:
        \begin{equation}\label{smooth}
        F(\omega) = f\left(\int_0^T h_1(t)\cdot
        dB_t(\omega),\ldots,\int_0^T h_n(t) \cdot dB_t(\omega)\right),
        \end{equation}
        where, for any $h_i=(h_{i}^1,\dots,h_{i}^d) \in H_d$ 
        \[
        \int_0^{T}h_i(t) \cdot dB_t := \sum_{j=1}^d \int_0^T
        h_{i}^j(t)\, dB^j_t. 
        \]
        Any functional of the form (\ref{smooth}) is called 
        \textit{smooth} and we denote the class of all such functionals
        by $\mathcal{S}$. Then the Malliavin derivative of $F$,
        denoted by $\D F \in L^2(\Omega ; H_d)$ is given by:
        \begin{equation}\label{MallDer}
        \D F = \sum_{i=1}^n \partial^i f\left(\int_0^{T}h_1(u) \cdot
        dB_u,\ldots,\int_0^{T}h_n(u) \cdot dB_u \right)
        h_i .
        \end{equation}
\end{definition}
\noindent We note the isometry $L^2(\Omega \times [0,T] ; \mathbb{R}^d) \simeq L^2(\Omega ; H_d)$. This allows us to identify $\D F$ with a process $\left( \D_r F\right)_{r \in [0,T]}$ taking values in $\RR^d$, which we often do. We also denote by $\left( \D^j_r F\right)_{r \in [0,T]}$, $j=1, \ldots, d,$ the components of this process.

The set of smooth functionals (random variables) 
$\mathcal{S}$ is dense in $L^p(\Omega)$, for any $p\geq 1$ and $\D$ is closable as operator from $L^p(\Omega)$ to $L^p\left( \Omega; H_d \right)$.
We define  $\DD^{1,p}$ is the closure of the set $\mathcal{S}$ within $L^p(\Omega; \mathbb{R}^d)$
with respect to the
norm:
\[
\|F\|_{\mathbb{D}^{1,p}} = \left(\EE \left|F \right|^p + \EE \left\| \D F \right\|^p_{H_d
}\right)^{\frac{1}{p}}.
\]
The higher order Malliavin derivatives are  defined in a similar manner. For smooth random variables, we denote the iterated derivative by
$\D^{(k)} F$, $k\ge 2$, which is a random variable with values in $H_d^{\otimes k}$
defined as
\[
\D^{(k)} F := \sum_{i_1,\ldots,i_k = 1}^n
\partial^{(i_1,\ldots,i_k)}f\left(\int_0^{\infty}h_1(u) \cdot
dB_u,\ldots,\int_0^{\infty}h_n(u) \cdot dB_u \right)h_{i_1} \otimes
\ldots \otimes h_{i_k}.
\]
The above expression for  
$\D^{(k)} F$ coincides  with that obtained by iteratively applying the
Malliavin derivative.
%Indeed, for $h \in H$, $F \in \mathcal{S}$, it
%is easily seen that $D_hF \in \mathcal{S}$. As per \eqref{directderiv},
%it can be shown
%that,
%\[
%D_{h_k}D_{h_{k-1}}\ldots D_{h_1}F= \langle D^k F, 
%h_1 \otimes\ldots \otimes h_k\rangle_{H^{\otimes k}}.
%\]
In an analogous way, one can
close the operator $\D^{(k)}$ from $L^p(\Omega)$ to $L^p(\Omega ;
H_d^{\otimes k})$. So, for any $p \geq 1$ and natural $k \geq 1$,
we define $\mathbb{D}^{k,p}$ to be the closure of $\mathcal{S}$ with respect to the
norm:
\[
\|F\|_{\mathbb{D}^{k,p}} := \left( \mathbb{E}|F|^p + \sum_{j=1}^k \mathbb{E} \left\| \D^{(j)}
F \right\|^p_{H_d^{\otimes j}} \right)^{1/p}.
\]

%Note that for $p = 2$ the following isometry holds  $L^p(\Omega \times
%[0,\infty)^k ; \mathbb{R}^d) \simeq L^2(\Omega ; H^{\otimes k})$. Hence
%one may identify $D^k F$ as a process: $D_{t_1,\ldots, t_k}^k F$.

%A random variable $F$ is said to be \textit{smooth in the Malliavin sense} if
%$F \in \mathbb{D}^{k,p}$ for all $p \geq 1$ and all $k \ge 1$. We 
%denote by $\mathbb{D}^{\infty}$ the set of all smooth 
%random variables in the Malliavin sense. 
%For example, the solution $X_t^x$ to equation (\ref{StrSDE}) satisfies $X^i_t \in
%\mathbb{D}^{k,p}$ for all $t \in [0,\infty)$ and $p \geq 1$ provided $V_0, \ldots,
%V_d \in \mathcal{C}_b^{\infty}(\mathbb{R}^N;\mathbb{R}^N)$ (see Theorem \ref{MallDiffSDE} below).

Moreover, there is nothing which restricts consideration to $\mathbb{R}^d$-valued
random variables. Indeed, one can consider more general
Hilbert space-valued random variables, and the theory would extend
in an appropriate way. To this end, denote $\mathbb{D}^{k,p}(E)$ to be the
appropriate space of $E$-valued random variables, where $E$ is some
separable Hilbert space. For more details, see
\cite{Nualart}, where also the proof of the following chain rule formula can be found: 

\begin{proposition}[Chain Rule for the Malliavin Derivative]
        
        If $\varphi : \mathbb{R}^m \rightarrow \mathbb{R}$ is a continuously
        differentiable function with bounded partial derivatives, and $F =
        (F_1, \ldots, F_m)$ is a random vector with components belonging to
        $\mathbb{D}^{1,p}$ for some $p\geq 1$. Then $\varphi(F) \in
        \mathbb{D}^{1,p}$, with
        \[
        \D\varphi(F) =  \nabla \varphi(F) \D F = \sum_{i=1}^m
        \partial^i \varphi(F) \, \D F_i,
        \]
        where $\nabla \varphi$ is the row vector 
        $(\partial^1 \varphi, \ldots, \partial^m
        \varphi)$ and $DF$ is the matrix 
        $(\D^j F_i)_{1 \leq i \leq m,1 \leq j \leq d}$.
\end{proposition}

\begin{lemma}[The Malliavin derivative and integration]\label{DiffInts}
        Consider an $\mathbb{F}$-adapted process
        $f:[0,T] \times \Omega \to \RR^{d}$, and suppose that for each
        $t \in [0,T]$ and $i \in \{0 \ldots, d\}$, we have $f_i(t) \in \mathbb{D}^{1,2}$. Moreover, suppose that:
        \begin{equation}\label{intcriteria}
        \mathbb{E} \int_0^T | f(t) |^2 \, dt < \infty   \quad\quad \mathbb{E} \int_0^T \|\D f(t)\|^2_{ H_d} \, dt < \infty.
        \end{equation}
        Then $F_t := \sum_{i=1}^d \int_0^t f_i(s) dB^i_s \in \mathbb{D}^{1,2}$, with
        \begin{align*}
                \D_r F_t = \left\{ f(r) + \sum_{i=1}^d \int_r^t \D_r f_i(s) \, dB^i_s \right\} \mathbf{1}_{ \{0 \leq r \leq t \} }.
        \end{align*}
        Similarly, for any $i \in \{1, \ldots, d \} $, $G^i_t:= \int_0^t f_i(s) ds$ is an element of  $\, \mathbb{D}^{1,2}$, with
        \begin{align*}
                \D_r G^i_t =  \left\{ \int_r^t \D_r f_i(s) \, ds \right\} \mathbf{1}_{ \{0 \leq r \leq t \} }.
        \end{align*}
\end{lemma}

\begin{proof} See Nualart \cite[Proposition 1.3.8]{Nualart} for details. \end{proof}

The divergence operator - which is the adjoint of the Malliavin derivative - plays a vital role in the construction of our integration by parts
formula. This operator is also called the Skorohod integral.
It coincides with a generalisation of the It\^{o} integral to
anticipating integrands. A detailed discussion of the divergence operator can be found in Nualart \cite{Nualart}.
\begin{definition}[Divergence operator]\label{divergenceoperator}
        Denote by $\delta$ the adjoint of the operator $\D$. That is,
        $\delta$ is an unbounded operator on $L^2(\Omega \times
        [0,T];\mathbb{R}^d)$ with values in $L^2(\Omega;\RR)$ such that:
        \begin{enumerate}
                \item Dom $\delta$  $= \{u \in L^2(\Omega \times
                [0,T];\mathbb{R}^d); |\mathbb{E}(\ip{\D F}{u}_{H_d})| \leq c
                \|F\|_{L^2(\Omega)},\ \  \forall F \in \mathbb{D}^{1,2}\}$.
                
                \item For every $u \in \textrm{Dom } \delta$, then $\delta(u) \in
                L^2(\Omega)$ satisfies:
                \[
                \mathbb{E}(F \delta(u)) = \mathbb{E}(\ip{\D F}{u}_{H_d}).
                \]
        \end{enumerate}
\end{definition}
%\noindent The following important results are shown in Section 1.5 of Nualart \cite{Nualart}:\\
%
%1.  $D$ is continuous from $\mathbb{D}^{k,p}(E)$ into $\mathbb{D}^{k-1,p}(H \otimes E)$
%
%2. $\ip{DF}{DG}_H \in \mathbb{D}^{\infty}$ if $F,G \in \mathbb{D}^{\infty}$
%
%3. $\delta$ is continuous from 
%$\mathbb{D}^{\infty}(H)$ into $\mathbb{D}^{\infty}$.
\begin{remark} \label{adjointito}If $u =(u^1,...,u^d)\in \textrm{Dom } \delta$ is $\mathbb{F}$-adapted, then the adjoint $\delta(u)$, is nothing more than the It\^{o} integral of $u$ with respect to the d-dimensional Brownian motion $B_t = (B_t^1,\ldots,B_t^d)$. i.e.
        \[
        \delta(u)= \sum_{i=1}^d \int_0^T u^i( s) \, dB^i_s.
        \]
\end{remark}

\subsection{Proofs from Section \ref{sec:regsoln} }
\label{sec:KSpfs}

The first goal of this section is to prove Theorem \ref{th:XisKSP}. Since each type  of derivative (w.r.t. $x$, $\mu$ or $v$) of $\Xd_t$ satisfies a linear equation, we will introduce a general linear equation and, first, derive some a priori $L^p$ estimates on the solution. Then, we will show this linear equation is again differentiable under certain assumptions on the coefficients.
In the following, we consider an equation with coefficients  $a_1, a_2, a_3$, which depend on $(t,x,[\theta],\bv) \in [0,T] \times \RR^N \times \p_2(\RR^N) \times (\RR^N)^{\# \bv} $ with initial condition given by a constant value $a_0$ \footnote{When applying Lemma \ref{lem:linearestimate} to control the derivatives of $\Xd_t$, $a_0$ will be either 1 in the case of the $\partial_{x_i}\Xd_t$ or $0$ in all other cases.} Below, we denote $v_r$ as one element of the tuple $\bv=(v_1, \ldots, v_{\# \bv})$.

\begin{lemma}
        \label{lem:linearestimate}
        Let $Y^{x, [\theta]}(\bv)$ solve the following SDE
        \begin{align}
                \label{eq:LinearY}
                \begin{split}
                        Y^{x, [\theta]}_t(\bv) = a_0 + \sum_{i=0}^d \int_0^t \bigg\{ & a^i_1(s,x,[\theta]) \, Y^{x, [\theta]}_s(\bv) + a^i_2(s,x,[\theta],\bv) \\
                        &  + \widetilde{\EE} \left[  a^i_3(s,x,[\theta],\t \theta) \, \widetilde{Y}^{\t \theta, [\theta]}_s(\bv) + \sum_{r=1}^{\# \bv}  a^i_3(s,x,[\theta],v_r) \, \widetilde{Y}^{v_r, [\theta]}_s(\bv) \right] \bigg\} \, dB^i_s,
                \end{split}
        \end{align}
        where, for all $i =1 , \ldots, d$, the coefficients $(t,x,[\theta],\bv) \mapsto a_k(t,x,[\theta],\bv)$ are continuous in $L^p(\Omega)$ $\forall p \geq 1$, $k=1,2,3$, and
        \begin{align*}
                & a_0 \in \RR^N,\footnote{when Lemma \ref{lem:linearestimate} to control the derivatives of $\Xd_t$, $a_0$ will be either 1 in the case of the $\partial{x_i}\Xd_t$ and $0$ in all other cases.}\\
                 & a^i_1  : \Omega \times [0,T] \times \RR^N \times \p_2(\RR^N) \to \RR^{N \times N}\\
                & a_2^i : \Omega \times [0,T] \times \RR^N \times \p_2(\RR^N) \times (\RR^N)^{\# \bv} \to \RR^N, \\
                & a_3^i   : \widetilde{\Omega} \times \Omega \times [0,T] \times \RR^N \times \p_2(\RR^N) \times \RR^N \to \RR^{N \times N}.
        \end{align*}
In \eqref{eq:LinearY}, $\widetilde{Y}^{\t \theta, [\theta]}$ is a copy of $Y^{x, [\theta]}$ on the probability space $(\t \Omega, \t \F, \t \PP)$ driven by the Brownian motion $\t B$ and with $x=\t \theta.$ Similarly,  $\widetilde{Y}^{v_r, [\theta]}$ is a copy of $Y^{x, [\theta]}$ on the probability space $(\t \Omega, \t \F, \t \PP)$ driven by the Brownian motion $\t B$ and with $x=v.$
        If we make the following boundedness assumptions
        \begin{enumerate}
                \item $\sup_{x \in \RR^N, [\theta] \in \p_2(\RR^N), \bv \in (\RR^N)^{\# \bv}}   \|a_2(\cdot,x,[\theta],\bv) \|_{\cS^p_T} < \infty$,
                \item $a_1$ and $a_3$ are uniformly bounded,
                \item $\sup_{x \in \RR^N, [\theta] \in \p_2(\RR^N), \bv \in (\RR^N)^{\# \bv}}     \|a_2(\cdot,\theta,[\theta],\bv) \|_{\cS^2_T} < \infty$,
                %\item $a_3$ is uniformly bounded.
                %\item $\sup_{t \in [0,T],x \in \RR^N, [\theta] \in \p_2(\RR^N), v \in \RR^N} \EE \left( \|a_3(t,x,[\theta],v)\|^p_{L^2(\t \Omega)} \right)< \infty$
        \end{enumerate}
        %       We also assume that each of the above bounds holds for $p=2$ when we replace $x$ by $\theta$ in each coefficient.
        then we have the following estimate for $C=C(p,T,a_1,a_3)$
        \begin{align}
                \label{eq:linbd}
                \begin{split}
                        \left\|Y^{x,[\theta]} (\bv)\right\|_{\cS^p_T} \leq & C \left( \left|a_0\right|+  \, \left\| a_2 (\cdot,x,[\theta],\bv) \right\|_{\cS^p_T}   +  \, \left\| a_2(\cdot,\theta,[\theta],\bv) \right\|_{\cS^2_T} \right)   .
                \end{split}
        \end{align}
        Moreover, we also get that the mapping
        \[
        [0,T] \times \RR^N \times \p_2(\RR^N) \times (\RR^N)^{\# \bv} \ni (t,x,[\theta],v) \mapsto Y^{x, [\theta]}_t(\bv) \in L^p(\Omega)
        \]
        is continuous.
        
\end{lemma}
\begin{proof}
        Wherever there is no confusion, we drop the arguments $(t,x,[\theta],\bv)$ to lighten notation. We will write, for example, $a_3|_{v=\t \theta}$ to denote $a_3(s,x,[\theta],\t \theta)$.
Let $\iota,\kappa:[0,T]\mapsto [0,\infty)$ be defined as 
\begin{eqnarray*}
\iota(t)&=&\left\|Y^{\t\theta,[\theta]} (\bv)\right\|_{\cS^2_t}^2+\sum_{r=1}^{\# \bv}  \left\|Y^{v_{r},[\theta]} (\bv)\right\|_{\cS^2_t}^2,\ \ t\in [0,T]\\
\kappa(t)&=&\left\|Y^{x,[\theta]} (\bv)\right\|_{\cS^p_t}^p,\ \ t\in [0,T].\end{eqnarray*}
We deduce from \eqref{eq:LinearY} and Burkholder-Davis-Gundy inequality  that there exists a constant $C$ such that for any $t\in [0,T]$ we have   
 \[
                \iota(t) \leq   C \bigg\{ |a_0|^2 +\|a_2(\cdot,\theta,[\theta],\bv) \|_{\cS^2_t} ^{2}+  \int_0^t  (\|a_1\|_{\infty}^2+\|a_3\|_{\infty}^2)\iota(s) ds \bigg\},
        \]
        so by Gronwall's inequality,
        \begin{equation*}
        \left\|Y^{\t\theta,[\theta]} (\bv)\right\|_{\cS^2_t}^2+\sum_{r=1}^{\# \bv}  \left\|Y^{v_{r},[\theta]} (\bv)\right\|_{\cS^2_t}^2\le
 C e^{\|a_1\|_{\infty}^2+\|a_3\|_{\infty}^2)T}(|a_0|^2 +\|a_2(\cdot,\theta,[\theta],\bv) \|_{\cS^2_T}  ^{2}).
       \end{equation*}
        Then, applying the Burkholder-Davis-Gundy inequality and the above estimate to $Y^{x,[\theta]}_t(\bv)$ we deduce that 
        \[
              \kappa(t) \le\  C \left( |a_0|^p+ \|a_3\|_\infty^p \iota(T)^{p\over 2}+ \|a_2(\cdot,x,[\theta],\bv) \|_{\cS^p_T}^p+\int_0^T \|a_1\|^p_{\infty} \kappa(s)ds\right).   \]
        So applying Gronwall's inequality again and our estimate on $\iota(T)$ we get \eqref{eq:linbd}. 

Now, for a quantity $G$ depending on $(t,x,[\theta],\bv)$ we introduce the notation 
        \begin{align*}
                & \Delta_t G :=G(t,x,[\theta],\bv) - G(t',x,[\theta],\bv) \\
                &  \Delta_x G :=G(t,x,[\theta],\bv) - G(t,x',[\theta],\bv) \\
                & \Delta_{\theta} G := G(t,x,[\theta],\bv) - G(t,x,[\theta'],\bv) \\
                & \Delta_{\bv} G  := G(t,x,[\theta],\bv) - G(t,x,[\theta],\bv') .
        \end{align*}
        We can split the difference $Y^{x, [\theta]}_t(\bv) - Y^{x', [\theta']}_{t'}(\bv')$ into
        \begin{align*}
                Y^{x, [\theta]}_t(\bv) - Y^{x', [\theta']}_{t'}(\bv') = 
%               \left( Y^{x, [\theta]}_{t}(\bv)  - Y^{x, [\theta]}_{t'}(\bv)  \right) + \left( Y^{x, [\theta]}_{t'}(\bv)  - Y^{x', [\theta]}_{t'}(\bv)  \right) \\
%               & \quad + \left( Y^{x', [\theta]}_{t'}(\bv)  - Y^{x', [\theta']}_{t'}(\bv)  \right) 
%               + \left( Y^{x', [\theta']}_{t'}(\bv) - Y^{x', [\theta']}_{t'}(\bv' \right) \\
                 \Delta_t Y^{x, [\theta]}(\bv) + \Delta_x Y_{t'}^{\theta}(\bv) + \Delta_{\theta} Y_{t'}^{x'}(\bv) + \Delta_{\bv} Y_{t'}^{x',[\theta']},
        \end{align*}
        and consider each term individually. First, 
        \begin{align*}  
                \Delta_t Y^{x, [\theta]}(\bv) & =  \sum_{i=0}^d   \int_{t'}^t \bigg\{  a^i_1 \, Y^{x, [\theta]}_s(\bv) + a^i_2
                + \widetilde{\EE} \left[  a^i_3|_{v= \t \theta} \, \widetilde{Y}^{\t \theta, [\theta]}_s(\bv) + \sum_{r=1}^{\# \bv}  a^i_3|_{v= v_r} \, \widetilde{Y}^{v_r, [\theta]}_s(\bv) \right] \bigg\} \, dB^i_s
     .   \end{align*}
        The integrand is bounded in $L^p(\Omega)$ uniformly in time, so using the Burkholder-Davis-Gundy inequality, we get
        \[
        \left \| \Delta_t Y^{x, [\theta]}(\bv) \right \|_p  \leq C \left( |t-t'|^{\frac{1}{2}} \right).
        \] 
        Using the continuity assumption on $a_0$, we see that this goes to 0 as $t \to t'$.
        Second,
        \begin{align*}
                \Delta_x Y_t^{\theta}(\bv) = \Delta_x a_0 + \sum_{i=0}^d \int_0^t \bigg\{ & a^i_1 \Delta_x Y_s^{[\theta]}(\bv) + Y_s^{x,[\theta]}(\bv) \Delta_x a_1^i  + \Delta_x a_2^i \\
                &+  \widetilde{\EE} \left[ \Delta_x a^i_3|_{v= \t \theta} \, \widetilde{Y}^{\t \theta, [\theta]}_s(\bv) + \sum_{r=1}^{\# \bv}  \Delta_xa^i_3|_{v= v_r} \, \widetilde{Y}^{v_r, [\theta]}_s(\bv) \right] \bigg\} dB^i_s
     .   \end{align*}
        This is again a linear equation. The same argument used to obtain \eqref{eq:linbd}, except using the $L^p$-norm instead of the $\cS^p_T$-norm, gives
        \begin{align*}
                & \left \| \Delta_x Y_t^{\theta}(\bv) \right \|^p_p  \leq C \,  \sup_{s \in [0,t]} \EE \left(   Y_s^{x,[\theta]}(\bv) \Delta_x a_1^i + \Delta_x a_2^i + \widetilde{\EE} \left[  \widetilde{Y}_s^{\t \theta,[\theta]}(\bv) \left. \Delta_x a_3^i\right|_{v=\t \theta} + \sum_{r=1}^{\# \bv}  \Delta_xa^i_3|_{v=v_r} \, \widetilde{Y}^{v_r, [\theta]}_s(\bv) \right]   \right)^p.
        \end{align*}
        Then, using H\"older's inequality, the fact that $Y_s^{x,[\theta]} (\bv)$ is bounded in $L^p(\Omega)$ for all $p \geq 1$ and the continuity assumptions on $a_1, a_2, a_3$, we see that the above quantity goes to 0. The arguments for $ \Delta_{\theta} Y_{t'}^{x'}(\bv)$ and $\Delta_{\bv} Y_{t'}^{x',[\theta']}$ are almost identical.
\end{proof}

Now, we consider the differentiability of the generic process $Y^{x,[\theta]}(\bv)$ satisfying the linear equation \eqref{eq:LinearY} under appropriate assumptions.
\begin{proposition}
        \label{prop:linearderivs}
        Suppose that the process $Y^{x,[\theta]}(\bv)$ is as in Lemma \ref{lem:linearestimate}.
        In addition to the assumptions of Lemma \ref{lem:linearestimate}, 
        we introduce the following differentiability assumptions:
        \begin{enumerate}
                \item[(a)] For  $k=1,2,3$, all $(s,[\theta],\bv) \in [0,T]\times \p_2(\RR^N)\times (\RR^N)^{\# \bv}$ and each $p \geq 1$, $\RR^N \ni x \mapsto a_k(s,x,[\theta],\bv) \in L^p(\Omega)$ is differentiable.
                
                \item[(b)] For $k=1,2,3$, all $(s,[\theta],x) \in [0,T]\times \p_2(\RR^N)\times (\RR^N)^{\# \bv}$ and each $p \geq 1$, $\RR^N \ni v \mapsto a_k(s,x,[\theta],\bv) \in L^p(\Omega)$ is differentiable.
                
                %\item[(c)] For all $(s,v) \in [0,T] \times \RR^N$  $L^2(\Omega) \ni \theta \mapsto a_k(s,\theta,[\theta],v) \in L^2(\Omega)$ is Fr\'echet differentiable for $k=0,1,2,3$.
                
                \item[(c)] For  all $(s,x,\bv) \in [0,T] \times \RR^N\times(\RR^N)^{\# \bv}$ the mapping  $L^2(\Omega) \ni \theta \mapsto a_2(s,\theta,[\theta],\bv) \in L^2(\Omega)$ is Fr\'echet differentiable. 
                
                \item[(d)] $a_k(s,x,[\theta],\bv) \in \DD^{1, \infty}$ for $k=1,2,3$ and all $(s,x,[\theta],\bv) \in [0,T] \times \RR^N \times \p_2(\RR^N)\times \RR^N$.
                %Similarly, we assume $a_3(s,x,[\theta],\bv)(\widetilde\omega) \in \DD^{1, \infty}$ $\widetilde{\PP}$-a.s.
                Moreover, we assume the following estimates on the Malliavin derivatives hold.
                \begin{align*}
                        & \sup_{r \in [0,T]} \EE \sup_{s\in [0,T]} | \D_r a_k(s,x,[\theta],\bv) |^p < \infty, \quad k=0,1,2,3.
                        %& \sup_{r \in [0,T]} \EE \sup_{s\in [0,T]} | D_r a_2(s,x,[\theta],\bv) |^p < \infty \\
                        %& \sup_{r \in [0,T]} \EE \sup_{s\in [0,T]} | D_r a_3(s,x,[\theta],\bv) |_{L^2(\t \Omega)} < \infty
                \end{align*}
        \end{enumerate}
        Then, for all $t \in [0,T]$ the following hold:
        \begin{enumerate}
                \item Under assumption (a), $x \mapsto Y^{x, [\theta]}_t(\bv)$ is differentiable in $L^p(\Omega)$ for all $p \geq 1$ and 
                \[\partial_x Y_t^{x, [\theta]}(\bv):= L^p - \lim_{h \to 0} \frac{1}{|h|}  \left(Y_t^{x+h, [\theta]}(\bv) - Y_t^{x, [\theta]}(\bv) \right)
                \]
                satisfies
                \begin{align*}
                        \partial_x Y^{x, [\theta]}_t(\bv) = \sum_{i=0}^d \int_0^t \bigg\{ & \partial_x a^i_1 \, Y^{x, [\theta]}_s(\bv) +  a^i_1 \, \partial_x Y^{x, [\theta]}_s (\bv) + \partial_x a^i_2 \\
                        & 
                        + \widetilde{\EE} \left[ \left.\partial_x a^i_3\right|_{v=\t \theta} \widetilde{Y}^{\t \theta, [\theta]}_s(\bv) + \sum_{r=1}^{\# \bv} \left.\partial_x a^i_3\right|_{v=v_r} \widetilde{Y}^{v_r, [\theta]}_s(\bv) \right] \bigg\} \, dB^i_s.
                \end{align*}
                
                \item Under assumption (b), $\bv \mapsto Y^{x, [\theta]}_t(\bv)$ is differentiable in $L^p(\Omega)$ for all $p \geq 1$ and \[
                \partial_{\bv} Y_t^{x, [\theta]}(\bv):= L^p - \lim_{h \to 0} \frac{1}{|h|}  \left(Y_t^{x, [\theta]}(\bv+h) - Y_t^{x, [\theta]}(\bv)\right)
                \]
                satisfies
                \begin{align*}
                        \partial_{v_j}  Y^{x, [\theta]}_t(\bv) =  \sum_{i=0}^d  \int_0^t \bigg\{  & a^i_1 \, \partial_{v_j} Y^{x, [\theta]}_s(\bv) 
                        + \partial_{v_j} a^i_2 
                        + \widetilde{\EE} \left[  \left. \partial_v a^i_3\right|_{v=v_j}  \widetilde{Y}^{v_j, [\theta]}_s (\bv) \right] \\
                        &  + \widetilde{\EE} \left[ \left. a^i_3\right|_{v=v_j} \partial_x \widetilde{Y}^{v_j, [\theta]}_s (\bv) +  \left.  a^i_3\right|_{v=\t \theta} \partial_{v_j} \widetilde{Y}^{\t \theta, [\theta]}_s (\bv) +\sum_{r=1}^{\# \bv}\left. a^i_3\right|_{v=v_r} \partial_{v_j} \widetilde{Y}^{v_r, [\theta]}_s (\bv) \right] \bigg\} \, dB^i_s.
                \end{align*}
                
                %\item Under assumption (c), $\theta \mapsto Y^{\theta, [\theta]}_t(v)$ is Fr\'echet differentiable.
                %Moreover, when assumptions (a) and (b) also hold, $L^2(\Omega)\ni \theta \mapsto Y_t^{\theta,[\theta]} \in L^2(\Omega)$, is Fr\'echet differentiable with
                %\[
                %D \left( Y_t^{\theta,[\theta]}\right) (\gamma)= \left. \left(\partial_xY_t^{x,[\theta]} \gamma + \widehat{\EE} \left[\partial_{\mu}Y_t^{x,[\theta]}(\widehat{\theta}) \, \widehat{\gamma} \right]\right) \right|_{x=\theta}.
                %\]
                
                \item Under assumption (a), (b) and (c), the maps $\theta \mapsto Y_t^{\theta,[\theta]}(\bv)$ and $\theta \mapsto Y^{x, [\theta]}_t(\bv)$ are Fr\'echet differentiable for all $(x,\bv) \in \RR^N \times (\RR^N)^{\# \bv}$, so $\partial_{\mu}Y^{x,[\theta]}_t(\bv)$ exists and it satisfies
                \begin{align*}
                        \partial_{\mu}  Y^{x, [\theta]}_t (\bv,v') =  \sum_{i=0}^d  \int_0^t \bigg\{ & \partial_{\mu} a^i_1 \: Y^{x, [\theta]}_s(\bv) +  a^i_1 \, \partial_{\mu} Y^{x, [\theta]}_s (\bv,v') 
                        + \partial_{\mu} a^i_2  \\
                        & + \widetilde{\EE} \left[ \partial_{\mu} a^i_3  \: \widetilde{Y}^{\t \theta, [\theta]}_s(\bv) + \partial_{v} a^i_3  \: \widetilde{Y}^{v', [\theta]}_s(\bv) + \left. a^i_3 \right|_{v=\t \theta} \: \partial_{\mu} \widetilde{Y}^{\t \theta, [\theta]}_s (\bv,v') \right]\\
                        & +   \widetilde{\EE} \left[  \left. a^i_3 \right|_{v=v'} \: \partial_{x} \widetilde{Y}^{v', [\theta]}_s (\bv)  + \sum_{r=1}^{\# \bv}   \left. a^i_3 \right|_{v=v_r} \: \partial_{\mu} \widetilde{Y}^{v_r, [\theta]}_s (\bv,v') \right] \bigg\} \, dB^i_s.
                \end{align*}
                % a^i_3(s,x, [\theta],v_1)  \partial_x\widetilde{Y}^{v_1, [\theta]}_s (v_1) +
                Moreover,  we have the representation, for all $\gamma \in L^2(\Omega)$,
                \[
                D \left( Y_t^{\theta,[\theta]}(\bv)\right) (\gamma)= \left. \left(\partial_xY_t^{x,[\theta]}(\bv) \gamma + \widehat{\EE} \left[\partial_{\mu}Y_t^{x,[\theta]}(\bv,\widehat{\theta}) \, \widehat{\gamma} \right]\right) \right|_{x=\theta}.
                \]

                \item Under assumption (e),
                $Y^{x, [\theta]}_t \in \DD^{1, \infty}$ and $\D_r Y^{x, [\theta]}$ satisfies 
                \begin{align}
                        \label{eq:malDerivLinear} 
                        \begin{split}
                                \D_r Y^{x, [\theta]}_t(\bv)  &= \left(a^j_1 \, Y^{x, [\theta]}_r(v) + a^j_2  + \widetilde{\EE} \left[ a^j_3 \widetilde{Y}^{x, [\theta]}_r(\bv) \right]\right)_{j=1, \ldots, d} \\
                                & + \sum_{i=0}^d \int_r^t \bigg\{  \D_r a^i_1 \, Y^{x, [\theta]}_s(\bv) +  a^i_1 \, \D_r Y^{x, [\theta]}_s(\bv) 
                                + \D_r a^i_2   + \widetilde{\EE} \left[ \D_r a^i_3|_{v=\t \theta} \widetilde{Y}^{x, [\theta]}_s(\bv) \right] \bigg\} \, dB^i_s.
                        \end{split}
                \end{align}
                Moreover, the following bound holds:
                \begin{align}
                        \label{eq:linearMalbd}
                        \begin{split}
                                \sup_{r \leq t} \EE \left[ \sup_{r \leq t \leq T} \left| \D_r Y_t^{x, [\theta]}(\bv) \right|^p \right]
                                \leq C \, & \sup_{r \leq t} \EE \left[ \sup_{r \leq t \leq T}  \left| \D_r a_1\right|^p  \right].
                        \end{split}
                \end{align}
                
        \end{enumerate}
\end{proposition}

%\begin{remark}
%Note that if $Y_t^{x,[\theta]}$ represents, for example, $\partial_{\mu} X_t^{x,[\theta]}(v)$, then part 1. also guarantees the existence of $\partial_v \partial_{\mu} X_t^{x,[\theta]}(v)$ if we replace all references to $x$ with $v$ in the statement.
%\end{remark}

\begin{proof} Parts 1. and 2. are standard results on differentiability of SDEs with respect to a real parameter.
        \begin{enumerate}
                %\item This proof is essentially the same as proving that $\theta \mapsto X_t^{\theta}(\gamma)$ is Fr\'echet differentiable $L^2(\Omega)$.
                
                \item[3.]
                The arguments to show that the maps $\theta \mapsto Y_t^{\theta,[\theta]}(\bv)$ and $\theta \mapsto Y^{x, [\theta]}_t(\bv)$ are Fr\'echet differentiable are essentially the same as those from Proposition \ref{prop:firstderivs} showing that $\theta \mapsto X^{\theta,[\theta]}(\bv)$ and $\theta \mapsto X^{x, [\theta]}_t(\bv)$ are Fr\'echet differentiable, so we omit them.
                
                Once we know these derivatives exist, it is fairly straightforward to see that they satisfy the equations
                \begin{align}
                        \begin{split}
                                D (Y_t^{\theta,[\theta]}(\bv))(\gamma) = \int_0^t & \bigg\{ D a_1(\gamma)|_{x=\theta} \:Y^{\theta, [\theta]}_s(\bv) +\partial_x a_1|_{x=\theta} \, \gamma \, Y^{\theta, [\theta]}_s(\bv)  
                                + a_1|_{x=\theta} D (Y_s^{\theta,[\theta]}(\bv))(\gamma) 
                                +  D a_2(\gamma)|_{x=\theta} \\ 
                                &+ \partial_x a_2|_{x=\theta} \, \gamma 
                                + \widetilde{\EE} \left[ \partial_v a_3|_{x=\theta,v=\t \theta} \,  \widetilde{\gamma} \, \widetilde{Y}^{\t \theta, [\theta]}_s(\bv) \right] + \widetilde{\EE} \left[ \partial_x a_3|_{x=\theta,v=\t \theta} \,  \widetilde{\gamma} \, \widetilde{Y}^{\t \theta, [\theta]}_s(\bv) \right] \\
                                & + \widetilde{\EE} \left[ D a_3(\gamma)|_{x=\theta,v=\t \theta} \, \widetilde{Y}^{\t \theta, [\theta]}_s(\bv) + a_3|_{x=\theta,v=\t \theta} \: D (\widetilde{Y}_s^{\t \theta,[\theta]}(\bv))(\gamma) \right]\\
                                & + \sum_{r=1}^{\# \bv} \widetilde{\EE} \left[  \left( D a^i_3(\gamma)|_{x=\theta,v=v_r} \,  + \partial_x a^i_3|_{x=\theta,v=v_r} \, \gamma \right) \widetilde{Y}^{v_r, [\theta]}_s(\bv) \right] \\
                                & + \widetilde{\EE} \left[  a^i_3(\gamma)|_{x=\theta,v=v_r}\, D (\widetilde{Y}^{v_r, [\theta]}_s(\bv)) (  \gamma)\right]  \bigg\} \, dB_s,
                        \end{split}
                \end{align}
                and
                \begin{align}
                        \begin{split}
                                D (Y_t^{x,[\theta]}(\bv))(\gamma) = \int_0^t  \bigg\{ & D a_1(\gamma) Y^{x, [\theta]}_s(\bv) + a_1 D (Y_s^{x,[\theta]}(\bv))(\gamma) 
                                +  D a_2(\gamma) + \widetilde{\EE} \left[ \partial_v a_3|_{v=\t \theta} \, \widetilde{\gamma} \, \widetilde{Y}^{\t \theta, [\theta]}_s(\bv) \right] \\
                                & + \widetilde{\EE} \left[ D a_3(\gamma) |_{v=\t \theta} \, \widetilde{Y}^{\t \theta, [\theta]}_s(\bv) + a_3|_{v=\t \theta} D (\widetilde{Y}_s^{\t \theta,[\theta]}(\bv))(\gamma) \right] \\
                                & + \sum_{r=1}^{\# \bv} \widetilde{\EE} \left[  D a^i_3(\gamma)|_{v=v_r} \,   \widetilde{Y}^{v_r, [\theta]}_s(\bv) + a^i_3(\gamma)|_{v=v_r}\, D (\widetilde{Y}^{v_r, [\theta]}_s(\bv)) (  \gamma) \right]  \bigg\} \, dB_s.
                        \end{split}
                \end{align}
                
                Now, taking the equation we claim is satisfied by $\partial_{\mu} Y^{x,[\theta]}_t(\bv,v')$, evaluating at $v'=\widehat{\theta}$, multiplying by $\widehat{\gamma}$, and taking expectation with respect to $\widehat{\PP}$, we can see that $\widehat{\EE}\left[ \partial_{\mu} Y^{x,[\theta]}_t(\bv,\widehat{\theta}) \widehat{\gamma}\right]$ satisfies the same equation as $D(Y_t^{x,\theta}(\bv))(\gamma)$, so by uniqueness, they are the same.    
                Similarly, computing 
                \[
                \left. \left(\partial_xY_t^{x,[\theta]}(\bv) \gamma + \widehat{\EE} \left[\partial_{\mu}Y_t^{x,[\theta]}(\bv,\widehat{\theta}) \, \widehat{\gamma} \right]\right) \right|_{x=\theta},
                \] we can see that it satisfies the same equation as $D \left( Y_t^{\theta,[\theta]}(\bv)\right) (\gamma)$.
                
                \item[4.] Equation \eqref{eq:LinearY}, fits into the standard framework for Malliavin differentiability of SDEs, since the only unkown term appearing inside the expectation with respect to $\widetilde{\PP}$ on the right hand side is $\widetilde{Y}_s^{\t \theta, [\theta]}$ does not depend on $\omega \in \Omega$. The conclusion is therefore a standard result \cite[Lemma 2.2.2]{Nualart}. The proof of the bound \eqref{eq:linearMalbd} is along the same lines as the proof of \eqref{eq:linbd}.
        \end{enumerate}
\end{proof}
We are now in a position to prove Theorem \ref{th:XisKSP}.

\begin{proof}[Proof of Theorem \ref{th:XisKSP}:]
                To ease the burden on notation, we will prove the theorem for dimension $N=1$. In this case, $\alpha$ and $\gamma$ are integers rather than multi-indices and $\bb$ is a multi-index on $ \{ 1, \ldots, \alpha \}$. We will show, by induction on $I := \alpha + |\bb|+ \gamma$, that $ \partial^{\gamma}_x \partial^{\bb}_{\bv} \partial^{\alpha}_{\mu} \Xd_t$ exists and  solves a linear equation of the form \eqref{eq:LinearY}.
        We can then use Lemma \ref{lem:linearestimate} to obtain an $L^p(\Omega)$ estimate on $ \partial^{\gamma}_x \partial^{\bb}_{\bv} \partial^{\alpha}_{\mu} \Xd_t$ at each level. In addition, we can obtain estimates on the $\DD^{m,p}$-norm of $ \partial^{\gamma}_x \partial^{\bb}_{\bv} \partial^{\alpha}_{\mu} \Xd_t$ at each level using arguments similar to the classical SDE case.
%       The inductive step is made complicated by the fact that for derivatives at level $I+1$, we must write the coefficients of the linear equation they satisfy in terms of the coefficients \emph{not only} at level $I$, but also some derivatives on the same level $I+1$. For example, the equation satisfied by $\partial_{\mu} \Xd_t(v)$ contains  $\partial_v X^{v,[\theta]}_s$, so to prove that $\partial_v\partial_{\mu} \Xd_t(v)$ exists, we need to first prove that  $\partial_v^2 X^{v,[\theta]}_s$ exists. This is why the order in which we prove that derivatives at each level exist matters.
        
% We denote the dependence on $\alpha, \bb, \gamma$ of the coefficients in this equation by
%       \[
%       a_k^i = a_k^{i,\alpha,\bb,\gamma}, \quad k=0, \ldots, 3, \: i=0, \ldots, d.
%       \]
%       %In doing so, we will also need to prove that $\partial^{\eta}_{\bv} \partial^{\xi}_\mu U_t^{\theta}(v)$ exists for all $\eta, \xi$ with $|\eta| + |\xi| \leq k-1$ where $U_t^{\theta}(v)$ is the solution of equation \eqref{eq:Utheta}. We will show that $\partial^{\eta}_{\bv} \partial^{\xi}_\mu U_t^{\theta}(v)$ satisfies a linear equation and we will label the coefficients
%       %\[
%       %       a_k^i = b_k^{i,\eta,\xi}, \quad k=0, \ldots, 3, \: i=0, \ldots, d.
%       %\]
        We will prove by induction that the following statements hold true for  $I=1, \ldots, k$:
        
        \begin{itemize}
                \item[(S1):] For all $\alpha, \bb, \gamma$ satisfying $\alpha + |\bb| + \gamma=I$, $\partial^{\gamma}_x \partial^{\bb}_{\bv} \partial^{\alpha}_{\mu} \Xd_t(\bv)$ exists and solves a linear equation of the form \eqref{eq:LinearY}. Moreover, $\| \partial^{\gamma}_x \partial^{\bb}_{\bv} \partial^{\alpha}_{\mu} \Xd (\bv) \|_{\cS^p_T}$ is bounded independently of $(x,[\theta],\bv)$ for all $p \geq 1$.

                \item[(S2):] $ \partial^{\gamma}_x \partial^{\bb}_{\bv} \partial^{\alpha}_{\mu} \Xd_t(\bv) \in \DD^{M-I,\infty}$ and, moreover,
                \[
                \sup_{r_1, \ldots, r_{M-I-1} \in [0,T]} \EE \left[ \sup_{r_1 \vee \ldots \vee r_{M-I-1} \leq t \leq T} \left| \D^{(M-I-1)}_{r_1, \ldots, r_{M-I-1}} \, \partial^{\gamma}_x \partial^{\bb}_{\bv} \partial^{\alpha}_{\mu} \Xd_t (\bv) \right|^p \right] \leq C \,(1+|x|+\|\theta\|_2)^m,
                \]
                for all $p \geq 1$, where $m=1$ unless the coefficients $\Vs$ are bounded, in which case $m=0$.
        \end{itemize}
        
        \underline{$I=1$:}
        
         (S1):  $\partial_x \Xd_t$      and 
           $\partial_{\mu} \Xd_t(v_1)$ exists and are continuous by Proposition \ref{prop:firstderivs}. There is no derivative with respect to $v$ at this level. We can write 
           \[
           Y^{x,[\theta]}_t (v_1):= \begin{pmatrix}
           \partial_x \Xd_t \\ \partial_{\mu} \Xd_t(v_1)
           \end{pmatrix}
           \]
           in the form of equation \eqref{eq:LinearY} and identify the coeffcients:
                \begin{align*}
                        & a_0    = \begin{pmatrix}
                        1 \\  0
                        \end{pmatrix}\\
                        & a_1^{i}(s,x,[\theta]) =\begin{pmatrix}\partial V_i(\Xd_s,[\Xt_s]) & 0\\ 0 & \partial V_i(\Xd_s,[\Xt_s])        \end{pmatrix}  \\
                        & a_2^{i}(s,x,[\theta],v_1) = \begin{pmatrix}
                        0 \\  0
                        \end{pmatrix}\\
                        &a_3^{i}(s,x,[\theta],v) =\begin{pmatrix}0 & 0\\ \partial_{\mu} V_i\left(\Xd_s,[\Xt_s],\widetilde{X}^{v, [\theta]}_s \right) \mathbf{1}_{v=v_1} & \partial_{\mu} V_i\left(\Xd_s,[\Xt_s],\widetilde{X}^{v, [\theta]}_s \right) \mathbf{1}_{v \neq v_1}     \end{pmatrix}
     .           \end{align*} 
        We can now check that the assumptions of Lemma \ref{lem:linearestimate} are satisfied by the coefficients $a_1, a_2, a_3$ above to obtain a bound on  $\| Y^{x,[\theta]} (v_1) \|_{\cS^p_T}$. 
                
                %\item[(S2):] $D \Xt_t(\gamma)$ exists and is continuous by Proposition \ref{prop:firstderivs}.
                %We also have the required estimates by Corollary \ref{cor:derivestimate}.

        Going back to the equations satisfied by  $\partial_x \Xd_t$ and $ \partial_{\mu} \Xd_t(v)$, we see that the coefficients are $(k-1)$-times differentiable with bounded Lipschitz derivatives. Nualart \cite[Theorem 2.2.2]{Nualart} immediately tells us that $\partial_x \Xd_t,  \partial_{\mu} \Xd_t \in \DD^{k-1,\infty}$. Using the bound in \eqref{eq:linearMalbd}, we get for $Y^{x,[\theta]}_t = \partial_x \Xd_t$ or $ \partial_{\mu} \Xd_t(v)$,
                \begin{align*}
                        \sup_{r \leq t} \EE \left[ \sup_{r \leq t \leq T} \left| \D_r Y_t^{x, [\theta]}(\bv) \right|^p \right]
                        &\leq C \,  \sup_{r \leq t} \EE \left[ \sup_{r \leq t \leq T} \left| \D_r a_1(s,x,[\theta]) \right|^p \right] \\
                        &\leq C \,  \sup_{r \leq t} \EE \left[ \sup_{r \leq t \leq T} \left| \partial^2 V_i(\Xd_s, [\Xt_s]) \D_r \Xd_s \right|^p \right].
                \end{align*}
                Now, $\partial^2 V_i$ is bounded and it is easy to prove that
\[
\sup_{t  \in [0,T]} \EE \left| \sup_{r \in [0,T]} \D_r X_t^{x, [\theta]} \right|^p \leq C \,(1 +|x| + \|\theta\|_2 )^{mp},       
\]      
(where $m=1$ unless the coefficients $\Vs$ are bounded, in which case $m=0$) using a similar argument to deriving the bound \eqref{eq:linbd} for the solution of a linear equation.
                So, we get the required bound on the first Malliavin derivative of $Y^{x,[\theta]}(\bv)$. For the higher order Malliavin derivatives, following the proof in \cite[Theorem 2.2.2]{Nualart}, we see that each order Malliavin derivative satisfies a linear equation. Importantly in the equation satisfied by higher-order Malliavin derivatives, the coefficient $a_1^i$ in each equation is always $\partial V_i(\Xd_s, [\Xt_s])$. From the bound on the Malliavin derivative of a general linear equation in \eqref{eq:linearMalbd}, we see that this is the only term which contribute to the estimate. Hence, the same bound holds as above for each different order Malliavin derivative. Moreover, if all of the coefficients are bounded, the estimate is uniform in $(x,[\theta],\bv)$.

        \noindent \underline{$2 \leq I \leq k$:}
        
(S1): By the induction hypothesis, for any $\alpha, \bb, \gamma$ satisfying $\alpha + |\bb| + \gamma = I$, we can write $Y^{x,[\theta]}_t(\bv):=\partial^{\gamma}_x \partial^{\bb}_{\bv} \partial^{\alpha}_{\mu} \Xd_t(\bv)$ in the form of equation \eqref{eq:LinearY}.
Now, denote 
\[
Z^{x,[\theta]}_t (\bv,v'):= \begin{pmatrix}
\partial_x Y^{x,[\theta]}_t(\bv)  \\
 \partial_{\mu} Y^{x,[\theta]}_t(\bv,v')
 \\  \partial_{v_j}Y^{x,[\theta]}_t(\bv)
\end{pmatrix}.
\]
We will write this in the form of equation \eqref{eq:LinearY} with coefficients $b_1,b_2,b_3$. Using Proposition \ref{prop:linearderivs}, we identify these coefficients as
\begin{align*}
& b_1(s,x,[\theta]) =\partial V_i(\Xd_s,[\Xt_s])  \: \text{Id}_{3} \\
& b_2(s,x,[\theta],\bv) = \begin{pmatrix}
\partial_x a^{i}_1 \: Y^{x,[\theta]}_s(\bv) + \partial_x a^{i}_2 + \widetilde{\EE} \left[ \partial_x a_3^{i}|_{v=\t \theta} \:  \widetilde{Y}^{\t \theta,[\theta]}_s(\bv) + \sum_{r=1}^{\# \bv} \partial_x a_3^{i}|_{v=v_r} \:  \widetilde{Y}^{v_r,[\theta]}_s(\bv)  \right]
\\ \partial_{\mu} a^{i}_1 \: Y^{x,[\theta]}_s(\bv) + \partial_{\mu} a^{i}_2 + \widetilde{\EE} \left[ \partial_v a_3|_{v=v'} \:  Y^{v',[\theta]}_s(\bv) + \partial_{\mu} a_3|_{v=\t \theta} \: Y^{\t \theta,[\theta]}_s(\bv) \right]
\\   \partial_{v_j} a^{i}_2 + \widetilde{\EE} \left[ \partial_{v_j} a_2^{i} \: \widetilde{Y}^{v_j,[\theta]}_s(\bv) \right]
\end{pmatrix}\\
& b_3(s,x,[\theta],v) =\begin{pmatrix}0 & 0 & 0
\\ a_3^{i}(s,x,[\theta],v) \mathbf{1}_{v=v'} & a_3^i(s,x,[\theta],v) & 0
\\a_3^{i}(s,x,[\theta],v) \mathbf{1}_{v=v_j} & 0 & a_3^{i}(s,x,[\theta],v)      \end{pmatrix}
. \end{align*} 
Now, to obtain a bound on the $\cS^p_T$-norm of $Z^{x,[\theta]} (\bv,v')$ one just has to check that the coefficients $b_1, b_2, b_3$ satisfy the assumptions of Lemma \ref{lem:linearestimate}, which is straightforward. 

\noindent (S2): This is the same as the case $I=1$.
\end{proof}

The functions belonging to the set $\KK^q_r(E,M)$ satisfy the following properties, which we make use of when developing integration by parts formulas in Section \ref{sec:IBPdecoupled}.

\begin{lemma}[Properties of local Kusuoka-Stroock processes]\label{KSP}
        %Let $H_d:= L^2 \left([0,T];\RR^d \right)$. 
        The following hold 
        \begin{enumerate}
                \item Suppose $\Psi \in \KK^q_{r}(\RR,M)$ and $ \Psi$ is $\mathbb{F}$-adapted. For
                $i=1,\hdots,d$, define
                \[
                g_i(t,x,\mu ) :=\int_0^t \Psi(s,x, \mu) \, dB^i_s  \quad \textit{and}
                \quad g_0(t,x, \mu) := \int_0^t \Psi(s,x, \mu)\, ds .
                \]
                Then, for $i=1, \ldots, d$, $g_i \in \KK^q_{r+1}(\RR,M)$ and $g_0 \in \KK^q_{r+2}(\RR,M)$.
                \item If $\Psi_i \in \KK^{q_i}_{r_i}(E,M_i)$ for $i=1, \hdots, n$, then
                \[
                \prod_{i=1}^n \Psi_i \in \KK^{q_1+ \cdots + q_n}_{r_1+ \hdots + r_n}(E, \min_{i}M_i) \quad
                \textit{and} \quad \sum_{i=1}^n \Psi_i \in
                \KK^{\max_i q_i}_{\min_i r_i}(E,\min_i M_i).
                \]
                \item If $\Psi \in \KK^q_r(H_d,M)$, then $g(t,x,\mu): = \textstyle \int_0^{t} \Psi(t,x,\mu)(r) \, dr \in \KK^q_r(\RR^d,M)$. Conversely, if $ \t \Psi \in \KK^q_r(\RR^d,M)$, then $ \t g(t,x,\mu): = \t \Psi(\cdot,x,\mu) \mathbf{1}_{[0,t]}(\cdot) \in \KK_{r+1}^q(H_d,M)$.
                \item If $\Psi \in \KK^q_r(\RR,M)$, then $\D \Psi \in \KK^q_r(H_d,M-1)$.
                \item If $\Psi \in \KK^{q_1}_{r_1}(\RR,M_1)$ and $u \in \KK^{q_2}_{r_2}(H_d,M_2)$ then, $\langle \D \Psi,u \rangle_{H_d} \in \KK^{q_1+q_2}_{r_1+r_2}(\RR, (M_1-1) \wedge M_2)$.
                \item If $\Psi \in \KK^{q_1}_{r_1}(\RR^N,M_1)$ and $u \in \KK^{q_2}_{r_2}(H_{d \times N},M_2)$ is $\mathbb{F}$-adapted then, \newline $\delta \left(  u \Psi \right) \in \KK^{q_1+q_2}_{r_1+r_2}(\RR, (M_1-1) \wedge M_2)$.
                \item If $\, \Psi \in \KK^q_{r}(\RR,M)$ then, $\partial_x \Psi \in \KK^q_r(\RR, M-1)$ and $(x,v,\mu) \mapsto\partial_{\mu} \Psi(x,\mu,v)$ is a Kusuoka-Stroock process on $\RR^{2N} \times \p_2(\RR^N)$ in the class $\KK^q_r(\RR, M-1)$.
        \end{enumerate}
        
\end{lemma}

\begin{proof}
        These results are straightforward generalisations of results in \cite{KusStrIII} and \cite{crisan2010cubature}.
\end{proof}

Now, we show that certain processes, which will make up the Malliavin weights in our integration by parts formulas, belong to specific Kusuoka-Stroock classes. The arguments make extensive use of the properties of generic Kusuoka-Stroock processes on $\RR^N \times \p_2(\RR^N)$ in Lemma \ref{KSP}.     

        \begin{proposition}
                \label{prop:KSprocesses}
                If $\Vs \in \C^{k,k}_{b,Lip}(\RR^N \times \p_2(\RR^N);\RR^N)$ and (UE) holds, then the following are true:
                \begin{enumerate}
                        \item Let $|\alpha|=1$, and $ \,\Phi_1 = \sigma^{\top} (\sigma \sigma^{\top})^{-1} (X^{x,[\theta]}_{.}, [X_.^{\theta}]) \partial^{\alpha}_x X^{x,[\theta]}_. \mathbf{1}_{[0,t]}(\cdot)$. Then, $\Phi_1   \in \KK^2_1(H_{d},k-1)$ and if $\Vs$ are uniformly bounded then $\Phi_1 \in \KK^0_1(H_{d},k-1)$.
                        \item For all $i,j \in \{1, \ldots, N\}$,
                        $(\partial_x X^{x,[\theta]}_t)^{-1}_{i,j} \in \KK^1_0(\RR,k-2)$ and if $\Vs$ are uniformly bounded then $(\partial_x X^{x,[\theta]}_t)^{-1}_{i,j} \in \KK^0_0(\RR,k-2)$.
                        \item $(\partial_x X^{x,[\theta]}_t)^{-1}  \partial_{\mu} \Xd_t  \in \KK_{0}^{2}(\RR^{N \times N},k-2)$ and if $\Vs$ uniformly bounded then $(\partial_x X^{x,[\theta]}_t)^{-1}  \partial_{\mu} \Xd_t  \in \KK_{0}^{0}(\RR^N,k-2)$.
                \end{enumerate}
        \end{proposition}
        
        \begin{proof}
                
                1. First, note that from Assumption \ref{ass:UE}, it follows that the matrix $(\sigma \sigma^{\top})^{-1}(x, \mu)$ has a an operator norm bounded uniformly in $(x, \mu)$. Therefore $\sigma^{\top} (\sigma \sigma^{\top})^{-1} (\cdot, \cdot) $ has linear growth. Also, its elements are $k$-times differentiable in $(x, [\theta])$, so $\sigma^{\top} (\sigma \sigma^{\top})^{-1} (X^{x,[\theta]}_{t}, [X_t^{\theta}]) \in \KK_0^1(\RR^{d \times N}, k)$. When $|\alpha|=1$, $\partial^{\alpha}_x X^{x,\mu}_t \in \KK_0^1(\RR^N, k-1)$ by part 7 of Lemma \ref{KSP}, so the product
                $\sigma^{\top} (\sigma \sigma^{\top})^{-1} (X^{x,[\theta]}_{t}, [X_t^{\theta}]) \partial^{\alpha}_x X^{x,\mu}_t \in \KK_1^2(\RR^d, k-1)$. 
                Hence, by Lemma \ref{KSP} part 3., $\Phi_1  \in \KK^2_1(H_{d},k)$. 
                
                2.  $(\partial_x X^{x ,[\theta]}_t)^{-1}$ satisfies the following linear equation
                \begin{align}
                        \label{eq:jacinv}
                        (\partial_x \Xd_t)^{-1} &= \text{Id}_N  - \sum_{i=1}^d \int_0^t (\partial_x \Xd_s)^{-1} \, \partial V_i \left(\Xd_s, [\Xt_s ] \right)   \, dB^i_s \\
                        \nonumber       & \phantom{XXX} - \int_0^t (\partial_x \Xd_s)^{-1} \, \partial \bar{V}_0 \left(\Xd_s, [\Xt_s ] \right)   \, ds,
                \end{align}
                where $\bar{V}_0 =V_0 -  \textstyle \frac{1}{2} \sum_{j=1}^d \partial V_j V_j$. This can be seen by applying It\^{o}'s formula to the product $(\partial_x X^{x ,[\theta]}_t)^{-1} \partial_x \Xd_t$. 
                The proof of Theorem \ref{th:XisKSP} works just as well for this equation. The only thing to note is that the above equation contains second derivatives of the vector fields. This leads to the conclusion $(\partial_x \Xd_t)^{-1} \in \KK^1_0(\RR^{N \times N},k-2)$.
                
                3. To prove the claim, it is enough to note $(\partial_x X^{x,\mu}_t)^{-1} \in \KK_0^1(\RR^{N \times N}, k-2)$ from part 2 of this lemma and $\partial_{\mu} \Xd_t \in \KK_{0}^1(\RR^{N \times N },k-1)$, which comes from Lemma \ref{KSP} part 7.
                
        \end{proof}
        
        We can now prove Proposition \ref{prop:defwieghts}.
        
        \begin{proof}[Proof of Proposition \ref{prop:defwieghts}]
                \begin{enumerate}
                        \item[$I^1_{\alpha}$:]
                        First, fix $|\alpha|=1$.
                        We want to apply Lemma \ref{KSP} part 6. with $f=\Psi$ and $u = \left(\sigma^{\top} (\sigma \sigma^{\top})^{-1} (X^{x,\mu}_{.}, [X_.^{\theta}]) \partial_x X^{x,\mu}_.\right)_{\alpha} \mathbf{1}_{[0,t]} $. We recall Proposition \ref{prop:KSprocesses} part 1. to see that $u  \in \KK^2_1(H_{d},k-1)$ or $\KK^0_1(H_{d},k-1)$ if $V_i$ is uniformly bounded, which proves that
                        \[
                        \delta \left(r \mapsto \Psi(t,x, [\theta]) \,   \left( \sigma^{\top} (\sigma \sigma^{\top})^{-1} (X^{x,\mu}_r, [X_r^{\theta}]) \partial_x X^{x,\mu}_r \right)_{\alpha} \right) \in \KK^{q+2}_{r+1}(\RR,(k \wedge n)-1)
                        \]
                        (or $\KK^q_{r+1}(\RR,k-1)$ if $V_i$ is bounded) and hence, dividing by $\sqrt{t}$, we get that $I^1_{\alpha}(\Psi) \in \KK^{q+2}_r(\RR,(k \wedge n)-1)$ for $|\alpha|=1$. For $|\alpha|>1$, we iterate this argument and get $I^1_{\alpha}(\Psi) \in \KK^{q+2|\alpha|}_r(\RR,(k \wedge n)-|\alpha|)$.
                        
                        \item[$I^2_{\alpha}$:]
                        We recall from Proposition \ref{prop:KSprocesses} part 2. that:
                        For all $i,j \in \{1, \ldots, N\}$,
                        $(\partial_x X^{x,\mu}_t)^{-1}_{i,j} \in \KK^1_0(\RR,k-2)$ and if $V_i$ are uniformly bounded, $(\partial_x X^{x,\mu}_t)^{-1}_{i,j} \in \KK^0_0(\RR,k-2)$.
                        So, the product $(\partial_x X^{x,\mu}_t)^{-1}_{j,i}  \Psi(t,x, [\theta]) \in \KK^{q+1}_r(\RR,n \wedge (k-2))$ and hence the sum  $\textstyle\sum_{j=1}^N (\partial_x X^{x,\mu}_t)^{-1}_{j,i}  \Psi(t,x, [\theta]) \in \KK^{q+1}_r(\RR,n \wedge (k-2))$.  When the vector fields are uniformly bounded,
                        \[ \sum_{j=1}^N (\partial_x X^{x,\mu}_t)^{-1}_{j,i}  \Psi(t,x, [\theta]) \in \KK^q_r(\RR,n \wedge (k-2)).\]
                        Hence, by applying $I^1$ to these terms and using the first result of this proposition, we get that $I^2_{(i)}(\Psi) \in \KK^{q+3}_r(\RR,[n \wedge (k-2)]-1)$. For $|\alpha|>1$, we iterate this argument and get $I^2_{\alpha}(\Psi) \in \KK^{q+3|\alpha|}_r(\RR,[n \wedge (k-2)]-|\alpha|)$.
                        
                        \item[$I^3_{\alpha}$:]
                        Note that $\sqrt{t} \partial^i\Psi(t,x, [\theta]) \in \KK^q_{r+1}(\RR,n-1)$ so that $I^1_{(i)}(\Psi) + \sqrt{t} \partial^i\Psi \in \KK^{q+2}_r(\RR, (n \wedge k)-1) $ . For $|\alpha|>1$, we iterate this argument and get $I^3_{\alpha}(\Psi) \in \KK^{q+2|\alpha|}_r(\RR,(k \wedge n)-|\alpha|)$.
                        
                        \item[$\I^1_{\alpha}$:]
                        We recall from Proposition \ref{prop:KSprocesses} that $(\partial_x X^{x,\mu}_t)^{-1}  \partial_{\mu} \Xd_t  \in \KK_{0}^{2}(\RR^{N \times N},k-2)$, so $(\partial_x X^{x,\mu}_t)^{-1} \partial_{\mu} \Xd_t \Psi(t,x, [\theta]) \in \KK_{r}^{q+2}(\RR^{N \times N},n \wedge (k-2))$, then we apply Lemma \ref{KSP} part 6. with $u = \left(\sigma^{\top} (\sigma \sigma^{\top})^{-1} (X^{x,\mu}_{.}, [X_.^{\theta}]) \partial_x X^{x,\mu}_.\right)_{\alpha} \mathbf{1}_{[0,t]} $ which is in $\KK^2_1(H_{d},k-1)$ as before, and $f:= (\partial_x X^{x,\mu}_t)^{-1}  \partial_{\mu} \Xd_t   \Psi(t,x, [\theta]) \in \KK_{r}^{q+2}(\RR^{N \times N},n \wedge (k-2))$. So $\delta(uf) \in \KK^{q+4}_{r+1}(\RR;[n \wedge (k-2)-1])$. Hence, $ \I^1_{\alpha}(\Psi) \in \KK^{q+4}_{r}(\RR;[n \wedge (k-2)-1])$. 
                        For $|\alpha|>1$, we iterate this argument and get $\I^1_{\alpha}(\Psi) \in \KK^{q+4|\alpha|}_r(\RR,[n \wedge (k-2) ]-|\alpha|)$.
                        
                        \item[$\I^3_{\alpha}$:]
                        Note that $\sqrt{t} \partial_{\mu}\Psi(v) \in \KK^q_{r+1}(\RR^{N \times N},n-1)$ so that \newline $\I^1_{\gamma_1}(\Psi)(v) +( \partial_{\mu}\Psi(v))_{\beta_1} \in \KK^{q+4|\alpha|}_r(\RR,[n \wedge (k-2) ]-|\alpha|)$ .
                \end{enumerate}
        \end{proof}

\bibliography{mybiblio}

\end{document}